\documentclass[10pt,a4paper]{amsart}

\usepackage[austrian,english]{babel}
\usepackage[latin1]{inputenc}
\usepackage[a4paper]{geometry}

\usepackage{color}
\usepackage{hyperref}
\usepackage{enumerate}
\usepackage[font=footnotesize]{caption}
\usepackage{xparse}
\usepackage{tikz,tikz-3dplot}
\tikzset{every picture/.append style={scale=0.45,
mydot/.style={circle,fill=blue,inner sep=0pt,minimum size=1mm},
mybead/.style={shape=circle,draw,inner sep=0pt,minimum size=2.5ex}}}
\tdplotsetmaincoords{60}{120}

\usepackage{amsmath}
\usepackage{amssymb}
\usepackage{mathtools}
\usepackage[scr=boondox]{mathalfa}
\usepackage{faktor}

\newcommand{\mytilde}{\raise.17ex\hbox{$\scriptstyle \mathtt{\sim}$}}

\newcommand{\arxiv}[1]{\href{https://arxiv.org/abs/#1}{\texttt{arXiv:#1}}}
\newcommand{\doi}[1]{\href{http://doi.org/#1}{\texttt{doi.org/#1}}}

\newcommand{\jcta}{J.~Combin.\ Theory Ser.~A}

\newcommand{\slc}{S{\'e}m.~Lothar.\ Combin.}

\newcommand{\myi}[1]{\item\label{Item:#1}}
\NewDocumentCommand\myfig{O{t} m m m}{
\begin{figure}[#1]
\begin{center}
#4
\caption{#2}
\label{Figure:#3}
\end{center}
\end{figure}
}
\NewDocumentCommand\myfigh{m m m}{
\smallskip
\noindent
\begin{minipage}{\linewidth}
\captionsetup{type=figure}
\begin{center}
#3
\captionof{figure}{#1}
\label{Figure:#2}
\end{center}
\smallskip
\end{minipage}
}

\NewDocumentEnvironment{mymathenvironment}{m m O{}}{
	\begin{#1}[#3]\label{#1:#2}
}{
	\end{#1}
}
\newenvironment{eq*}{
	\begin{equation*}
}{
	\end{equation*}
}
\newenvironment{eq}[1]{
	\begin{equation}\label{Equation:#1}
}{
	\end{equation}
}

\NewDocumentEnvironment{thm}{m}{\begin{mymathenvironment}{Theorem}{#1}}{\end{mymathenvironment}}
\NewDocumentEnvironment{cor}{m}{\begin{mymathenvironment}{Corollary}{#1}}{\end{mymathenvironment}}
\NewDocumentEnvironment{conj}{m}{\begin{mymathenvironment}{Conjecture}{#1}}{\end{mymathenvironment}}
\NewDocumentEnvironment{mydef}{m}{\begin{mymathenvironment}{Definition}{#1}}{\end{mymathenvironment}}
\NewDocumentEnvironment{ex}{m}{\begin{mymathenvironment}{Example}{#1}}{\end{mymathenvironment}}
\NewDocumentEnvironment{lem}{m}{\begin{mymathenvironment}{Lemma}{#1}}{\end{mymathenvironment}}
\NewDocumentEnvironment{prob}{m}{\begin{mymathenvironment}{Problem}{#1}}{\end{mymathenvironment}}
\NewDocumentEnvironment{prop}{m}{\begin{mymathenvironment}{Proposition}{#1}}{\end{mymathenvironment}}
\NewDocumentEnvironment{rem}{m}{\begin{mymathenvironment}{Remark}{#1}}{\end{mymathenvironment}}

\newcommand{\refx}[2]{#2~\ref{#2:#1}}

\newcommand{\refc}[1]{\refx{#1}{Corollary}}

\newcommand{\reff}[1]{\refx{#1}{Figure}}
\newcommand{\refi}[1]{\eqref{Item:#1}}
\newcommand{\refl}[1]{\refx{#1}{Lemma}}

\newcommand{\refp}[1]{\refx{#1}{Proposition}}
\newcommand{\refq}[1]{\eqref{Equation:#1}}

\newcommand{\refs}[1]{\refx{#1}{Section}}
\newcommand{\reft}[1]{\refx{#1}{Theorem}}

\newcommand{\A}{\mathcal{A}}
\newcommand{\abs}[1]{\left| #1 \right|}
\newcommand{\B}{\mathcal{B}}
\newcommand{\bop}[1]{\boldsymbol{\operatorname{#1}}}

\newcommand{\cand}{\mathcal{C}}

\newcommand{\e}{\bop{e}}

\newcommand{\F}{\mathcal{F}}

\newcommand{\HG}{\operatorname{HG}}
\newcommand{\I}{\mathcal{I}}

\newcommand{\n}{\bop{n}}
\newcommand{\N}{\mathbb{N}}
\renewcommand{\O}{\mathcal{O}}

\newcommand{\rpp}{\operatorname{RPP}}
\newcommand{\RSK}{\operatorname{RSK}}
\newcommand{\s}{\bop{s}}
\renewcommand{\S}{\mathfrak{S}}

\newcommand{\tab}{{T}}
\newcommand{\tleq}{\trianglelefteq}
\newcommand{\tr}{\operatorname{tr}}
\newcommand{\w}{\bop{w}}
\newcommand{\Z}{\mathbb{Z}}

\numberwithin{equation}{section}

\title{Inserting rim-hooks into reverse plane partitions}
\author[Robin Sulzgruber]{Robin Sulzgruber}
\thanks{Research supported by the Austrian Science Fund (FWF), grant S50-N15 in the framework of the Special Research Program ``Algorithmic and Enumerative Combinatorics'' (SFB F50). The author was partially supported by the Department of Mathematics at KTH}
\date{February 2018}
\address{Institutionen f{\"o}r matematik, KTH, Lindstedtsv{\"a}gen 25, 100~44 Stockholm, Sweden}
\subjclass[2010]{05A19}
\keywords{bijection; reverse plane partitions; RSK}

\begin{document}
\thispagestyle{empty}
\begin{abstract}
A new algorithm for inserting rim-hooks into reverse plane partitions is presented.
The insertion is used to define a bijection between reverse plane partitions of a fixed shape and multi-sets of rim-hooks.
In turn this yields a bijective proof of the fact that the generating function for reverse plane partitions of a fixed shape, which was first obtained by R.~Stanley, factors into a product featuring the hook-lengths of this shape.
Our bijection turns out to be equivalent to a map defined by I.~Pak by different means, and can be related to the Hillman--Grassl correspondence and the Robinson--Schensted--Knuth correspondence.
\end{abstract}
\maketitle
\pagenumbering{arabic}
\pagestyle{headings}


Reverse plane partitions, like their relatives the plane partitions and semi-standard Young tableaux, appear naturally in the context of symmetric functions and the representation theory of the symmetric group.
The generating function for reverse plane partitions was first obtained by R.~Stanley.

\begin{thm}{Stanley}\textnormal{\cite[Prop.~18.3]{Stanley1971}}
The generating function for reverse plane partitions of shape $\lambda$ is given by
\begin{eq*}
\sum_{\pi}q^{\abs{\pi}}
=\prod_{u\in\lambda}\frac{1}{1-q^{h(u)}},
\end{eq*}
where $h(u)$ denotes the hook-length of the cell $u\in\lambda$.
\end{thm}

A bijective proof of this result was found by A.~Hillman and R.~Grassl~\cite{HilGra1976}.
Their algorithm has since been related to the Robinson--Schensted--Knuth correspondence and M.~Sch{\"u}tzenberger's jeu de taquin~\cite{Gansner1981,Kadell1997}, appears in standard works~\cite{StanleyEC2,Sagan2001}, and has become an integral part of enumerative combinatorics.
In particular, E.~Gansner used the Hillman--Grassl correspondence to obtain a refined generating function.

\begin{thm}{Gansner}\textnormal{\cite[Thm.~3.2]{Gansner1981}}
The trace generating function for reverse plane partitions of shape $\lambda$ is given by
\begin{eq*}
\sum_{\pi}\prod_{k\in\Z}q_k^{\tr_k(\pi)}
=\prod_{u\in\lambda}\frac{1}{1-q^{H(u)}}\\
\quad\text{where}\quad
\tr_k(\pi)=\sum_{\substack{(i,j)\in\lambda\\j-i=k}}\pi(i,j)
\quad\text{and}\quad
q^{H(i,j)}=\prod_{k=j-\lambda_j'}^{\lambda_i-i}q_k
\,.
\end{eq*}
\end{thm}

Note that \reft{Stanley} is obtained from \reft{Gansner} by setting all variables equal to $q$.

More recently A.~Morales, I.~Pak and G.~Panova~\cite{MPP:skew_shapes_I} used the Hillman--Grassl correspondence to obtain a (trace) generating function for reverse plane partitions of skew shape.

\smallskip
In this paper we present an alternative bijective proof of Theorems~\ref{Theorem:Stanley} and~\ref{Theorem:Gansner}.
In \refs{bricks} we propose an algorithm for inserting rim-hooks into reverse plane partitions.
This algorithm is best perceived as a simple set of rules for building reverse plane partitions, viewed as arrangements of stacks of cubes, using bricks in the shape of rim-hooks.
The insertion procedure is formalised in Sections~\ref{Section:rimhooks} and~\ref{Section:insertion}.

\smallskip
In \refs{factors} we address the question whether the insertion of rim-hooks can be reversed.
The full answer to this question is given in \refs{bijection}.
Given a rim-hook $h$ and a reverse plane partition $\pi$, the insertion of $h$ into $\pi$ may either fail or succeed.
In the latter case a new reverse plane partition is obtained, which is denoted by $h*\pi$.
Given a multi-set of rim-hooks, it might be possible to use those rim-hooks to build multiple reverse plane partitions.
That is, there might be different ways to successively insert the rim-hooks into the zero reverse plane partition with distinct outcome.
Conversely, given a reverse plane partition there might be distinct multi-sets of rim-hooks that can be used to build it.
However, it turns out that each reverse plane partition can be built in a unique way if we demand that rim-hooks are inserted in lexicographic order.
This is the main result of this paper.

\begin{thm}{main}
Let $\lambda$ be a partition.
Then for each reverse plane partition $\pi$ of shape $\lambda$ there exists a unique ordered sequence of rim-hooks $h_1\leq\dots\leq h_s$ of $\lambda$ such that $h_i$ inserts into $h_{i+1}*\cdots*h_s*0$ for all $i\in[s]$, and
\begin{eq*}
\pi=h_1*\cdots*h_s*0
\,,
\end{eq*}%
where $0$ denotes the reverse plane partition of shape $\lambda$ with all entries equal to zero.
The correspondence $\pi\mapsto(h_1,\dots,h_s)$ gives rise to a bijection $\Phi$ between reverse plane partitions of shape $\lambda$ and multi-sets of rim-hooks of $\lambda$.
\end{thm}

\smallskip
A proof of \reft{main} is given in \refs{bijection}.
The bijection $\Phi$ is then used to prove Theorems~\ref{Theorem:Stanley} and~\ref{Theorem:Gansner}.

\smallskip
The inverse of our bijection extracts a rim-hook by decreasing a reverse plane partition along a path and is especially reminiscent of the method of A.~Hillman and R.~Grassl.
Thus, once it can be said with some confidence that the two maps do not coincide, the natural question arises whether one can be easily obtained from the other.
More generally, it is natural to ask if some of the deeper connections between the Hillman--Grassl correspondence and RSK mentioned above are also reflected in our map.
We elaborate on connections between rim-hook insertion and known bijections in \refs{connections}.

The precise relation between our map and the Hillman--Grassl correspondence remains unclear except in special cases.
Some nice observations can be made for permutation matrices and standard Young tableaux.
See Theorems~\ref{Theorem:syt} and~\ref{Theorem:rsk}.
Instead it turns out that $\Phi$ is equivalent to a different bijection that already appears in the literature, although it is much less well-known than the bijection of A.~Hillman and R.~Grassl.
I.~Pak~\cite{Pak:hook_length_formula_geometric} encountered the same map in his geometric proof of the hook-length formula based on polytopes.
His definition of the map as a concatenation of piecewise linear bijections is based on ideas from a paper of I.~Pak and A.~Postnikov~\cite{PP:oscillating_tableaux}.
I.~Pak also mentions a previous paper by A.~Berenstein and A.~Kirillov~\cite{BerKir} in which similar ideas appear in connection with jeu de taquin.
Another account on this map and its relation to RSK is due to S.~Hopkins~\cite{Hopkins:rsk} and emerged from a combinatorics seminar with A.~Postnikov held at MIT.
Among other things S.~Hopkins also observed that the map can be used to obtain \reft{Gansner}, which is the same as \cite[Cor.~16]{Hopkins:rsk}.
Our description differs significantly from the approach of I.~Pak and S.~Hopkins.
A proof that the resulting bijections are equivalent is given in \reft{pak}.

Rim-hook insertion has also served as the motivation for some recent work of A.~Garver and R.~Patrias~\cite{GarverPatrias} who found a common framework for obtaining both our bijection and the Hillman--Grassl correspondence from RSK.

\smallskip
An extended abstract of this article appeared in the proceedings of FPSAC 2017 in London~\cite{Sul:building_rpps}.

\section{Acknowledgements}

The problem that originally motivated the finding of the present insertion algorithm for rim-hooks was communicated to me by Laura Colmenarejo.
Moreover I have benefited from discussions with Alexander Garver, Sam Hopkins, Christian Krattenthaler, Rebecca Patrias and Hugh Thomas.
I am especially grateful to Sam Hopkins and Hugh Thomas for pointing out the connections between my work and previously known results.
I also thank Darij Grinberg for his thorough reading of the article and his useful comments.

\section{Building with bricks}\label{Section:bricks}

The aim of this section is to give an informal description of a simple and intuitive algorithm for building reverse plane partitions using rim-hook-shaped bricks.
A precise formal description and definitions are found in the subsequent sections.

\smallskip
A \emph{partition} is a top and left justified array of cells.
A \emph{reverse plane partition} is a top and left justified array of non-negative integers such that rows and columns are weakly increasing.
Alternatively, a reverse plane partition can be visualised as a three-dimensional object, by placing $k$ unit cubes on top of the number $k$.
\reff{partition} shows examples of these three notions.
Note that we have rotated the reverse plane partition to obtain its representation by cubes such that all stacks of cubes are visible. 

\myfigh
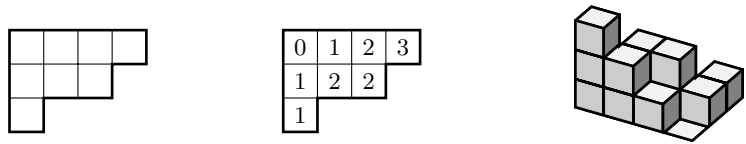
{The partition $\lambda=(4,3,1)$, a reverse plane partition $\pi$ of shape $\lambda$, and the representation of $\pi$ as stacks of cubes.}
{partition}
{
\begin{tikzpicture}
\begin{scope}
\draw(0,1)--(1,1)(0,2)--(3,2)(1,1)--(1,3)(2,1)--(2,3)(3,2)--(3,3)
;
\draw[line width=1pt](0,0)--(1,0)--(1,1)--(3,1)--(3,2)--(4,2)--(4,3)--(0,3)--cycle
;
\end{scope}
\begin{scope}[xshift=8cm]
\draw(0,1)--(1,1)(0,2)--(3,2)(1,1)--(1,3)(2,1)--(2,3)(3,2)--(3,3)
;
\draw[line width=1pt](0,0)--(1,0)--(1,1)--(3,1)--(3,2)--(4,2)--(4,3)--(0,3)--cycle
;
\draw[xshift=5mm,yshift=5mm]
(0,2)node{\small{$0$}}
(1,2)node{\small{$1$}}
(2,2)node{\small{$2$}}
(3,2)node{\small{$3$}}
(0,1)node{\small{$1$}}
(1,1)node{\small{$2$}}
(2,1)node{\small{$2$}}
(0,0)node{\small{$1$}};
\end{scope}
\begin{scope}[xshift=18cm,yshift=2cm,tdplot_main_coords]
\clip(3,0,0)--(3,4,0)--(0,4,0)--(0,4,1)--(0,3,1)--(1,3,1)--(1,3,2)--(1,2,2)--(1,1,2)--(2,1,2)--(2,1,3)--(2,0,3)--(3,0,3)--(3,0,0)--cycle
;
\draw[line width=1pt,fill=black!5]
(2,3,0)--(3,3,0)--(3,4,0)--(2,4,0)--cycle
(2,2,1)--(3,2,1)--(3,3,1)--(2,3,1)--cycle
(2,1,2)--(3,1,2)--(3,2,2)--(2,2,2)--cycle
(2,0,3)--(3,0,3)--(3,1,3)--(2,1,3)--cycle
(1,3,1)--(2,3,1)--(2,4,1)--(1,4,1)--cycle
(1,2,2)--(2,2,2)--(2,3,2)--(1,3,2)--cycle
(1,1,2)--(2,1,2)--(2,2,2)--(1,2,2)--cycle
(0,3,1)--(1,3,1)--(1,4,1)--(0,4,1)--cycle
;
\draw[line width=1pt,fill=black!20]
(3,2,0)--(3,3,0)--(3,3,1)--(3,2,1)--cycle
(3,1,1)--(3,2,1)--(3,2,2)--(3,1,2)--cycle
(3,1,0)--(3,2,0)--(3,2,1)--(3,1,1)--cycle
(3,0,2)--(3,1,2)--(3,1,3)--(3,0,3)--cycle
(3,0,1)--(3,1,1)--(3,1,2)--(3,0,2)--cycle
(3,0,0)--(3,1,0)--(3,1,1)--(3,0,1)--cycle
(2,3,0)--(2,4,0)--(2,4,1)--(2,3,1)--cycle
(2,2,1)--(2,3,1)--(2,3,2)--(2,2,2)--cycle
;
\draw[line width=1pt,fill=black!50]
(2,3,0)--(2,3,1)--(3,3,1)--(3,3,0)--cycle
(2,2,1)--(2,2,2)--(3,2,2)--(3,2,1)--cycle
(2,1,2)--(2,1,3)--(3,1,3)--(3,1,2)--cycle
(1,4,0)--(1,4,1)--(2,4,1)--(2,4,0)--cycle
(1,3,1)--(1,3,2)--(2,3,2)--(2,3,1)--cycle
(0,4,0)--(0,4,1)--(1,4,1)--(1,4,0)--cycle
;
\draw[line width=2pt](3,0,0)--(3,4,0)--(0,4,0)--(0,4,1)--(0,3,1)--(1,3,1)--(1,3,2)--(1,2,2)--(1,1,2)--(2,1,2)--(2,1,3)--(2,0,3)--(3,0,3)--(3,0,0)--cycle
;
\end{scope}
\end{tikzpicture}
} 

A \emph{rim-hook} (or \emph{border strip} or \emph{ribbon}) of a partition is a connected subset of cells that does not contain a two-by-two square, such that removing the rim-hook yields again a partition.
\reff{rimhooks} shows the rim-hooks of the partition $\lambda=(4,3,1)$.

\myfigh
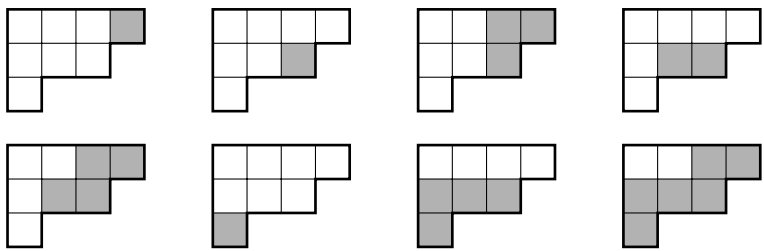
{All eight rim-hooks of the partition $\lambda=(4,3,1)$.}
{rimhooks}
{
\begin{tikzpicture}
\begin{scope}[xshift=0cm,yshift=4cm]
\fill[black!30](3,2)--(4,2)--(4,3)--(3,3)--cycle
;
\draw(0,1)--(1,1)(0,2)--(3,2)(1,1)--(1,3)(2,1)--(2,3)(3,2)--(3,3)
;
\draw[line width=1pt](0,0)--(1,0)--(1,1)--(3,1)--(3,2)--(4,2)--(4,3)--(0,3)--cycle
;
\end{scope}
\begin{scope}[xshift=6cm,yshift=4cm]
\fill[black!30](2,1)--(3,1)--(3,2)--(2,2)--cycle
;
\draw(0,1)--(1,1)(0,2)--(3,2)(1,1)--(1,3)(2,1)--(2,3)(3,2)--(3,3)
;
\draw[line width=1pt](0,0)--(1,0)--(1,1)--(3,1)--(3,2)--(4,2)--(4,3)--(0,3)--cycle
;
\end{scope}
\begin{scope}[xshift=12cm,yshift=4cm]
\fill[black!30](2,1)--(3,1)--(3,2)--(4,2)--(4,3)--(2,3)--cycle
;
\draw(0,1)--(1,1)(0,2)--(3,2)(1,1)--(1,3)(2,1)--(2,3)(3,2)--(3,3)
;
\draw[line width=1pt](0,0)--(1,0)--(1,1)--(3,1)--(3,2)--(4,2)--(4,3)--(0,3)--cycle
;
\end{scope}
\begin{scope}[xshift=18cm,yshift=4cm]
\fill[black!30](1,1)--(3,1)--(3,2)--(1,2)--cycle
;
\draw(0,1)--(1,1)(0,2)--(3,2)(1,1)--(1,3)(2,1)--(2,3)(3,2)--(3,3)
;
\draw[line width=1pt](0,0)--(1,0)--(1,1)--(3,1)--(3,2)--(4,2)--(4,3)--(0,3)--cycle
;
\end{scope}
\begin{scope}[xshift=0cm]
\fill[black!30](1,1)--(3,1)--(3,2)--(4,2)--(4,3)--(2,3)--(2,2)--(1,2)--cycle
;
\draw(0,1)--(1,1)(0,2)--(3,2)(1,1)--(1,3)(2,1)--(2,3)(3,2)--(3,3)
;
\draw[line width=1pt](0,0)--(1,0)--(1,1)--(3,1)--(3,2)--(4,2)--(4,3)--(0,3)--cycle
;
\end{scope}
\begin{scope}[xshift=6cm]
\fill[black!30](0,0)--(1,0)--(1,1)--(0,1)--cycle
;
\draw(0,1)--(1,1)(0,2)--(3,2)(1,1)--(1,3)(2,1)--(2,3)(3,2)--(3,3)
;
\draw[line width=1pt](0,0)--(1,0)--(1,1)--(3,1)--(3,2)--(4,2)--(4,3)--(0,3)--cycle
;
\end{scope}
\begin{scope}[xshift=12cm]
\fill[black!30](0,0)--(1,0)--(1,1)--(3,1)--(3,2)--(0,2)--cycle
;
\draw(0,1)--(1,1)(0,2)--(3,2)(1,1)--(1,3)(2,1)--(2,3)(3,2)--(3,3)
;
\draw[line width=1pt](0,0)--(1,0)--(1,1)--(3,1)--(3,2)--(4,2)--(4,3)--(0,3)--cycle
;
\end{scope}
\begin{scope}[xshift=18cm]
\fill[black!30](0,0)--(1,0)--(1,1)--(3,1)--(3,2)--(4,2)--(4,3)--(2,3)--(2,2)--(0,2)--cycle
;
\draw(0,1)--(1,1)(0,2)--(3,2)(1,1)--(1,3)(2,1)--(2,3)(3,2)--(3,3)
;
\draw[line width=1pt](0,0)--(1,0)--(1,1)--(3,1)--(3,2)--(4,2)--(4,3)--(0,3)--cycle
;
\end{scope}
\end{tikzpicture}
}

For the purpose of this section we view rim-hooks as rigid three dimensional bricks of height one.
See \reff{rimhook_bricks}.

\myfigh
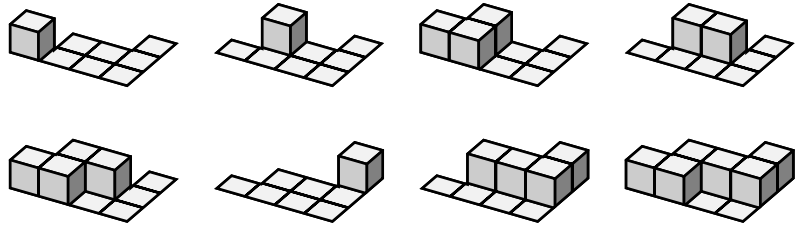
{Bricks in the shapes of rim-hooks.}
{rimhook_bricks}
{
\begin{tikzpicture}[tdplot_main_coords]
\begin{scope}[xshift=0cm,yshift=0cm]
\clip(3,0,0)--(3,4,0)--(0,4,0)--(0,3,0)--(1,3,0)--(1,2,0)--(1,1,0)--(2,1,0)--(2,1,1)--(2,0,1)--(3,0,1)--(3,0,0)--cycle;
\draw[line width=1pt,fill=black!5]
(2,3,0)--(3,3,0)--(3,4,0)--(2,4,0)--cycle
(2,2,0)--(3,2,0)--(3,3,0)--(2,3,0)--cycle
(2,1,0)--(3,1,0)--(3,2,0)--(2,2,0)--cycle
(2,0,1)--(3,0,1)--(3,1,1)--(2,1,1)--cycle
(1,3,0)--(2,3,0)--(2,4,0)--(1,4,0)--cycle
(1,2,0)--(2,2,0)--(2,3,0)--(1,3,0)--cycle
(1,1,0)--(2,1,0)--(2,2,0)--(1,2,0)--cycle
(0,3,0)--(1,3,0)--(1,4,0)--(0,4,0)--cycle
;
\draw[line width=1pt,fill=black!20]
(3,0,0)--(3,1,0)--(3,1,1)--(3,0,1)--cycle
;
\draw[line width=1pt,fill=black!50]
(2,1,0)--(2,1,1)--(3,1,1)--(3,1,0)--cycle
;
\draw[line width=2pt](3,0,0)--(3,4,0)--(0,4,0)--(0,3,0)--(1,3,0)--(1,2,0)--(1,1,0)--(2,1,0)--(2,1,1)--(2,0,1)--(3,0,1)--(3,0,0)--cycle;
\end{scope}
\begin{scope}[xshift=6cm,yshift=0cm]
\clip(3,0,0)--(3,4,0)--(0,4,0)--(0,3,0)--(1,3,0)--(1,2,0)--(1,2,1)--(1,1,1)--(2,1,1)--(2,1,0)--(2,0,0)--(3,0,0)--cycle;
\draw[line width=1pt,fill=black!5]
(2,3,0)--(3,3,0)--(3,4,0)--(2,4,0)--cycle
(2,2,0)--(3,2,0)--(3,3,0)--(2,3,0)--cycle
(2,1,0)--(3,1,0)--(3,2,0)--(2,2,0)--cycle
(2,0,0)--(3,0,0)--(3,1,0)--(2,1,0)--cycle
(1,3,0)--(2,3,0)--(2,4,0)--(1,4,0)--cycle
(1,2,0)--(2,2,0)--(2,3,0)--(1,3,0)--cycle
(1,1,1)--(2,1,1)--(2,2,1)--(1,2,1)--cycle
(0,3,0)--(1,3,0)--(1,4,0)--(0,4,0)--cycle
;
\draw[line width=1pt,fill=black!20]
(2,1,0)--(2,2,0)--(2,2,1)--(2,1,1)--cycle
;
\draw[line width=1pt,fill=black!50]
(1,2,0)--(1,2,1)--(2,2,1)--(2,2,0)--cycle
;
\draw[line width=2pt](3,0,0)--(3,4,0)--(0,4,0)--(0,3,0)--(1,3,0)--(1,2,0)--(1,2,1)--(1,1,1)--(2,1,1)--(2,1,0)--(2,0,0)--(3,0,0)--cycle;
\end{scope}
\begin{scope}[xshift=12cm,yshift=0cm]
\clip(3,0,0)--(3,4,0)--(0,4,0)--(0,3,0)--(1,3,0)--(1,2,0)--(1,2,1)--(1,1,1)--(2,1,1)--(2,0,1)--(3,0,1)--(3,0,0)--cycle;
\draw[line width=1pt,fill=black!5]
(2,3,0)--(3,3,0)--(3,4,0)--(2,4,0)--cycle
(2,2,0)--(3,2,0)--(3,3,0)--(2,3,0)--cycle
(2,1,1)--(3,1,1)--(3,2,1)--(2,2,1)--cycle
(2,0,1)--(3,0,1)--(3,1,1)--(2,1,1)--cycle
(1,3,0)--(2,3,0)--(2,4,0)--(1,4,0)--cycle
(1,2,0)--(2,2,0)--(2,3,0)--(1,3,0)--cycle
(1,1,1)--(2,1,1)--(2,2,1)--(1,2,1)--cycle
(0,3,0)--(1,3,0)--(1,4,0)--(0,4,0)--cycle
;
\draw[line width=1pt,fill=black!20]
(3,1,0)--(3,2,0)--(3,2,1)--(3,1,1)--cycle
(3,0,0)--(3,1,0)--(3,1,1)--(3,0,1)--cycle
;
\draw[line width=1pt,fill=black!50]
(2,2,0)--(2,2,1)--(3,2,1)--(3,2,0)--cycle
(1,2,0)--(1,2,1)--(2,2,1)--(2,2,0)--cycle
;
\draw[line width=2pt](3,0,0)--(3,4,0)--(0,4,0)--(0,3,0)--(1,3,0)--(1,2,0)--(1,2,1)--(1,1,1)--(2,1,1)--(2,0,1)--(3,0,1)--(3,0,0)--cycle;
\end{scope}
\begin{scope}[xshift=18cm,yshift=0cm]
\clip(3,0,0)--(3,4,0)--(0,4,0)--(0,3,0)--(1,3,0)--(1,3,1)--(1,2,1)--(1,1,1)--(2,1,1)--(2,1,0)--(2,0,0)--(3,0,0)--cycle;
\draw[line width=1pt,fill=black!5]
(2,3,0)--(3,3,0)--(3,4,0)--(2,4,0)--cycle
(2,2,0)--(3,2,0)--(3,3,0)--(2,3,0)--cycle
(2,1,0)--(3,1,0)--(3,2,0)--(2,2,0)--cycle
(2,0,0)--(3,0,0)--(3,1,0)--(2,1,0)--cycle
(1,3,0)--(2,3,0)--(2,4,0)--(1,4,0)--cycle
(1,2,1)--(2,2,1)--(2,3,1)--(1,3,1)--cycle
(1,1,1)--(2,1,1)--(2,2,1)--(1,2,1)--cycle
(0,3,0)--(1,3,0)--(1,4,0)--(0,4,0)--cycle
;
\draw[line width=1pt,fill=black!20]
(2,2,0)--(2,3,0)--(2,3,1)--(2,2,1)--cycle
(2,1,0)--(2,2,0)--(2,2,1)--(2,1,1)--cycle
;
\draw[line width=1pt,fill=black!50]
(1,3,0)--(1,3,1)--(2,3,1)--(2,3,0)--cycle
;
\draw[line width=2pt](3,0,0)--(3,4,0)--(0,4,0)--(0,3,0)--(1,3,0)--(1,3,1)--(1,2,1)--(1,1,1)--(2,1,1)--(2,1,0)--(2,0,0)--(3,0,0)--cycle;
\end{scope}
\begin{scope}[xshift=0cm,yshift=-4cm]
\clip(3,0,0)--(3,4,0)--(0,4,0)--(0,3,0)--(1,3,0)--(1,3,1)--(1,2,1)--(1,1,1)--(2,1,1)--(2,0,1)--(3,0,1)--(3,0,0)--cycle;
\draw[line width=1pt,fill=black!5]
(2,3,0)--(3,3,0)--(3,4,0)--(2,4,0)--cycle
(2,2,0)--(3,2,0)--(3,3,0)--(2,3,0)--cycle
(2,1,1)--(3,1,1)--(3,2,1)--(2,2,1)--cycle
(2,0,1)--(3,0,1)--(3,1,1)--(2,1,1)--cycle
(1,3,0)--(2,3,0)--(2,4,0)--(1,4,0)--cycle
(1,2,1)--(2,2,1)--(2,3,1)--(1,3,1)--cycle
(1,1,1)--(2,1,1)--(2,2,1)--(1,2,1)--cycle
(0,3,0)--(1,3,0)--(1,4,0)--(0,4,0)--cycle
;
\draw[line width=1pt,fill=black!20]
(3,1,0)--(3,2,0)--(3,2,1)--(3,1,1)--cycle
(3,0,0)--(3,1,0)--(3,1,1)--(3,0,1)--cycle
(2,2,0)--(2,3,0)--(2,3,1)--(2,2,1)--cycle
;
\draw[line width=1pt,fill=black!50]
(2,2,0)--(2,2,1)--(3,2,1)--(3,2,0)--cycle
(1,3,0)--(1,3,1)--(2,3,1)--(2,3,0)--cycle
;
\draw[line width=2pt](3,0,0)--(3,4,0)--(0,4,0)--(0,3,0)--(1,3,0)--(1,3,1)--(1,2,1)--(1,1,1)--(2,1,1)--(2,0,1)--(3,0,1)--(3,0,0)--cycle;
\end{scope}
\begin{scope}[xshift=6cm,yshift=-4cm]
\clip(3,0,0)--(3,4,0)--(0,4,0)--(0,4,1)--(0,3,1)--(1,3,1)--(1,3,0)--(1,2,0)--(1,1,0)--(2,1,0)--(2,0,0)--(3,0,0)--cycle;
\draw[line width=1pt,fill=black!5]
(2,3,0)--(3,3,0)--(3,4,0)--(2,4,0)--cycle
(2,2,0)--(3,2,0)--(3,3,0)--(2,3,0)--cycle
(2,1,0)--(3,1,0)--(3,2,0)--(2,2,0)--cycle
(2,0,0)--(3,0,0)--(3,1,0)--(2,1,0)--cycle
(1,3,0)--(2,3,0)--(2,4,0)--(1,4,0)--cycle
(1,2,0)--(2,2,0)--(2,3,0)--(1,3,0)--cycle
(1,1,0)--(2,1,0)--(2,2,0)--(1,2,0)--cycle
(0,3,1)--(1,3,1)--(1,4,1)--(0,4,1)--cycle
;
\draw[line width=1pt,fill=black!20]
(1,3,0)--(1,4,0)--(1,4,1)--(1,3,1)--cycle
;
\draw[line width=1pt,fill=black!50]
(0,4,0)--(0,4,1)--(1,4,1)--(1,4,0)--cycle
;
\draw[line width=2pt](3,0,0)--(3,4,0)--(0,4,0)--(0,4,1)--(0,3,1)--(1,3,1)--(1,3,0)--(1,2,0)--(1,1,0)--(2,1,0)--(2,0,0)--(3,0,0)--cycle;
\end{scope}
\begin{scope}[xshift=12cm,yshift=-4cm]
\clip(3,0,0)--(3,4,0)--(0,4,0)--(0,4,1)--(0,3,1)--(1,3,1)--(1,2,1)--(1,1,1)--(2,1,1)--(2,1,0)--(2,0,0)--(3,0,0)--cycle;
\draw[line width=1pt,fill=black!5]
(2,3,0)--(3,3,0)--(3,4,0)--(2,4,0)--cycle
(2,2,0)--(3,2,0)--(3,3,0)--(2,3,0)--cycle
(2,1,0)--(3,1,0)--(3,2,0)--(2,2,0)--cycle
(2,0,0)--(3,0,0)--(3,1,0)--(2,1,0)--cycle
(1,3,1)--(2,3,1)--(2,4,1)--(1,4,1)--cycle
(1,2,1)--(2,2,1)--(2,3,1)--(1,3,1)--cycle
(1,1,1)--(2,1,1)--(2,2,1)--(1,2,1)--cycle
(0,3,1)--(1,3,1)--(1,4,1)--(0,4,1)--cycle
;
\draw[line width=1pt,fill=black!20]
(2,3,0)--(2,4,0)--(2,4,1)--(2,3,1)--cycle
(2,2,0)--(2,3,0)--(2,3,1)--(2,2,1)--cycle
(2,1,0)--(2,2,0)--(2,2,1)--(2,1,1)--cycle
;
\draw[line width=1pt,fill=black!50]
(1,4,0)--(1,4,1)--(2,4,1)--(2,4,0)--cycle
(0,4,0)--(0,4,1)--(1,4,1)--(1,4,0)--cycle
;
\draw[line width=2pt](3,0,0)--(3,4,0)--(0,4,0)--(0,4,1)--(0,3,1)--(1,3,1)--(1,2,1)--(1,1,1)--(2,1,1)--(2,1,0)--(2,0,0)--(3,0,0)--cycle;
\end{scope}
\begin{scope}[xshift=18cm,yshift=-4cm]
\clip(3,0,0)--(3,4,0)--(0,4,0)--(0,4,1)--(0,3,1)--(1,3,1)--(1,2,1)--(1,1,1)--(2,1,1)--(2,0,1)--(3,0,1)--(3,0,0)--cycle;
\draw[line width=1pt,fill=black!5]
(2,3,0)--(3,3,0)--(3,4,0)--(2,4,0)--cycle
(2,2,0)--(3,2,0)--(3,3,0)--(2,3,0)--cycle
(2,1,1)--(3,1,1)--(3,2,1)--(2,2,1)--cycle
(2,0,1)--(3,0,1)--(3,1,1)--(2,1,1)--cycle
(1,3,1)--(2,3,1)--(2,4,1)--(1,4,1)--cycle
(1,2,1)--(2,2,1)--(2,3,1)--(1,3,1)--cycle
(1,1,1)--(2,1,1)--(2,2,1)--(1,2,1)--cycle
(0,3,1)--(1,3,1)--(1,4,1)--(0,4,1)--cycle
;
\draw[line width=1pt,fill=black!20]
(3,1,0)--(3,2,0)--(3,2,1)--(3,1,1)--cycle
(3,0,0)--(3,1,0)--(3,1,1)--(3,0,1)--cycle
(2,3,0)--(2,4,0)--(2,4,1)--(2,3,1)--cycle
(2,2,0)--(2,3,0)--(2,3,1)--(2,2,1)--cycle
;
\draw[line width=1pt,fill=black!50]
(2,2,0)--(2,2,1)--(3,2,1)--(3,2,0)--cycle
(1,4,0)--(1,4,1)--(2,4,1)--(2,4,0)--cycle
(0,4,0)--(0,4,1)--(1,4,1)--(1,4,0)--cycle
;
\draw[line width=2pt](3,0,0)--(3,4,0)--(0,4,0)--(0,4,1)--(0,3,1)--(1,3,1)--(1,2,1)--(1,1,1)--(2,1,1)--(2,0,1)--(3,0,1)--(3,0,0)--cycle;
\end{scope}
\end{tikzpicture}
}

A \emph{part} of a rim-hook is a maximal set of cells (cubes) contained in the same row of the partition.
For example, the bottom right rim-hook in \reff{rimhooks} has three parts of sizes two, three and one.
See \reff{parts}.

\myfigh{}{parts}{
\begin{tikzpicture}[tdplot_main_coords]
\begin{scope}[xshift=0cm,yshift=0cm]
\clip(3,0,0)--(3,4,0)--(0,4,0)--(0,4,1)--(0,3,1)--(1,3,1)--(1,2,1)--(1,1,1)--(2,1,1)--(2,0,1)--(3,0,1)--(3,0,0)--cycle;
\draw[line width=1pt,fill=black!5]
(2,3,0)--(3,3,0)--(3,4,0)--(2,4,0)--cycle
(2,2,0)--(3,2,0)--(3,3,0)--(2,3,0)--cycle
;
\draw[line width=1pt,fill=red]
(2,0,1)--(3,0,1)--(3,1,1)--(2,1,1)--cycle
(2,1,1)--(3,1,1)--(3,2,1)--(2,2,1)--cycle
;
\draw[line width=1pt,fill=yellow]
(1,3,1)--(2,3,1)--(2,4,1)--(1,4,1)--cycle
(1,2,1)--(2,2,1)--(2,3,1)--(1,3,1)--cycle
(1,1,1)--(2,1,1)--(2,2,1)--(1,2,1)--cycle
;
\draw[line width=1pt,fill=cyan]
(0,3,1)--(1,3,1)--(1,4,1)--(0,4,1)--cycle
;
\draw[line width=1pt,fill=black!20]
(3,1,0)--(3,2,0)--(3,2,1)--(3,1,1)--cycle
(3,0,0)--(3,1,0)--(3,1,1)--(3,0,1)--cycle
(2,3,0)--(2,4,0)--(2,4,1)--(2,3,1)--cycle
(2,2,0)--(2,3,0)--(2,3,1)--(2,2,1)--cycle
;
\draw[line width=1pt,fill=black!50]
(2,2,0)--(2,2,1)--(3,2,1)--(3,2,0)--cycle
(1,4,0)--(1,4,1)--(2,4,1)--(2,4,0)--cycle
(0,4,0)--(0,4,1)--(1,4,1)--(1,4,0)--cycle
;
\draw[line width=2pt](3,0,0)--(3,4,0)--(0,4,0)--(0,4,1)--(0,3,1)--(1,3,1)--(1,2,1)--(1,1,1)--(2,1,1)--(2,0,1)--(3,0,1)--(3,0,0)--cycle;
\end{scope}
\end{tikzpicture}
}

Suppose you are given a reverse plane partition $\pi$ of shape $\lambda$ and a rim-hook $h$ of $\lambda$.
To insert $h$ into $\pi$ first try placing the brick on top of the reverse plane partition such that the cells of $h$ align with the corresponding cells of the shape of $\pi$.
See \reff{drop}.

\myfigh{}
{drop}
{
\begin{tikzpicture}[tdplot_main_coords]
\begin{scope}
\begin{scope}[xshift=0cm,yshift=-2cm]
\draw[line width=2pt,->](.5,.5,1.5)--(.5,.5,0);
\end{scope}
\begin{scope}[xshift=0cm,yshift=0cm]
\clip(1,0,0)--(1,2,0)--(0,2,0)--(0,2,1)--(0,1,1)--(0,1,1)--(0,0,1)--(1,0,1)--cycle;
\draw[line width=1pt,fill=red]
(0,0,1)--(1,0,1)--(1,1,1)--(0,1,1)--cycle
(0,1,1)--(1,1,1)--(1,2,1)--(0,2,1)--cycle
;
\draw[line width=1pt,fill=black!20]
(1,0,1)--(1,0,0)--(1,1,0)--(1,1,1)--cycle
(1,1,1)--(1,1,0)--(1,2,0)--(1,2,1)--cycle
;
\draw[line width=1pt,fill=black!50]
(0,2,1)--(1,2,1)--(1,2,0)--(0,2,0)--cycle
;
\draw[line width=2pt](1,0,0)--(1,2,0)--(0,2,0)--(0,2,1)--(0,1,1)--(0,1,1)--(0,0,1)--(1,0,1)--cycle;
\end{scope}
\begin{scope}[yshift=-44mm]
\begin{scope}
\clip(3,0,0)--(3,3,0)--(0,3,0)--(0,3,1)--(0,1,1)--(0,1,1)--(0,0,1)--(3,0,1)--cycle;
\draw[line width=1pt,fill=black!5]
(0,0,1)--(1,0,1)--(1,1,1)--(0,1,1)--cycle
(0,1,1)--(1,1,1)--(1,2,1)--(0,2,1)--cycle
(0,2,1)--(1,2,1)--(1,3,1)--(0,3,1)--cycle
(1,0,1)--(2,0,1)--(2,1,1)--(1,1,1)--cycle
(1,1,0)--(2,1,0)--(2,2,0)--(1,2,0)--cycle
(1,2,0)--(2,2,0)--(2,3,0)--(1,3,0)--cycle
(2,0,1)--(3,0,1)--(3,1,1)--(2,1,1)--cycle
(2,1,0)--(3,1,0)--(3,2,0)--(2,2,0)--cycle
(2,2,0)--(3,2,0)--(3,3,0)--(2,3,0)--cycle
;
\draw[line width=1pt,fill=black!20]
(1,1,1)--(1,1,0)--(1,2,0)--(1,2,1)--cycle
(1,2,1)--(1,2,0)--(1,3,0)--(1,3,1)--cycle
(3,0,1)--(3,0,0)--(3,1,0)--(3,1,1)--cycle
;
\draw[line width=1pt,fill=black!50]
(0,3,1)--(1,3,1)--(1,3,0)--(0,3,0)--cycle
(1,1,1)--(2,1,1)--(2,1,0)--(1,1,0)--cycle
(2,1,1)--(3,1,1)--(3,1,0)--(2,1,0)--cycle
;
\draw[line width=2pt](3,0,0)--(3,3,0)--(0,3,0)--(0,3,1)--(0,1,1)--(0,1,1)--(0,0,1)--(3,0,1)--cycle;
\end{scope}
\begin{scope}[xshift=5cm]
\clip(3,0,0)--(3,3,0)--(0,3,0)--(0,3,1)--(0,2,1)--(0,2,2)--(0,1,2)--(0,1,2)--(0,0,2)--(1,0,2)--(1,0,1)--(3,0,1)--cycle;
\draw[line width=1pt,fill=black!5]
(0,2,1)--(1,2,1)--(1,3,1)--(0,3,1)--cycle
(1,0,1)--(2,0,1)--(2,1,1)--(1,1,1)--cycle
(1,1,0)--(2,1,0)--(2,2,0)--(1,2,0)--cycle
(1,2,0)--(2,2,0)--(2,3,0)--(1,3,0)--cycle
(2,0,1)--(3,0,1)--(3,1,1)--(2,1,1)--cycle
(2,1,0)--(3,1,0)--(3,2,0)--(2,2,0)--cycle
(2,2,0)--(3,2,0)--(3,3,0)--(2,3,0)--cycle
;
\draw[line width=1pt,fill=red]
(0,0,2)--(1,0,2)--(1,1,2)--(0,1,2)--cycle
(0,1,2)--(1,1,2)--(1,2,2)--(0,2,2)--cycle
;
\draw[line width=1pt,fill=black!20]
(1,0,2)--(1,0,1)--(1,1,1)--(1,1,2)--cycle
(1,1,2)--(1,1,1)--(1,2,1)--(1,2,2)--cycle
(1,1,1)--(1,1,0)--(1,2,0)--(1,2,1)--cycle
(1,2,1)--(1,2,0)--(1,3,0)--(1,3,1)--cycle
(3,0,1)--(3,0,0)--(3,1,0)--(3,1,1)--cycle
;
\draw[line width=1pt,fill=black!50]
(0,2,2)--(1,2,2)--(1,2,1)--(0,2,1)--cycle
(0,3,1)--(1,3,1)--(1,3,0)--(0,3,0)--cycle
(1,1,1)--(2,1,1)--(2,1,0)--(1,1,0)--cycle
(2,1,1)--(3,1,1)--(3,1,0)--(2,1,0)--cycle
;
\draw[line width=2pt](3,0,0)--(3,3,0)--(0,3,0)--(0,3,1)--(0,2,1)--(0,2,2)--(0,1,2)--(0,1,2)--(0,0,2)--(1,0,2)--(1,0,1)--(3,0,1)--cycle;
\end{scope}
\end{scope}
\end{scope}
\begin{scope}[xshift=15cm]
\begin{scope}[xshift=0cm,yshift=-2cm]
\draw[line width=2pt,->](.5,.5,1.5)--(.5,.5,0);
\end{scope}
\begin{scope}[xshift=0cm,yshift=0cm]
\clip(2,0,0)--(2,1,0)--(1,1,0)--(1,2,0)--(0,2,0)--(0,2,1)--(0,1,1)--(0,1,1)--(0,0,1)--(2,0,1)--cycle;
\draw[line width=1pt,fill=green]
(0,0,1)--(1,0,1)--(1,1,1)--(0,1,1)--cycle
(0,1,1)--(1,1,1)--(1,2,1)--(0,2,1)--cycle
(1,0,1)--(2,0,1)--(2,1,1)--(1,1,1)--cycle
;
\draw[line width=1pt,fill=black!20]
(1,1,1)--(1,1,0)--(1,2,0)--(1,2,1)--cycle
(2,0,1)--(2,0,0)--(2,1,0)--(2,1,1)--cycle
;
\draw[line width=1pt,fill=black!50]
(0,2,1)--(1,2,1)--(1,2,0)--(0,2,0)--cycle
(1,1,1)--(2,1,1)--(2,1,0)--(1,1,0)--cycle
;
\draw[line width=2pt](2,0,0)--(2,1,0)--(1,1,0)--(1,2,0)--(0,2,0)--(0,2,1)--(0,1,1)--(0,1,1)--(0,0,1)--(2,0,1)--cycle;
\end{scope}
\begin{scope}[yshift=-44mm]
\begin{scope}[xshift=0cm,yshift=0cm]
\clip(3,0,0)--(3,3,0)--(0,3,0)--(0,3,1)--(0,1,1)--(0,1,1)--(0,0,1)--(3,0,1)--cycle;
\draw[line width=1pt,fill=black!5]
(0,0,1)--(1,0,1)--(1,1,1)--(0,1,1)--cycle
(0,1,1)--(1,1,1)--(1,2,1)--(0,2,1)--cycle
(0,2,1)--(1,2,1)--(1,3,1)--(0,3,1)--cycle
(1,2,0)--(2,2,0)--(2,3,0)--(1,3,0)--cycle
(2,1,0)--(3,1,0)--(3,2,0)--(2,2,0)--cycle
(2,2,0)--(3,2,0)--(3,3,0)--(2,3,0)--cycle
(1,0,1)--(2,0,1)--(2,1,1)--(1,1,1)--cycle
(2,0,1)--(3,0,1)--(3,1,1)--(2,1,1)--cycle
(1,1,1)--(2,1,1)--(2,2,1)--(1,2,1)--cycle
;
\draw[line width=1pt,fill=black!20]
(1,2,1)--(1,2,0)--(1,3,0)--(1,3,1)--cycle
(2,1,1)--(2,1,0)--(2,2,0)--(2,2,1)--cycle
(3,0,1)--(3,0,0)--(3,1,0)--(3,1,1)--cycle
;
\draw[line width=1pt,fill=black!50]
(0,3,1)--(1,3,1)--(1,3,0)--(0,3,0)--cycle
(1,2,1)--(2,2,1)--(2,2,0)--(1,2,0)--cycle
(2,1,1)--(3,1,1)--(3,1,0)--(2,1,0)--cycle
;
\draw[line width=2pt](3,0,0)--(3,3,0)--(0,3,0)--(0,3,1)--(0,1,1)--(0,1,1)--(0,0,1)--(3,0,1)--cycle;
\end{scope}
\begin{scope}[xshift=5cm,yshift=0cm]
\clip(3,0,0)--(3,3,0)--(0,3,0)--(0,3,1)--(0,2,1)--(0,2,2)--(0,1,2)--(0,1,2)--(0,0,2)--(2,0,2)--(2,0,1)--(3,0,1)--cycle;
\draw[line width=1pt,fill=black!5]
(0,2,1)--(1,2,1)--(1,3,1)--(0,3,1)--cycle
(1,1,1)--(2,1,1)--(2,2,1)--(1,2,1)--cycle
(1,2,0)--(2,2,0)--(2,3,0)--(1,3,0)--cycle
(2,0,1)--(3,0,1)--(3,1,1)--(2,1,1)--cycle
(2,1,0)--(3,1,0)--(3,2,0)--(2,2,0)--cycle
(2,2,0)--(3,2,0)--(3,3,0)--(2,3,0)--cycle
;
\draw[line width=1pt,fill=green]
(0,0,2)--(1,0,2)--(1,1,2)--(0,1,2)--cycle
(0,1,2)--(1,1,2)--(1,2,2)--(0,2,2)--cycle
(1,0,2)--(2,0,2)--(2,1,2)--(1,1,2)--cycle
;
\draw[line width=1pt,fill=black!20]
(1,1,2)--(1,1,1)--(1,2,1)--(1,2,2)--cycle
(1,2,1)--(1,2,0)--(1,3,0)--(1,3,1)--cycle
(2,0,2)--(2,0,1)--(2,1,1)--(2,1,2)--cycle
(2,1,1)--(2,1,0)--(2,2,0)--(2,2,1)--cycle
(3,0,1)--(3,0,0)--(3,1,0)--(3,1,1)--cycle
;
\draw[line width=1pt,fill=black!50]
(0,2,2)--(1,2,2)--(1,2,1)--(0,2,1)--cycle
(0,3,1)--(1,3,1)--(1,3,0)--(0,3,0)--cycle
(1,1,2)--(2,1,2)--(2,1,1)--(1,1,1)--cycle
(1,2,1)--(2,2,1)--(2,2,0)--(1,2,0)--cycle
(2,1,1)--(3,1,1)--(3,1,0)--(2,1,0)--cycle
;
\draw[line width=2pt](3,0,0)--(3,3,0)--(0,3,0)--(0,3,1)--(0,2,1)--(0,2,2)--(0,1,2)--(0,1,2)--(0,0,2)--(2,0,2)--(2,0,1)--(3,0,1)--cycle;
\end{scope}
\end{scope}
\end{scope}
\end{tikzpicture}
}

This try fails if $\pi$ does not support the brick, that is, if the resulting arrangement of cubes has a hole somewhere.
In this case try the following:
Cut off the maximal number of parts of the brick that can be inserted without creating holes, and place this piece of the brick on top of the reverse plane partition.
Then shift the remainder of the brick diagonally by one and try to insert it in the new position.
See \reff{cut1}.

\myfigh{}{cut1}{
\begin{tikzpicture}[tdplot_main_coords]
\begin{scope}
\begin{scope}[xshift=0cm,yshift=-2cm]
\draw[line width=2pt,->](.5,.5,1.5)--(.5,.5,0);
\end{scope}
\begin{scope}[xshift=0cm,yshift=0cm]
\clip(3,0,0)--(3,1,0)--(0,1,0)--(0,1,1)--(0,0,1)--(3,0,1)--cycle;
\draw[line width=1pt,fill=green]
(0,0,1)--(1,0,1)--(1,1,1)--(0,1,1)--cycle
;
\draw[line width=1pt,fill=cyan]
(1,0,1)--(2,0,1)--(2,1,1)--(1,1,1)--cycle
(2,0,1)--(3,0,1)--(3,1,1)--(2,1,1)--cycle
;
\draw[line width=1pt,fill=black!20]
(3,0,1)--(3,0,0)--(3,1,0)--(3,1,1)--cycle
;
\draw[line width=1pt,fill=black!50]
(0,1,1)--(1,1,1)--(1,1,0)--(0,1,0)--cycle
(1,1,1)--(2,1,1)--(2,1,0)--(1,1,0)--cycle
(2,1,1)--(3,1,1)--(3,1,0)--(2,1,0)--cycle
;
\draw[line width=2pt](3,0,0)--(3,1,0)--(0,1,0)--(0,1,1)--(0,0,1)--(3,0,1)--cycle;
\end{scope}
\begin{scope}[yshift=-44mm]
\begin{scope}[xshift=0cm,yshift=0cm]
\clip(3,0,0)--(3,3,0)--(0,3,0)--(0,3,1)--(0,1,1)--(0,1,1)--(0,0,1)--(1,0,1)--(1,0,0)--(3,0,0)--cycle;
\draw[line width=1pt,fill=black!5]
(0,0,1)--(1,0,1)--(1,1,1)--(0,1,1)--cycle
(0,1,1)--(1,1,1)--(1,2,1)--(0,2,1)--cycle
(0,2,1)--(1,2,1)--(1,3,1)--(0,3,1)--cycle
(1,0,0)--(2,0,0)--(2,1,0)--(1,1,0)--cycle
(1,1,0)--(2,1,0)--(2,2,0)--(1,2,0)--cycle
(1,2,0)--(2,2,0)--(2,3,0)--(1,3,0)--cycle
(2,0,0)--(3,0,0)--(3,1,0)--(2,1,0)--cycle
(2,1,0)--(3,1,0)--(3,2,0)--(2,2,0)--cycle
(2,2,0)--(3,2,0)--(3,3,0)--(2,3,0)--cycle
;
\draw[line width=1pt,fill=black!20]
(1,0,1)--(1,0,0)--(1,1,0)--(1,1,1)--cycle
(1,1,1)--(1,1,0)--(1,2,0)--(1,2,1)--cycle
(1,2,1)--(1,2,0)--(1,3,0)--(1,3,1)--cycle
;
\draw[line width=1pt,fill=black!50]
(0,3,1)--(1,3,1)--(1,3,0)--(0,3,0)--cycle
;
\draw[line width=2pt](3,0,0)--(3,3,0)--(0,3,0)--(0,3,1)--(0,1,1)--(0,1,1)--(0,0,1)--(1,0,1)--(1,0,0)--(3,0,0)--cycle;
\end{scope}
\begin{scope}[xshift=5cm,yshift=0cm]
\clip(3,0,0)--(3,3,0)--(0,3,0)--(0,3,1)--(0,1,1)--(0,1,1)--(0,0,1)--(3,0,1)--cycle;
\draw[line width=1pt,fill=black!5]
(0,0,1)--(1,0,1)--(1,1,1)--(0,1,1)--cycle
(0,1,1)--(1,1,1)--(1,2,1)--(0,2,1)--cycle
(0,2,1)--(1,2,1)--(1,3,1)--(0,3,1)--cycle
(1,2,0)--(2,2,0)--(2,3,0)--(1,3,0)--cycle
(2,1,0)--(3,1,0)--(3,2,0)--(2,2,0)--cycle
(2,2,0)--(3,2,0)--(3,3,0)--(2,3,0)--cycle
;
\draw[line width=1pt,fill=cyan]
(1,0,1)--(2,0,1)--(2,1,1)--(1,1,1)--cycle
(2,0,1)--(3,0,1)--(3,1,1)--(2,1,1)--cycle
;
\draw[line width=1pt,fill=green]
(1,1,1)--(2,1,1)--(2,2,1)--(1,2,1)--cycle
;
\draw[line width=1pt,fill=black!20]
(1,2,1)--(1,2,0)--(1,3,0)--(1,3,1)--cycle
(2,1,1)--(2,1,0)--(2,2,0)--(2,2,1)--cycle
(3,0,1)--(3,0,0)--(3,1,0)--(3,1,1)--cycle
;
\draw[line width=1pt,fill=black!50]
(0,3,1)--(1,3,1)--(1,3,0)--(0,3,0)--cycle
(1,2,1)--(2,2,1)--(2,2,0)--(1,2,0)--cycle
(2,1,1)--(3,1,1)--(3,1,0)--(2,1,0)--cycle
;
\draw[line width=2pt](3,0,0)--(3,3,0)--(0,3,0)--(0,3,1)--(0,1,1)--(0,1,1)--(0,0,1)--(3,0,1)--cycle;
\end{scope}
\end{scope}
\end{scope}
\begin{scope}[xshift=15cm]
\begin{scope}[xshift=0cm,yshift=-2cm]
\draw[line width=2pt,->](.5,.5,1.5)--(.5,.5,0);
\end{scope}
\begin{scope}[xshift=0cm,yshift=0cm]
\clip(2,0,0)--(2,1,0)--(1,1,0)--(1,2,0)--(0,2,0)--(0,2,1)--(0,1,1)--(0,1,1)--(0,0,1)--(2,0,1)--cycle;
\draw[line width=1pt,fill=red]
(1,0,1)--(2,0,1)--(2,1,1)--(1,1,1)--cycle
;
\draw[line width=1pt,fill=yellow]
(0,0,1)--(1,0,1)--(1,1,1)--(0,1,1)--cycle
(0,1,1)--(1,1,1)--(1,2,1)--(0,2,1)--cycle
;
\draw[line width=1pt,fill=black!20]
(1,1,1)--(1,1,0)--(1,2,0)--(1,2,1)--cycle
(2,0,1)--(2,0,0)--(2,1,0)--(2,1,1)--cycle
;
\draw[line width=1pt,fill=black!50]
(0,2,1)--(1,2,1)--(1,2,0)--(0,2,0)--cycle
(1,1,1)--(2,1,1)--(2,1,0)--(1,1,0)--cycle
;
\draw[line width=2pt](2,0,0)--(2,1,0)--(1,1,0)--(1,2,0)--(0,2,0)--(0,2,1)--(0,1,1)--(0,1,1)--(0,0,1)--(2,0,1)--cycle;
\end{scope}
\begin{scope}[yshift=-44mm]
\begin{scope}[xshift=0cm,yshift=0cm]
\clip(3,0,0)--(3,3,0)--(0,3,0)--(0,3,1)--(0,1,1)--(0,1,1)--(0,0,1)--(1,0,1)--(1,0,0)--(3,0,0)--cycle;
\draw[line width=1pt,fill=black!5]
(0,0,1)--(1,0,1)--(1,1,1)--(0,1,1)--cycle
(0,1,1)--(1,1,1)--(1,2,1)--(0,2,1)--cycle
(0,2,1)--(1,2,1)--(1,3,1)--(0,3,1)--cycle
(1,0,0)--(2,0,0)--(2,1,0)--(1,1,0)--cycle
(1,1,0)--(2,1,0)--(2,2,0)--(1,2,0)--cycle
(1,2,0)--(2,2,0)--(2,3,0)--(1,3,0)--cycle
(2,0,0)--(3,0,0)--(3,1,0)--(2,1,0)--cycle
(2,1,0)--(3,1,0)--(3,2,0)--(2,2,0)--cycle
(2,2,0)--(3,2,0)--(3,3,0)--(2,3,0)--cycle
;
\draw[line width=1pt,fill=black!20]
(1,0,1)--(1,0,0)--(1,1,0)--(1,1,1)--cycle
(1,1,1)--(1,1,0)--(1,2,0)--(1,2,1)--cycle
(1,2,1)--(1,2,0)--(1,3,0)--(1,3,1)--cycle
;
\draw[line width=1pt,fill=black!50]
(0,3,1)--(1,3,1)--(1,3,0)--(0,3,0)--cycle
;
\draw[line width=2pt](3,0,0)--(3,3,0)--(0,3,0)--(0,3,1)--(0,1,1)--(0,1,1)--(0,0,1)--(1,0,1)--(1,0,0)--(3,0,0)--cycle;
\end{scope}
\begin{scope}[xshift=5cm,yshift=0cm]
\clip(3,0,0)--(3,3,0)--(0,3,0)--(0,3,1)--(0,1,1)--(0,1,1)--(0,0,1)--(2,0,1)--(2,0,0)--(3,0,0)--cycle;
\draw[line width=1pt,fill=black!5]
(0,0,1)--(1,0,1)--(1,1,1)--(0,1,1)--cycle
(0,1,1)--(1,1,1)--(1,2,1)--(0,2,1)--cycle
(0,2,1)--(1,2,1)--(1,3,1)--(0,3,1)--cycle
(2,0,0)--(3,0,0)--(3,1,0)--(2,1,0)--cycle
(2,1,0)--(3,1,0)--(3,2,0)--(2,2,0)--cycle
(2,2,0)--(3,2,0)--(3,3,0)--(2,3,0)--cycle
;
\draw[line width=1pt,fill=yellow]
(1,1,1)--(2,1,1)--(2,2,1)--(1,2,1)--cycle
(1,2,1)--(2,2,1)--(2,3,1)--(1,3,1)--cycle
;
\draw[line width=1pt,fill=red]
(1,0,1)--(2,0,1)--(2,1,1)--(1,1,1)--cycle
;
\draw[line width=1pt,fill=black!20]
(2,0,1)--(2,0,0)--(2,1,0)--(2,1,1)--cycle
(2,1,1)--(2,1,0)--(2,2,0)--(2,2,1)--cycle
(2,2,1)--(2,2,0)--(2,3,0)--(2,3,1)--cycle
;
\draw[line width=1pt,fill=black!50]
(0,3,1)--(1,3,1)--(1,3,0)--(0,3,0)--cycle
(1,3,1)--(2,3,1)--(2,3,0)--(1,3,0)--cycle
;
\draw[line width=2pt](3,0,0)--(3,3,0)--(0,3,0)--(0,3,1)--(0,1,1)--(0,1,1)--(0,0,1)--(2,0,1)--(2,0,0)--(3,0,0)--cycle;
\end{scope}
\end{scope}
\end{scope}
\end{tikzpicture}
}

It is not demanded that the remainder of the brick is inserted at the same height as the initial segment.
Furthermore it is allowed to cut the brick multiple times if needed, as long as all cuts happen between different parts.
See \reff{cut2} for examples.

\myfigh{}
{cut2}
{\begin{tikzpicture}[tdplot_main_coords]
\begin{scope}
\begin{scope}[xshift=0cm,yshift=-2cm]
\draw[line width=2pt,->](.5,.5,1.5)--(.5,.5,0);
\end{scope}
\begin{scope}[xshift=0cm,yshift=0cm]
\clip(2,0,0)--(2,1,0)--(0,1,0)--(0,1,1)--(0,0,1)--(2,0,1)--cycle;
\draw[line width=1pt,fill=red]
(0,0,1)--(1,0,1)--(1,1,1)--(0,1,1)--cycle
;
\draw[line width=1pt,fill=yellow]
(1,0,1)--(2,0,1)--(2,1,1)--(1,1,1)--cycle
;
\draw[line width=1pt,fill=black!20]
(2,0,1)--(2,0,0)--(2,1,0)--(2,1,1)--cycle
;
\draw[line width=1pt,fill=black!50]
(0,1,1)--(1,1,1)--(1,1,0)--(0,1,0)--cycle
(1,1,1)--(2,1,1)--(2,1,0)--(1,1,0)--cycle
;
\draw[line width=2pt](2,0,0)--(2,1,0)--(0,1,0)--(0,1,1)--(0,0,1)--(2,0,1)--cycle;
\end{scope}
\begin{scope}[yshift=-44mm]
\begin{scope}[xshift=0cm,yshift=0cm]
\clip(3,0,0)--(3,3,0)--(0,3,0)--(0,3,1)--(0,2,1)--(0,2,2)--(0,1,2)--(0,1,2)--(0,0,2)--(1,0,2)--(1,0,1)--(3,0,1)--cycle;
\draw[line width=1pt,fill=black!5]
(0,0,2)--(1,0,2)--(1,1,2)--(0,1,2)--cycle
(0,1,2)--(1,1,2)--(1,2,2)--(0,2,2)--cycle
(0,2,1)--(1,2,1)--(1,3,1)--(0,3,1)--cycle
(1,0,1)--(2,0,1)--(2,1,1)--(1,1,1)--cycle
(1,1,0)--(2,1,0)--(2,2,0)--(1,2,0)--cycle
(1,2,0)--(2,2,0)--(2,3,0)--(1,3,0)--cycle
(2,0,1)--(3,0,1)--(3,1,1)--(2,1,1)--cycle
(2,1,0)--(3,1,0)--(3,2,0)--(2,2,0)--cycle
(2,2,0)--(3,2,0)--(3,3,0)--(2,3,0)--cycle
;
\draw[line width=1pt,fill=black!20]
(1,0,2)--(1,0,1)--(1,1,1)--(1,1,2)--cycle
(1,1,2)--(1,1,1)--(1,2,1)--(1,2,2)--cycle
(1,1,1)--(1,1,0)--(1,2,0)--(1,2,1)--cycle
(1,2,1)--(1,2,0)--(1,3,0)--(1,3,1)--cycle
(3,0,1)--(3,0,0)--(3,1,0)--(3,1,1)--cycle
;
\draw[line width=1pt,fill=black!50]
(0,2,2)--(1,2,2)--(1,2,1)--(0,2,1)--cycle
(0,3,1)--(1,3,1)--(1,3,0)--(0,3,0)--cycle
(1,1,1)--(2,1,1)--(2,1,0)--(1,1,0)--cycle
(2,1,1)--(3,1,1)--(3,1,0)--(2,1,0)--cycle
;
\draw[line width=2pt](3,0,0)--(3,3,0)--(0,3,0)--(0,3,1)--(0,2,1)--(0,2,2)--(0,1,2)--(0,1,2)--(0,0,2)--(1,0,2)--(1,0,1)--(3,0,1)--cycle;
\end{scope}
\begin{scope}[xshift=5cm,yshift=0cm]
\clip(3,0,0)--(3,3,0)--(0,3,0)--(0,3,1)--(0,2,1)--(0,2,2)--(0,1,2)--(0,1,2)--(0,0,2)--(2,0,2)--(2,0,1)--(3,0,1)--cycle;
\draw[line width=1pt,fill=black!5]
(0,0,2)--(1,0,2)--(1,1,2)--(0,1,2)--cycle
(0,1,2)--(1,1,2)--(1,2,2)--(0,2,2)--cycle
(0,2,1)--(1,2,1)--(1,3,1)--(0,3,1)--cycle
(1,2,0)--(2,2,0)--(2,3,0)--(1,3,0)--cycle
(2,0,1)--(3,0,1)--(3,1,1)--(2,1,1)--cycle
(2,1,0)--(3,1,0)--(3,2,0)--(2,2,0)--cycle
(2,2,0)--(3,2,0)--(3,3,0)--(2,3,0)--cycle
;
\draw[line width=1pt,fill=yellow]
(1,0,2)--(2,0,2)--(2,1,2)--(1,1,2)--cycle
;
\draw[line width=1pt,fill=red]
(1,1,1)--(2,1,1)--(2,2,1)--(1,2,1)--cycle
;
\draw[line width=1pt,fill=black!20]
(1,1,2)--(1,1,1)--(1,2,1)--(1,2,2)--cycle
(1,2,1)--(1,2,0)--(1,3,0)--(1,3,1)--cycle
(2,0,2)--(2,0,1)--(2,1,1)--(2,1,2)--cycle
(2,1,1)--(2,1,0)--(2,2,0)--(2,2,1)--cycle
(3,0,1)--(3,0,0)--(3,1,0)--(3,1,1)--cycle
;
\draw[line width=1pt,fill=black!50]
(0,2,2)--(1,2,2)--(1,2,1)--(0,2,1)--cycle
(0,3,1)--(1,3,1)--(1,3,0)--(0,3,0)--cycle
(1,1,2)--(2,1,2)--(2,1,1)--(1,1,1)--cycle
(1,2,1)--(2,2,1)--(2,2,0)--(1,2,0)--cycle
(2,1,1)--(3,1,1)--(3,1,0)--(2,1,0)--cycle
;
\draw[line width=2pt](3,0,0)--(3,3,0)--(0,3,0)--(0,3,1)--(0,2,1)--(0,2,2)--(0,1,2)--(0,1,2)--(0,0,2)--(2,0,2)--(2,0,1)--(3,0,1)--cycle;
\end{scope}
\end{scope}
\end{scope}
\begin{scope}[xshift=15cm]
\begin{scope}[xshift=0cm,yshift=-2cm]
\draw[line width=2pt,->](.5,.5,1.5)--(.5,.5,0);
\end{scope}
\begin{scope}[xshift=0cm,yshift=0cm]
\clip(3,0,0)--(3,1,0)--(0,1,0)--(0,1,1)--(0,0,1)--(3,0,1)--cycle;
\draw[line width=1pt,fill=yellow]
(0,0,1)--(1,0,1)--(1,1,1)--(0,1,1)--cycle
;
\draw[line width=1pt,fill=cyan]
(1,0,1)--(2,0,1)--(2,1,1)--(1,1,1)--cycle
;
\draw[line width=1pt,fill=green]
(2,0,1)--(3,0,1)--(3,1,1)--(2,1,1)--cycle
;
\draw[line width=1pt,fill=black!20]
(3,0,1)--(3,0,0)--(3,1,0)--(3,1,1)--cycle
;
\draw[line width=1pt,fill=black!50]
(0,1,1)--(1,1,1)--(1,1,0)--(0,1,0)--cycle
(1,1,1)--(2,1,1)--(2,1,0)--(1,1,0)--cycle
(2,1,1)--(3,1,1)--(3,1,0)--(2,1,0)--cycle
;
\draw[line width=2pt](3,0,0)--(3,1,0)--(0,1,0)--(0,1,1)--(0,0,1)--(3,0,1)--cycle;
\end{scope}
\begin{scope}[yshift=-44mm]
\begin{scope}[xshift=0cm,yshift=0cm]
\clip(3,0,0)--(3,3,0)--(0,3,0)--(0,3,1)--(0,1,1)--(0,1,1)--(0,0,1)--(2,0,1)--(2,0,0)--(3,0,0)--cycle;
\draw[line width=1pt,fill=black!5]
(0,0,1)--(1,0,1)--(1,1,1)--(0,1,1)--cycle
(0,1,1)--(1,1,1)--(1,2,1)--(0,2,1)--cycle
(0,2,1)--(1,2,1)--(1,3,1)--(0,3,1)--cycle
(2,0,0)--(3,0,0)--(3,1,0)--(2,1,0)--cycle
(2,1,0)--(3,1,0)--(3,2,0)--(2,2,0)--cycle
(2,2,0)--(3,2,0)--(3,3,0)--(2,3,0)--cycle
(1,1,1)--(2,1,1)--(2,2,1)--(1,2,1)--cycle
(1,2,1)--(2,2,1)--(2,3,1)--(1,3,1)--cycle
(1,0,1)--(2,0,1)--(2,1,1)--(1,1,1)--cycle
;
\draw[line width=1pt,fill=black!20]
(2,0,1)--(2,0,0)--(2,1,0)--(2,1,1)--cycle
(2,1,1)--(2,1,0)--(2,2,0)--(2,2,1)--cycle
(2,2,1)--(2,2,0)--(2,3,0)--(2,3,1)--cycle
;
\draw[line width=1pt,fill=black!50]
(0,3,1)--(1,3,1)--(1,3,0)--(0,3,0)--cycle
(1,3,1)--(2,3,1)--(2,3,0)--(1,3,0)--cycle
;
\draw[line width=2pt](3,0,0)--(3,3,0)--(0,3,0)--(0,3,1)--(0,1,1)--(0,1,1)--(0,0,1)--(2,0,1)--(2,0,0)--(3,0,0)--cycle;
\end{scope}
\begin{scope}[xshift=5cm,yshift=0cm]
\clip(3,0,0)--(3,3,0)--(0,3,0)--(0,3,1)--(0,1,1)--(0,1,1)--(0,0,1)--(3,0,1)--cycle;
\draw[line width=1pt,fill=black!5]
(0,0,1)--(1,0,1)--(1,1,1)--(0,1,1)--cycle
(0,1,1)--(1,1,1)--(1,2,1)--(0,2,1)--cycle
(0,2,1)--(1,2,1)--(1,3,1)--(0,3,1)--cycle
(1,0,1)--(2,0,1)--(2,1,1)--(1,1,1)--cycle
(1,1,1)--(2,1,1)--(2,2,1)--(1,2,1)--cycle
(1,2,1)--(2,2,1)--(2,3,1)--(1,3,1)--cycle
;
\draw[line width=1pt,fill=green]
(2,0,1)--(3,0,1)--(3,1,1)--(2,1,1)--cycle
;
\draw[line width=1pt,fill=cyan]
(2,1,1)--(3,1,1)--(3,2,1)--(2,2,1)--cycle
;
\draw[line width=1pt,fill=yellow]
(2,2,1)--(3,2,1)--(3,3,1)--(2,3,1)--cycle
;
\draw[line width=1pt,fill=black!20]
(3,0,1)--(3,0,0)--(3,1,0)--(3,1,1)--cycle
(3,1,1)--(3,1,0)--(3,2,0)--(3,2,1)--cycle
(3,2,1)--(3,2,0)--(3,3,0)--(3,3,1)--cycle
;
\draw[line width=1pt,fill=black!50]
(0,3,1)--(1,3,1)--(1,3,0)--(0,3,0)--cycle
(1,3,1)--(2,3,1)--(2,3,0)--(1,3,0)--cycle
(2,3,1)--(3,3,1)--(3,3,0)--(2,3,0)--cycle
;
\draw[line width=2pt](3,0,0)--(3,3,0)--(0,3,0)--(0,3,1)--(0,1,1)--(0,1,1)--(0,0,1)--(3,0,1)--cycle;
\end{scope}
\end{scope}
\end{scope}
\end{tikzpicture}
}

If this procedure terminates successfully we say $h$ \emph{inserts} into $\pi$ and we denote the resulting reverse plane partition by $h*\pi$.
Sometimes the algorithm fails to produce a reverse plane partition.
\reff{fail} shows three examples where $h$ does not insert into $\pi$.
Note in particular the third example, in which insertion fails because it is demanded that each part of $h$ remains intact and cannot be cut in two.

\bigskip
\myfigh{}{fail}{
\begin{tikzpicture}[tdplot_main_coords]
\begin{scope}[xshift=0cm]
\begin{scope}[xshift=0cm,yshift=-2cm]
\draw[line width=2pt,->](.5,.5,1.5)--(.5,.5,0);
\end{scope}
\begin{scope}[xshift=0cm,yshift=0cm]
\clip(3,0,0)--(3,1,0)--(1,1,0)--(1,2,0)--(0,2,0)--(0,2,1)--(0,1,1)--(0,1,1)--(0,0,1)--(3,0,1)--cycle;
\draw[line width=1pt,fill=black!5]
(0,0,1)--(1,0,1)--(1,1,1)--(0,1,1)--cycle
(0,1,1)--(1,1,1)--(1,2,1)--(0,2,1)--cycle
(1,0,1)--(2,0,1)--(2,1,1)--(1,1,1)--cycle
(2,0,1)--(3,0,1)--(3,1,1)--(2,1,1)--cycle
;
\draw[line width=1pt,fill=black!20]
(1,1,1)--(1,1,0)--(1,2,0)--(1,2,1)--cycle
(3,0,1)--(3,0,0)--(3,1,0)--(3,1,1)--cycle
;
\draw[line width=1pt,fill=black!50]
(0,2,1)--(1,2,1)--(1,2,0)--(0,2,0)--cycle
(1,1,1)--(2,1,1)--(2,1,0)--(1,1,0)--cycle
(2,1,1)--(3,1,1)--(3,1,0)--(2,1,0)--cycle
;
\draw[line width=2pt](3,0,0)--(3,1,0)--(1,1,0)--(1,2,0)--(0,2,0)--(0,2,1)--(0,1,1)--(0,1,1)--(0,0,1)--(3,0,1)--cycle;
\end{scope}
\begin{scope}[xshift=0cm,yshift=-44mm]
\clip(3,0,0)--(3,3,0)--(0,3,0)--(0,3,0)--(0,2,0)--(0,2,1)--(0,1,1)--(0,1,1)--(0,0,1)--(1,0,1)--(1,0,0)--(3,0,0)--cycle;
\draw[line width=1pt,fill=black!5]
(0,0,1)--(1,0,1)--(1,1,1)--(0,1,1)--cycle
(0,1,1)--(1,1,1)--(1,2,1)--(0,2,1)--cycle
(0,2,0)--(1,2,0)--(1,3,0)--(0,3,0)--cycle
(1,0,0)--(2,0,0)--(2,1,0)--(1,1,0)--cycle
(1,1,0)--(2,1,0)--(2,2,0)--(1,2,0)--cycle
(1,2,0)--(2,2,0)--(2,3,0)--(1,3,0)--cycle
(2,0,0)--(3,0,0)--(3,1,0)--(2,1,0)--cycle
(2,1,0)--(3,1,0)--(3,2,0)--(2,2,0)--cycle
(2,2,0)--(3,2,0)--(3,3,0)--(2,3,0)--cycle
;
\draw[line width=1pt,fill=black!20]
(1,0,1)--(1,0,0)--(1,1,0)--(1,1,1)--cycle
(1,1,1)--(1,1,0)--(1,2,0)--(1,2,1)--cycle
;
\draw[line width=1pt,fill=black!50]
(0,2,1)--(1,2,1)--(1,2,0)--(0,2,0)--cycle
;
\draw[line width=2pt](3,0,0)--(3,3,0)--(0,3,0)--(0,3,0)--(0,2,0)--(0,2,1)--(0,1,1)--(0,1,1)--(0,0,1)--(1,0,1)--(1,0,0)--(3,0,0)--cycle;
\end{scope}
\end{scope}
\begin{scope}[xshift=10cm]
\begin{scope}[xshift=0cm,yshift=-2cm]
\draw[line width=2pt,->](.5,.5,1.5)--(.5,.5,0);
\end{scope}
\begin{scope}[xshift=0cm,yshift=0cm]
\clip(3,0,0)--(3,1,0)--(0,1,0)--(0,1,1)--(0,0,1)--(3,0,1)--cycle;
\draw[line width=1pt,fill=black!5]
(0,0,1)--(1,0,1)--(1,1,1)--(0,1,1)--cycle
(1,0,1)--(2,0,1)--(2,1,1)--(1,1,1)--cycle
(2,0,1)--(3,0,1)--(3,1,1)--(2,1,1)--cycle
;
\draw[line width=1pt,fill=black!20]
(3,0,1)--(3,0,0)--(3,1,0)--(3,1,1)--cycle
;
\draw[line width=1pt,fill=black!50]
(0,1,1)--(1,1,1)--(1,1,0)--(0,1,0)--cycle
(1,1,1)--(2,1,1)--(2,1,0)--(1,1,0)--cycle
(2,1,1)--(3,1,1)--(3,1,0)--(2,1,0)--cycle
;
\draw[line width=2pt](3,0,0)--(3,1,0)--(0,1,0)--(0,1,1)--(0,0,1)--(3,0,1)--cycle;
\end{scope}
\begin{scope}[xshift=0cm,yshift=-44mm]
\clip(3,0,0)--(3,3,0)--(0,3,0)--(0,3,1)--(0,1,1)--(0,1,1)--(0,0,1)--(2,0,1)--(2,0,0)--(3,0,0)--cycle;
\draw[line width=1pt,fill=black!5]
(0,0,1)--(1,0,1)--(1,1,1)--(0,1,1)--cycle
(0,1,1)--(1,1,1)--(1,2,1)--(0,2,1)--cycle
(0,2,1)--(1,2,1)--(1,3,1)--(0,3,1)--cycle
(1,0,1)--(2,0,1)--(2,1,1)--(1,1,1)--cycle
(1,1,1)--(2,1,1)--(2,2,1)--(1,2,1)--cycle
(1,2,0)--(2,2,0)--(2,3,0)--(1,3,0)--cycle
(2,0,0)--(3,0,0)--(3,1,0)--(2,1,0)--cycle
(2,1,0)--(3,1,0)--(3,2,0)--(2,2,0)--cycle
(2,2,0)--(3,2,0)--(3,3,0)--(2,3,0)--cycle
;
\draw[line width=1pt,fill=black!20]
(1,2,1)--(1,2,0)--(1,3,0)--(1,3,1)--cycle
(2,0,1)--(2,0,0)--(2,1,0)--(2,1,1)--cycle
(2,1,1)--(2,1,0)--(2,2,0)--(2,2,1)--cycle
;
\draw[line width=1pt,fill=black!50]
(0,3,1)--(1,3,1)--(1,3,0)--(0,3,0)--cycle
(1,2,1)--(2,2,1)--(2,2,0)--(1,2,0)--cycle
;
\draw[line width=2pt](3,0,0)--(3,3,0)--(0,3,0)--(0,3,1)--(0,1,1)--(0,1,1)--(0,0,1)--(2,0,1)--(2,0,0)--(3,0,0)--cycle;
\end{scope}
\end{scope}
\begin{scope}[xshift=20cm]
\begin{scope}[xshift=0cm,yshift=-2cm]
\draw[line width=2pt,->](.5,.5,1.5)--(.5,.5,0);
\end{scope}
\begin{scope}[xshift=0cm,yshift=0cm]
\clip(1,0,0)--(1,2,0)--(0,2,0)--(0,2,1)--(0,1,1)--(0,1,1)--(0,0,1)--(1,0,1)--cycle;
\draw[line width=1pt,fill=black!5]
(0,0,1)--(1,0,1)--(1,1,1)--(0,1,1)--cycle
(0,1,1)--(1,1,1)--(1,2,1)--(0,2,1)--cycle
;
\draw[line width=1pt,fill=black!20]
(1,0,1)--(1,0,0)--(1,1,0)--(1,1,1)--cycle
(1,1,1)--(1,1,0)--(1,2,0)--(1,2,1)--cycle
;
\draw[line width=1pt,fill=black!50]
(0,2,1)--(1,2,1)--(1,2,0)--(0,2,0)--cycle
;
\draw[line width=2pt](1,0,0)--(1,2,0)--(0,2,0)--(0,2,1)--(0,1,1)--(0,1,1)--(0,0,1)--(1,0,1)--cycle;
\end{scope}
\begin{scope}[xshift=0cm,yshift=-44mm]
\clip(3,0,0)--(3,3,0)--(0,3,0)--(0,3,1)--(0,1,1)--(0,1,2)--(0,0,2)--(1,0,2)--(1,0,1)--(2,0,1)--(2,0,0)--(3,0,0)--cycle;
\draw[line width=1pt,fill=black!5]
(0,0,2)--(1,0,2)--(1,1,2)--(0,1,2)--cycle
(0,1,1)--(1,1,1)--(1,2,1)--(0,2,1)--cycle
(0,2,1)--(1,2,1)--(1,3,1)--(0,3,1)--cycle
(1,0,1)--(2,0,1)--(2,1,1)--(1,1,1)--cycle
(1,1,1)--(2,1,1)--(2,2,1)--(1,2,1)--cycle
(1,2,0)--(2,2,0)--(2,3,0)--(1,3,0)--cycle
(2,0,0)--(3,0,0)--(3,1,0)--(2,1,0)--cycle
(2,1,0)--(3,1,0)--(3,2,0)--(2,2,0)--cycle
(2,2,0)--(3,2,0)--(3,3,0)--(2,3,0)--cycle
;
\draw[line width=1pt,fill=black!20]
(1,0,2)--(1,0,1)--(1,1,1)--(1,1,2)--cycle
(1,2,1)--(1,2,0)--(1,3,0)--(1,3,1)--cycle
(2,0,1)--(2,0,0)--(2,1,0)--(2,1,1)--cycle
(2,1,1)--(2,1,0)--(2,2,0)--(2,2,1)--cycle
;
\draw[line width=1pt,fill=black!50]
(0,1,2)--(1,1,2)--(1,1,1)--(0,1,1)--cycle
(0,3,1)--(1,3,1)--(1,3,0)--(0,3,0)--cycle
(1,2,1)--(2,2,1)--(2,2,0)--(1,2,0)--cycle
;
\draw[line width=2pt](3,0,0)--(3,3,0)--(0,3,0)--(0,3,1)--(0,1,1)--(0,1,2)--(0,0,2)--(1,0,2)--(1,0,1)--(2,0,1)--(2,0,0)--(3,0,0)--cycle;
\end{scope}
\end{scope}
\end{tikzpicture}
}
\bigskip

Our main results can be phrased as follows.
First of all every reverse plane partition can be built as described above using only bricks of rim-hook shape as building blocks.
Secondly, given a multi-set of bricks there is always a way to sort them (lexicographically) such that they can be successively inserted into the zero reverse plane partition.
Thirdly, each reverse plane partition can be built in a unique way such that all building blocks are inserted in lexicographic order.
Altogether we obtain a bijective correspondence between reverse plane partitions whose shape is a fixed partition $\lambda$ and the multi-sets of rim-hooks of $\lambda$.

\section{Candidates and rim-hooks}\label{Section:rimhooks}

In this section we fix notation concerning partitions and rim-hooks, and introduce some concepts that will be used throughout the remainder of the paper.

\smallskip
Let $\N=\{0,1,2,\dots\}$ denote the set of non-negative integers.
Given $n\in\N$ set $[n]=\{1,\dots,n\}$.
A \emph{cell} is a pair $(i,j)\in\Z^2$.
Denote the \emph{north}, \emph{east}, \emph{south} and \emph{west neighbours} of $u=(i,j)$ by
\begin{align*}
\n u=(i-1,j),&&
\e u=(i,j+1),&&
\s u=(i+1,j),&&
\w u=(i,j-1).
\end{align*}
A \emph{partition} $\lambda$ is a weakly decreasing sequence $\lambda_1\geq\lambda_2\geq\dots\geq\lambda_r>0$ of positive integers.
The elements $\lambda_i$ are called \emph{parts} of the partition.
The number of parts is called the \emph{length} of the partition and is denoted by $\ell(\lambda)$.
We identify each partition with a set of cells $\lambda=\{(i,j):i\in[\ell(\lambda)],j\in[\lambda_{i}]\}$ called the \emph{Young diagram} of $\lambda$.
The \emph{conjugate} of a partition $\lambda$ is the partition $\lambda'=\{(j,i):(i,j)\in\lambda\}$.
The \emph{hook} $H(u)$ of a cell $u\in\lambda$ consists of the cell $u$ itself and those cells $v\in\lambda$ that lie directly east of $u$ or directly south of $u$.
The \emph{hook-length} $h(u)=\lambda_i+\lambda_j'-i-j+1$ denotes the cardinality of $H(u)$.
For example, in the partition $\lambda=(4,3,1)$ above we have $H(1,2)=\{(1,2),(1,3),(1,4),(2,2)\}$ and thus $h(1,2)=4$.

A \emph{reverse plane partition} of shape $\lambda$ is a map $\pi:\lambda\to\N$ such that $\pi(u)\leq\pi(\e u)$ and $\pi(u)\leq\pi(\s u)$ for all $u\in\lambda$.
By convention $\pi(i,j)=0$ if $i\leq0$ or $j\leq0$ and $\pi(i,j)=\infty$ if $i,j\geq1$ but $(i,j)\notin\lambda$.
Let $\rpp_{\lambda}$ denote the set of reverse plane partitions of shape $\lambda$.
\reff{partition} shows the partition $\lambda=(4,3,1)$, a reverse plane partition $\pi$ of shape $\lambda$ and the representation of $\pi$ as an arrangement of stacks of cubes.
The map $\pi:\lambda\to\N$ defined by $\pi(u)=0$ for all $u\in\lambda$ is called the \emph{zero reverse plane partition}.
The \emph{size} of a reverse plane partition $\pi\in\rpp_{\lambda}$ is defined as $\abs{\pi}=\sum_{u\in\lambda}\pi(u)$.

\smallskip
A \emph{north-east-path} in $\lambda$ is a sequence $P=(u_0,u_1,\dots,u_{s})$ of cells $u_k\in\lambda$ such that $u_{k}\in\{\n u_{k-1},\e u_{k-1}\}$ for all $k\in[s]$.
We call $\ell(P)=s$ the \emph{length}, $\alpha(P)=u_0$ the \emph{head} and $\omega(P)=u_{s}$ the \emph{tail} of the path $P$.
Sometimes it is more convenient to consider south-west-paths instead.
A \emph{south-west-path} in $\lambda$ is a sequence $Q=(v_0,v_1,\dots,v_{s})$ of cells $v_k\in\lambda$ such that $v_k\in\{\s v_{k-1},\w v_{k-1}\}$ for all $k\in[s]$.
Denote by $P'=(u_{s},\dots,u_1,u_0)$ the \emph{reverse path} of $P$.
Clearly the reverse of a south-west-path is a north-east-path and vice versa.
Given a south-west-path $Q$, set $\ell(Q)=s$, $\alpha(Q)=v_s$ and $\omega(Q)=v_0$ so that all notions are independent of the fact whether $Q$ is regarded as a north-east-path or as a south-west-path.
That is, $\alpha(P)=\alpha(P')$, $\omega(P)=\omega(P')$ and so forth.
If no special care is required we sometimes say \emph{path} to mean either north-east-path or south-west-path or both.

\smallskip
A \emph{rim-hook} of $\lambda$ is a north-east-path $h$ in $\lambda$ such that $\s\alpha(h)\notin\lambda$, $\e\omega(h)\notin\lambda$ and $\e\s u\notin\lambda$ for all $u\in h$.
For each cell $(i,j)\in\lambda$ there is a (unique) rim-hook $h$ with $\alpha(h)=(\lambda_j',j)$ and $\omega(h)=(i,\lambda_i)$.
This correspondence is a bijection between the cells of $\lambda$ and the rim-hooks of $\lambda$.
Denote the rim-hook corresponding to the cell $u\in\lambda$ by $h^u$.
Note that the length of the rim-hook $h^u$ is equal to the hook-length of the cell $u$, that is, $\ell(h^u)=h(u)$.

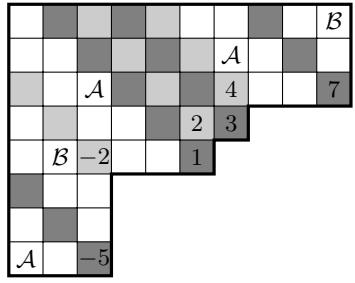
\myfig{The contents of the outer corners are $-5,1,3,7$. The contents of the inner corners are $-2,2,4$.}
{innerouter}
{
\begin{tikzpicture}
\fill[black!50]
(0,2)--(1,2)--(1,1)--(2,1)--(2,0)--(3,0)--(3,1)--(2,1)--(2,2)--(1,2)--(1,3)--(0,3)--cycle
(1,7)--(2,7)--(2,6)--(3,6)--(3,5)--(4,5)--(4,4)--(5,4)--(5,3)--(6,3)--(6,4)--(7,4)--(7,5)--(6,5)--(6,6)--(5,6)--(5,7)--(4,7)--(4,8)--(1,8)--cycle
(7,7)--(8,7)--(8,6)--(9,6)--(9,5)--(10,5)--(10,6)--(9,6)--(9,7)--(8,7)--(8,8)--(7,8)--cycle
;
\fill[black!20]
(0,5)--(1,5)--(1,4)--(2,4)--(2,3)--(3,3)--(3,4)--(2,4)--(2,5)--(1,5)--(1,6)--(0,6)--cycle
(2,7)--(3,7)--(3,6)--(4,6)--(4,5)--(5,5)--(5,4)--(6,4)--(6,5)--(5,5)--(5,6)--(4,6)--(4,7)--(3,7)--(3,8)--(2,8)--cycle
(4,7)--(5,7)--(5,6)--(6,6)--(6,5)--(7,5)--(7,6)--(6,6)--(6,7)--(5,7)--(5,8)--(4,8)--cycle
;
\draw(0,1)--(3,1)(0,2)--(3,2)(0,3)--(3,3)(0,4)--(6,4)(0,5)--(7,5)(0,6)--(10,6)(0,7)--(10,7)(1,8)--(1,0)(2,8)--(2,0)(3,8)--(3,3)(4,8)--(4,3)(5,8)--(5,3)(6,8)--(6,4)(7,8)--(7,5)(8,8)--(8,5)(9,8)--(9,5)
;
\draw[very thick](0,0)--(3,0)--(3,3)--(6,3)--(6,4)--(7,4)--(7,5)--(10,5)--(10,8)--(0,8)--cycle
;
\draw[xshift=5mm,yshift=5mm]
(2,0)node{\small{$-5$}}(2,3)node{\small{$-2$}}(5,3)node{\small{$1$}}(5,4)node{\small{$2$}}(6,4)node{\small{$3$}}(6,5)node{\small{$4$}}(9,5)node{\small{$7$}}(0,0)node{\small{$\A$}}(2,5)node{\small{$\A$}}(6,6)node{\small{$\A$}}(1,3)node{\small{$\B$}}(9,7)node{\small{$\B$}}
;
\end{tikzpicture}
}

\smallskip
A cell $u\in\lambda$ is called \emph{outer corner} if $\e u,\s u\notin\lambda$ and \emph{inner corner} if $\e u,\s u\in\lambda$ but $\e\s u\notin\lambda$.
The \emph{content} of a cell $u=(i,j)$ is defined as
\begin{eq*}
c(u)=j-i.
\end{eq*}
Let $i_1,\dots,i_r$ be the contents of the inner corners of $\lambda$ and $o_1,\dots,o_{r+1}$ be the contents of the outer corners of $\lambda$ ordered such that
\begin{eq*}
o_1<i_1<o_2<\dots<o_r<i_r<o_{r+1}.
\end{eq*}
Divide $\lambda$ into four regions
\begin{eq*}
\begin{split}
\I&=\{u\in\lambda:c(u)=i_k\text{ for some }k\in[r]\}\\
\O&=\{u\in\lambda:c(u)=o_k\text{ for some }k\in[r+1]\}\\
\A&=\{u\in\lambda:c(u)<o_1\text{ or }i_k<c(u)<o_{k+1}\text{ for some }k\in[r]\}\\
\B&=\{u\in\lambda:o_k<c(u)<i_k\text{ for some }k\in[r]\text{ or }o_{r+1}<c(u)\}
\end{split}
\end{eq*}
See \reff{innerouter}.
The motivation for these definitions is as follows.
Let $h$ be a rim-hook of $\lambda$.
Then $\alpha(h)\in\A\cup\O$ and $\omega(h)\in\O\cup\B$.
More generally if $u$ is a cell of $h$ and $u\in\I\cup\A$ then $\e u\in h$.
Similarly, if $u\in h$ and $u\in\B\cup\I$ then also $\s u\in h$.
We will consider paths that are similar to rim-hooks in the sense that they satisfy similar properties.

\smallskip
Equip the cells of $\lambda$ with two total orders: the \emph{reverse lexicographic order} $(\lambda,\leq)$ and the \emph{content order} $(\lambda,\tleq)$.
Let $u,v\in\lambda$, $u=(i,j)$ and $v=(k,\ell)$.
Then $u\leq v$ if and only if either $j>\ell$ or $j=\ell$ and $i\geq k$.
Moreover $u\tleq v$ if and only if either $c(u)>c(v)$ or $c(u)=c(v)$ and $i\geq k$.
Both orders are indicated for the partition $\lambda=(4,3,1)$ in \reff{orders}.

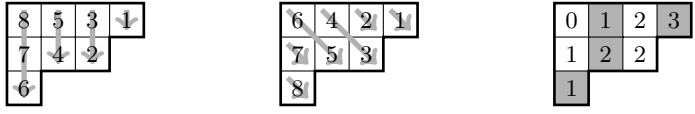
\myfig{The reverse lexicographic order (left) and the content order (middle) for the cells of the partition $\lambda=(4,3,1)$.
The set of candidates $\cand(\pi)$ (right) of a reverse plane partition of shape $\lambda$.}
{orders}
{
\begin{tikzpicture}
\begin{scope}[xshift=0cm]
\draw(0,1)--(1,1)(0,2)--(3,2)(1,1)--(1,3)(2,1)--(2,3)(3,2)--(3,3)
;
\draw[black!30,line width=2pt,->](.5,2.8)--(.5,.2);
\draw[black!30,line width=2pt,->](1.5,2.8)--(1.5,1.2);
\draw[black!30,line width=2pt,->](2.5,2.8)--(2.5,1.2);
\draw[black!30,line width=2pt,->](3.5,2.8)--(3.5,2.2)
;
\draw[line width=1pt](0,0)--(1,0)--(1,1)--(3,1)--(3,2)--(4,2)--(4,3)--(0,3)--cycle
;
\draw[xshift=5mm,yshift=5mm]
(0,2)node{\small{$8$}}(1,2)node{\small{$5$}}(2,2)node{\small{$3$}}(3,2)node{\small{$1$}}(0,1)node{\small{$7$}}(1,1)node{\small{$4$}}(2,1)node{\small{$2$}}(0,0)node{\small{$6$}}
;
\end{scope}
\begin{scope}[xshift=8cm]
\draw(0,1)--(1,1)(0,2)--(3,2)(1,1)--(1,3)(2,1)--(2,3)(3,2)--(3,3)
;
\draw[black!30,line width=2pt,->](.2,.8)--(.8,.2);
\draw[black!30,line width=2pt,->](.2,1.8)--(.8,1.2);
\draw[black!30,line width=2pt,->](.2,2.8)--(1.8,1.2);
\draw[black!30,line width=2pt,->](1.2,2.8)--(2.8,1.2);
\draw[black!30,line width=2pt,->](2.2,2.8)--(2.8,2.2);
\draw[black!30,line width=2pt,->](3.2,2.8)--(3.8,2.2)
;
\draw[line width=1pt](0,0)--(1,0)--(1,1)--(3,1)--(3,2)--(4,2)--(4,3)--(0,3)--cycle
;
\draw[xshift=5mm,yshift=5mm]
(0,2)node{\small{$6$}}(1,2)node{\small{$4$}}(2,2)node{\small{$2$}}(3,2)node{\small{$1$}}(0,1)node{\small{$7$}}(1,1)node{\small{$5$}}(2,1)node{\small{$3$}}(0,0)node{\small{$8$}}
;
\end{scope}
\begin{scope}[xshift=16cm]
\fill[black!30]
(0,0)--(1,0)--(1,1)--(0,1)--cycle
(1,1)--(2,1)--(2,3)--(1,3)--cycle
(3,2)--(4,2)--(4,3)--(3,3)--cycle
;
\draw
(0,1)--(1,1)
(0,2)--(3,2)
(1,1)--(1,3)
(2,1)--(2,3)
(3,2)--(3,3)
;
\draw[line width=1pt](0,0)--(1,0)--(1,1)--(3,1)--(3,2)--(4,2)--(4,3)--(0,3)--cycle;
\draw[xshift=5mm,yshift=5mm]
(0,2)node{\small{$0$}}
(1,2)node{\small{$1$}}
(2,2)node{\small{$2$}}
(3,2)node{\small{$3$}}
(0,1)node{\small{$1$}}
(1,1)node{\small{$2$}}
(2,1)node{\small{$2$}}
(0,0)node{\small{$1$}};
\end{scope}
\end{tikzpicture}
}

Define a total order on the rim-hooks of $\lambda$ by letting $h^u\leq h^v$ if and only if $u\leq v$.
Equivalently, given rim-hooks $f$ and $h$ we have $f\leq h$ if and only if either $c(\alpha(f))>c(\alpha(h))$ or $c(\alpha(f))=c(\alpha(h))$ and $c(\omega(f))\leq c(\omega(h))$.
\reff{rimhooks} shows all rim-hooks of the partition $\lambda=(4,3,1)$ in reverse lexicographic order starting with the minimum.

\smallskip
Let $\pi$ be a reverse plane partition of shape $\lambda$.
Define the set of \emph{candidates} of $\pi$ as
\begin{eq*}
\cand(\pi)
=\big\{u\in\O:\pi(u)>\pi(\w u)\big\}
\cup\big\{u\in\A:\pi(u)>\pi(\w u)\text{ and }\pi(u)>\pi(\n u)\big\}.
\end{eq*}
See \reff{orders} for an example.
The motivation for the definition of candidates will become clearer later on.
Candidates are an important tool for the analysis of our insertion algorithm, especially when dealing with the reverse procedure.
For now we prove a simple criterion for the existence of candidates.

\begin{lem}{cand} Let $\lambda$ be a partition and $\pi\in\rpp_{\lambda}$.
Let $u\in\I\cup\A$ and $v\in\A\cup\O$ be cells in the same row such that $\pi(u)<\pi(v)$ and $i_k\leq c(u)<c(v)\leq o_{k+1}$ for some $k\in[r]$ or $c(u)<c(v)\leq o_1$.
Then there exists a candidate $w\in\cand(\pi)$ with $c(w)>c(u)$.
\end{lem}

\begin{proof} First note that we may assume without loss of generality that $v\in\O$.
Choose $x=(x_1,x_2)$ in the same row of $\lambda$ as $u$ with $c(u)<c(x)\leq c(v)$ such that $x_2$ is minimal with $\pi(u)<\pi(x)$.
Moreover choose $y=(y_1,y_2)$ in the same column of $\lambda$ as $x$, that is, $y_2=x_2$, with $c(x)\leq c(y)\leq c(v)$ such that $y_1$ is minimal with $\pi(x)=\pi(y)$.
Then $c(y)>c(u)$ and $y\in\cand(\pi)$.
\end{proof}

Note that the proof of \refl{cand} still works when $v=(i,j)$ is the southernmost outer corner of $\lambda$, that is, $j-i=o_1$, and $u=(i,0)$, even though $u$ does not belong to $\lambda$.

\begin{cor}{existcand}
Let $\lambda$ be a partition and $\pi$ a non-zero reverse plane partition.
Then $\cand(\pi)\neq\emptyset$.
\end{cor}

\begin{proof}
Apply \refl{cand} to the cells $u,v$, where $v=(i,j)$ is the southernmost outer corner of $\lambda$ with $\pi(v)>0$, and $u$ is the easternmost cell in the same row as $v$ with $\pi(v)=0$.
Note that if $j-i=o_{k+1}$ for some $k\in[r]$, then $u\in\I\cup\A$ and $i_k\leq c(u)<o_{k+1}$.
If $j-i=o_1$ then $u\in\A$ or $u=(i,0)$ as in the remark above.
\end{proof}

\section{Inserting rim-hooks}\label{Section:insertion}

This section contains a formal definition of the insertion algorithm described in \refs{bricks}.

\smallskip
The insertion works by increasing a reverse plane partition along a path.
Let $\pi$ be a reverse plane partition of shape $\lambda$ and $P$ be a path in $\lambda$.
Define the maps $\pi\pm P:\lambda\to\Z$ by
\begin{align*}
(\pi\pm P)(u)&=
\begin{cases}
\pi(u)\pm1&\quad\text{if }u\in P,\\
\pi(u)&\quad\text{otherwise.}
\end{cases}
\end{align*}
We call the pair $(P,\pi)$ \emph{compatible} if the following two conditions are fulfilled.
\begin{align}\label{Equation:shh}
\text{If }u\in P\text{ and }u\in\I\cup\A&\text{ then }\e u\in P\text{ and }\pi(u)=\pi(\e u).\\
\label{Equation:shp}
\text{If }u,\s u\in P&\text{ then }\pi(u)=\pi(\s u).
\end{align}
We say a rim-hook $h$ of $\lambda$ \emph{inserts} into $\pi$ if there exists a path $P$ such that $\omega(P)=\omega(h)$, $\ell(P)=\ell(h)$, $(P,\pi)$ is compatible and $\pi+P$ is a reverse plane partition.

\smallskip
Our first goal is to show that such a path is unique if it exists at all.
Given a rim-hook $h$ of $\lambda$ construct a south-west-path $P(h,\pi)$ as follows. Let $\omega(P(h,\pi))=\omega(h)$ and while $\ell(P(h,\pi))<\ell(h)$ if $u$ is the current cell then move to
\begin{align*}
\text{the cell}\quad
\begin{cases}
\s u &\quad\text{if }u\in{\B\cup\I}\text{ and }\pi(u)=\pi(\s u),\\
\w u &\quad\text{otherwise.}
\end{cases}
\end{align*}

\begin{lem}{insertpath} Let $\lambda$ be a partition, $\pi\in\rpp_{\lambda}$, $h$ a rim-hook of $\lambda$ inserting into $\pi$ and let $P$ be a south-west-path in $\lambda$ such that $\omega(P)=\omega(h)$, $\ell(P)=\ell(h)$, $(P,\pi)$ is compatible and $\pi+P$ is a reverse plane partition.
Then $P=P(h,\pi)$.
\end{lem}

\begin{proof} Suppose $P$ is a south-west-path in $\lambda$ such that $\omega(P)=\omega(h)$, $\ell(P)=\ell(h)$, $(P,\pi)$ is compatible and $\pi+P$ is a reverse plane partition.
Suppose $u\in P$ with $u\triangleleft\alpha(P)$.
If $u\in\B\cup\I$ such that $\pi(u)=\pi(\s u)$ then $\s u\in P$ because $\pi+P$ is a reverse plane partition.
Otherwise $\w u\in P$ because $(P,\pi)$ is compatible.
Hence $P$ agrees with the construction of $P(h,\pi)$ and the claim follows.
\end{proof}

As a consequence of \refl{insertpath} if a rim-hook $h$ inserts into a reverse plane partition $\pi$ then $P(h,\pi)$ is the unique south-west-path $P$ such that $\omega(P)=\omega(h)$, $\ell(P)=\ell(h)$, $(P,\pi)$ is compatible and $\pi+P$ is a reverse plane partition.
In this case denote $h*\pi=\pi+P(h,\pi)$.

It is not difficult to verify that our definition of $h*\pi$ agrees with the examples of \refs{bricks}.
For example \reff{swpaths} shows the insertion paths corresponding to \reff{cut1}.

\myfig{The paths $P(h^u,\pi)$ for the cells $u=(1,3)$ (left) and $u=(2,2)$ (right).}
{swpaths}
{
\begin{tikzpicture}
\begin{scope}
\fill[black!30](1,1)--(3,1)--(3,3)--(2,3)--(2,2)--(1,2)--cycle
;
\draw(0,0)grid(3,3)
;
\draw[line width=1pt](0,0)--(0,3)--(3,3)--(3,0)--cycle
;
\draw[xshift=5mm,yshift=5mm]
(0,2)node{\small{$0$}}(1,2)node{\small{$0$}}(2,2)node{\small{$0$}}(0,1)node{\small{$0$}}(1,1)node{\small{$0$}}(2,1)node{\small{$0$}}(0,0)node{\small{$1$}}(1,0)node{\small{$1$}}(2,0)node{\small{$1$}}
;
\end{scope}
\begin{scope}[xshift=6cm]
\fill[black!30](0,1)--(3,1)--(3,2)--(0,2)--cycle
;
\draw(0,0)grid(3,3)
;
\draw[line width=1pt](0,0)--(0,3)--(3,3)--(3,0)--cycle
;
\draw[xshift=5mm,yshift=5mm]
(0,2)node{\small{$0$}}(1,2)node{\small{$0$}}(2,2)node{\small{$0$}}(0,1)node{\small{$0$}}(1,1)node{\small{$0$}}(2,1)node{\small{$0$}}(0,0)node{\small{$1$}}(1,0)node{\small{$1$}}(2,0)node{\small{$1$}}
;
\end{scope}
\end{tikzpicture}
}
\smallskip
We finish this section by proving a necessary condition for when the insertion algorithm fails.
At the same time this result is a first indication of the importance of candidates.

\begin{thm}{noinsert} Let $\lambda$ be a partition, $\pi\in\rpp_{\lambda}$ and $h$ a rim-hook of $\lambda$ that does not insert into $\pi$.
Then there exists a candidate $u\in\cand(\pi)$ with $u\triangleleft\alpha(P(h,\pi))$.
\end{thm}

\begin{proof} By construction $\omega(P(h,\pi))=\omega(h)$, $\ell(P(h,\pi))=\ell(h)$ and $P(h,\pi)$ satisfies \refq{shp}.
Hence if $h$ does not insert into $\pi$ then $P(h,\pi)$ does not satisfy \refq{shh} or $\pi+P(h,\pi)$ is not a reverse plane partition.

If $P(h,\pi)$ does not satisfy \refq{shh} then there exists $u\in P(h,\pi)$ with $u\in\I\cup\A$ and $\pi(u)<\pi(\e u)$.
Hence \refl{cand} yields the existence of a $v\in\cand(\pi)$ with $c(\alpha(P(h,\pi)))\leq c(u)<c(v)$.

On the other hand assume that $P(h,\pi)$ satisfies \refq{shh} and that $\pi+P(h,\pi)$ is not a reverse plane partition.
Then there exists $u\in P(h,\pi)$ such that
\begin{align}
\label{Equation:eu}
\pi(u)&=\pi(\e u)\text{ and }\e u\notin P(h,\pi)\text{ or }\\
\label{Equation:su}
\pi(u)&=\pi(\s u)\text{ and }\s u\notin P(h,\pi).
\end{align}
Let $u\in P(h,\pi)$ the minimal cell with respect to the content order that satisfies \refq{eu} or \refq{su}.

If $\e u\notin P(h,\pi)$ and $\pi(u)=\pi(\e u)$ then $\n u\in P(h,\pi)$ and thus $\n u\in\B\cup\I$ with $\pi(u)=\pi(\n u)$.
Consequently $\pi(\n u)=\pi(\e\n u)=\pi(\e u)$.
But now either $\e\n u\in P(h,\pi)$ satisfies \refq{su} or $\e\n u\notin P(h,\pi)$ and $\n u$ satisfies \refq{eu} contradicting the minimality of $u$.

If $\s u\notin P(h,\pi)$ and $\pi(u)=\pi(\s u)$ then $u\in\A\cup\O$ and $\e u,\e\s u\in\lambda$.
If $u\in\O$, $\e u\notin P(h,\pi)$ and $\pi(\e u)=\pi(u)$ then we are in the case treated above.
Thus we may assume $\pi(\e u)>\pi(u)$, but then $\e\s u\in\cand(\pi)$.
On the other hand if $\e u\in P(h,\pi)$ then we may assume $\pi(\e\s u)>\pi(\e u)$ by minimality of $u$ and again $\e\s u\in\cand(\pi)$.
\end{proof}

\section{Factors}\label{Section:factors}

In this section we address the question how and when rim-hook insertion can be inverted.

\smallskip
Given a reverse plane partition $\pi$ of shape $\lambda$ and a rim-hook of $\lambda$ it is natural to ask if there exists a reverse plane partition $\tilde{\pi}$ such that $\pi=h*\tilde{\pi}$.
The main reason why extracting rim-hooks is a non-trivial task is the fact that the reverse plane partition $\tilde{\pi}$ is in general not unique.
See \reff{extract}.
Candidates play in important role in resolving this ambiguity.

\myfig{}
{extract}
{
\begin{tikzpicture}[tdplot_main_coords]
\begin{scope}[xshift=0cm,yshift=-2cm]
\draw[line width=2pt,->](.5,.5,1.5)--(.5,.5,0);
\end{scope}
\begin{scope}[xshift=0cm,yshift=0cm]
\clip(2,0,0)--(2,1,0)--(0,1,0)--(0,1,1)--(0,0,1)--(2,0,1)--cycle;
\draw[line width=1pt,fill=cyan]
(0,0,1)--(1,0,1)--(1,1,1)--(0,1,1)--cycle
(1,0,1)--(2,0,1)--(2,1,1)--(1,1,1)--cycle
;
\draw[line width=1pt,fill=black!20]
(2,0,1)--(2,0,0)--(2,1,0)--(2,1,1)--cycle
;
\draw[line width=1pt,fill=black!50]
(0,1,1)--(1,1,1)--(1,1,0)--(0,1,0)--cycle
(1,1,1)--(2,1,1)--(2,1,0)--(1,1,0)--cycle
;
\draw[line width=2pt](2,0,0)--(2,1,0)--(0,1,0)--(0,1,1)--(0,0,1)--(2,0,1)--cycle;
\end{scope}
\begin{scope}[xshift=0cm,yshift=-44mm]
\clip(2,0,0)--(2,2,0)--(0,2,0)--(0,2,1)--(0,1,1)--(0,1,1)--(0,0,1)--(2,0,1)--cycle;
\draw[line width=1pt,fill=black!5]
(0,0,1)--(1,0,1)--(1,1,1)--(0,1,1)--cycle
(0,1,1)--(1,1,1)--(1,2,1)--(0,2,1)--cycle
(1,0,1)--(2,0,1)--(2,1,1)--(1,1,1)--cycle
(1,1,1)--(2,1,1)--(2,2,1)--(1,2,1)--cycle
;
\draw[line width=1pt,fill=black!20]
(2,0,1)--(2,0,0)--(2,1,0)--(2,1,1)--cycle
(2,1,1)--(2,1,0)--(2,2,0)--(2,2,1)--cycle
;
\draw[line width=1pt,fill=black!50]
(0,2,1)--(1,2,1)--(1,2,0)--(0,2,0)--cycle
(1,2,1)--(2,2,1)--(2,2,0)--(1,2,0)--cycle
;
\draw[line width=2pt](2,0,0)--(2,2,0)--(0,2,0)--(0,2,1)--(0,1,1)--(0,1,1)--(0,0,1)--(2,0,1)--cycle;
\end{scope}
\begin{scope}[xshift=4cm,yshift=-44mm]
\clip(2,0,0)--(2,2,0)--(0,2,0)--(0,2,1)--(0,1,1)--(0,1,2)--(0,0,2)--(2,0,2)--cycle;
\draw[line width=1pt,fill=black!5]
(0,1,1)--(1,1,1)--(1,2,1)--(0,2,1)--cycle
(1,1,1)--(2,1,1)--(2,2,1)--(1,2,1)--cycle
;
\draw[line width=1pt,fill=cyan]
(0,0,2)--(1,0,2)--(1,1,2)--(0,1,2)--cycle
(1,0,2)--(2,0,2)--(2,1,2)--(1,1,2)--cycle
;
\draw[line width=1pt,fill=black!20]
(2,0,2)--(2,0,1)--(2,1,1)--(2,1,2)--cycle
(2,0,1)--(2,0,0)--(2,1,0)--(2,1,1)--cycle
(2,1,1)--(2,1,0)--(2,2,0)--(2,2,1)--cycle
;
\draw[line width=1pt,fill=black!50]
(0,1,2)--(1,1,2)--(1,1,1)--(0,1,1)--cycle
(0,2,1)--(1,2,1)--(1,2,0)--(0,2,0)--cycle
(1,1,2)--(2,1,2)--(2,1,1)--(1,1,1)--cycle
(1,2,1)--(2,2,1)--(2,2,0)--(1,2,0)--cycle
;
\draw[line width=2pt](2,0,0)--(2,2,0)--(0,2,0)--(0,2,1)--(0,1,1)--(0,1,2)--(0,0,2)--(2,0,2)--cycle;
\end{scope}
\begin{scope}[xshift=15cm]
\begin{scope}[xshift=0cm,yshift=-2cm]
\draw[line width=2pt,->](.5,.5,1.5)--(.5,.5,0);
\end{scope}
\begin{scope}[xshift=0cm,yshift=0cm]
\clip(2,0,0)--(2,1,0)--(0,1,0)--(0,1,1)--(0,0,1)--(2,0,1)--cycle;
\draw[line width=1pt,fill=red]
(0,0,1)--(1,0,1)--(1,1,1)--(0,1,1)--cycle
;
\draw[line width=1pt,fill=yellow]
(1,0,1)--(2,0,1)--(2,1,1)--(1,1,1)--cycle
;
\draw[line width=1pt,fill=black!20]
(2,0,1)--(2,0,0)--(2,1,0)--(2,1,1)--cycle
;
\draw[line width=1pt,fill=black!50]
(0,1,1)--(1,1,1)--(1,1,0)--(0,1,0)--cycle
(1,1,1)--(2,1,1)--(2,1,0)--(1,1,0)--cycle
;
\draw[line width=2pt](2,0,0)--(2,1,0)--(0,1,0)--(0,1,1)--(0,0,1)--(2,0,1)--cycle;
\end{scope}
\begin{scope}[xshift=0cm,yshift=-44mm]
\clip(2,0,0)--(2,2,0)--(0,2,0)--(0,2,1)--(0,1,1)--(0,1,2)--(0,0,2)--(1,0,2)--(1,0,1)--(2,0,1)--cycle;
\draw[line width=1pt,fill=black!5]
(0,0,2)--(1,0,2)--(1,1,2)--(0,1,2)--cycle
(0,1,1)--(1,1,1)--(1,2,1)--(0,2,1)--cycle
(1,0,1)--(2,0,1)--(2,1,1)--(1,1,1)--cycle
(1,1,0)--(2,1,0)--(2,2,0)--(1,2,0)--cycle
;
\draw[line width=1pt,fill=black!20]
(1,0,2)--(1,0,1)--(1,1,1)--(1,1,2)--cycle
(1,1,1)--(1,1,0)--(1,2,0)--(1,2,1)--cycle
(2,0,1)--(2,0,0)--(2,1,0)--(2,1,1)--cycle
;
\draw[line width=1pt,fill=black!50]
(0,1,2)--(1,1,2)--(1,1,1)--(0,1,1)--cycle
(0,2,1)--(1,2,1)--(1,2,0)--(0,2,0)--cycle
(1,1,1)--(2,1,1)--(2,1,0)--(1,1,0)--cycle
;
\draw[line width=2pt](2,0,0)--(2,2,0)--(0,2,0)--(0,2,1)--(0,1,1)--(0,1,2)--(0,0,2)--(1,0,2)--(1,0,1)--(2,0,1)--cycle;
\end{scope}
\begin{scope}[xshift=4cm,yshift=-44mm]
\clip(2,0,0)--(2,2,0)--(0,2,0)--(0,2,1)--(0,1,1)--(0,1,2)--(0,0,2)--(2,0,2)--cycle;
\draw[line width=1pt,fill=black!5]
(0,0,2)--(1,0,2)--(1,1,2)--(0,1,2)--cycle
(0,1,1)--(1,1,1)--(1,2,1)--(0,2,1)--cycle
;
\draw[line width=1pt,fill=red]
(1,1,1)--(2,1,1)--(2,2,1)--(1,2,1)--cycle
;
\draw[line width=1pt,fill=yellow]
(1,0,2)--(2,0,2)--(2,1,2)--(1,1,2)--cycle
;
\draw[line width=1pt,fill=black!20]
(2,0,2)--(2,0,1)--(2,1,1)--(2,1,2)--cycle
(2,0,1)--(2,0,0)--(2,1,0)--(2,1,1)--cycle
(2,1,1)--(2,1,0)--(2,2,0)--(2,2,1)--cycle
;
\draw[line width=1pt,fill=black!50]
(0,1,2)--(1,1,2)--(1,1,1)--(0,1,1)--cycle
(0,2,1)--(1,2,1)--(1,2,0)--(0,2,0)--cycle
(1,1,2)--(2,1,2)--(2,1,1)--(1,1,1)--cycle
(1,2,1)--(2,2,1)--(2,2,0)--(1,2,0)--cycle
;
\draw[line width=2pt](2,0,0)--(2,2,0)--(0,2,0)--(0,2,1)--(0,1,1)--(0,1,2)--(0,0,2)--(2,0,2)--cycle;
\end{scope}
\end{scope}
\end{tikzpicture}
}

\smallskip
Let us formulate in terms of reducing $\pi$ along a path in $\lambda$.
A rim-hook $h$ of $\lambda$ is a \emph{factor} of $\pi$ if there exists a south-west-path $P$ such that $\omega(P)=\omega(h)$, $\ell(P)=\ell(h)$, $(P,\pi)$ is compatible and $\pi-P$ is again a reverse plane partition.
Denote the set of all factors of $\pi$ by $\F(\pi)$.

Given a candidate $v\in\cand(\pi)$ construct the north-east-path $Q(v,\pi)$ in $\lambda$ as follows: Let $\alpha(P)=v$ and, if $u$ is the current cell, move to
\begin{align*}
\text{the cell}\quad
\begin{cases}
\n u&\quad\text{if }u\in\O\cup\B\text{ and }\pi(u)=\pi(\n u),\\
\e u&\quad\text{if }u\in\I\cup\A\text{ or }\e u\in\lambda,\pi(u)>\pi(\n u),
\end{cases}
\end{align*}
and terminate the path if neither a north step nor an east step are possible according to these rules, that is, if $\pi(u)>\pi(\n u)$ and $\e u\notin\lambda$.
Denote by $h(v,\pi)$ the rim-hook of $\lambda$ defined by $\omega(h(v,\pi))=\omega(Q(v,\pi))$ and $\ell(h(v,\pi))=\ell(Q(v,\pi))$.

\begin{lem}{factorpath} Let $\lambda$ be a partition, $\pi\in\rpp_{\lambda}$, $h\in\F(\pi)$ be a factor of $\pi$, $P$ be a south-west-path in $\lambda$ such that $\omega(P)=\omega(h)$, $\ell(P)=\ell(h)$, $(P,\pi)$ is compatible and $\pi-P$ is a reverse plane partition, and let $v=\alpha(P)$.
Then $v\in\cand(\pi)$ and $P$ is the reverse path of $Q(v,\pi)$.
\end{lem}

\begin{proof} First note that $v\in\A\cup\O$ because $\omega(P)=\omega(h)$, $\ell(P)=\ell(h)$ and $\alpha(h)\in\A\cup\O$.
Hence $v\in\cand(\pi)$ because of \refq{shh} and since $\pi-P$ is a reverse plane partition.
Now consider a cell $u\in P$.
If $\pi(u)=\pi(\n u)$ then $\n u\in P'$ as $\pi-P$ is a reverse plane partition.
If $u\in\I\cup\A$ or if $\e u\in\lambda$ and $\pi(u)>\pi(\n u)$ then $\e u\in P$ because $(P,\pi)$ is compatible.
Thus $P$ agrees with the construction of $Q(v,\pi)$.
\end{proof}

\begin{lem}{factoriff} Let $\lambda$ be a partition, $\pi\in\rpp_{\lambda}$ and $h$ a rim-hook of $\lambda$.
Then $h$ is a factor of $\pi$ if and only if there exists $\tilde\pi\in\rpp_{\lambda}$ such that $h$ inserts into $\tilde\pi$ and $h*\tilde{\pi}=\pi$.
\end{lem}

\begin{proof} Suppose there is a reverse plane partition $\tilde{\pi}$ such that $h$ inserts into $\tilde{\pi}$ and $h*\tilde{\pi}=\pi$.
Then $v=\alpha(P(h,\tilde{\pi}))\in\cand(\pi)$ as in the proof of \refl{factorpath} above.
Moreover $P(h,\tilde{\pi})=Q(v,\pi)$ and $h=h(v,\pi)$.

Conversely suppose that $h\in\F(\pi)$.
By \refl{factorpath} there exists $v\in\cand(\pi)$ such that $h=h(v,\pi)$, $(Q(v,\pi),\pi)$ is compatible and $\pi-Q(v,\pi)$ is a reverse plane partition.
We may set $\tilde{\pi}=\pi-Q(v,\pi)$.
\end{proof}

Note that by Lemmas~\ref{Lemma:factorpath} and~\ref{Lemma:factoriff} the reverse plane partition $\tilde{\pi}$ is unique as long as we fix the candidate $v=\alpha(P(h,\tilde{\pi}))=\alpha(Q(v,\pi))$.
Moreover, we have reduced the task of finding factors of $\pi$ to finding suitable candidates.

\smallskip
The following theorem guarantees the existence of factors.

\begin{thm}{existsfactor} Let $\lambda$ be a partition, $\pi\in\rpp_{\lambda}$ and $u\in\cand(\pi)$ a candidate.
Then there exists a factor $h\in\F(\pi)$ such that $h\leq h(u,\pi)$.
\end{thm}

\begin{proof} Without loss of generality we may assume that $u$ is the minimum of $\cand(\pi)$ with respect to the content order $\tleq$.
Let $P$ be the reverse of $Q(u,\pi)$ and $h=h(u,\pi)$.
Then by \refl{factorpath} $P$ is a south-west-path with $\omega(P)=\omega(h)$ and $\ell(P)=\ell(h)$ satisfying \refq{shp}.
Let $v\in\I\cup\A$ be a cell with $v\tleq u$.
By \refl{cand} and by choice of $u$ we have $\pi(v)=\pi(\e v)$.
Thus $P$ satisfies \refq{shh}.

Now suppose that $\pi-P$ is not a reverse plane partition and choose $v\in P$ maximal with respect to the content order such that $\n v\notin P$ and $\pi(\n v)=\pi(v)$ or $\w v\notin P$ and $\pi(\w v)=\pi(v)$.
If $\n v\notin P$ and $\pi(\n v)=\pi(v)$ then $v\in\I\cup\A$.
By \refl{cand} we have $\pi(\n\w v)=\pi(v)$.
Hence $\pi(\w v)=\pi(v)$ and $v\notin\cand(\pi)$.
Since $v\notin\cand(\pi)$ either $\s v\in P$ or $\w v\in P$.
But $\w v\in P$ contradicts the maximality of $v$.

We may therefore assume that $\w v\notin P$ and $\pi(\w v)=\pi(v)$ and $\s v\in P$. Then $\pi(\s v)=\pi(v)$ by construction of $P$.
Consequently $\pi(\s\w v)=\pi(v)$ and $\s v\notin\cand(\pi)$.
By maximality of $v$ we must have $\s\w v\in P$.
But then $\pi(\s\w v)=\pi(\w v)$ yields a contradiction to the maximality of $v$.

We conclude that $\pi-P$ is a reverse plane partition and $h\in\F(\pi)$.
\end{proof}

In particular the proof of \reft{existsfactor} shows that if $u$ is the minimum of $\cand(\pi)$ with respect to the content order then $h(u,\pi)$ is a factor of $\pi$.
We conclude a first factorisation theorem.

\begin{thm}{factorisation} Let $\lambda$ be a partition and $\pi\in\rpp_{\lambda}$.
Then there exists a sequence of rim-hooks $h_1,h_2,\dots,h_s$ of $\lambda$ such that $h_i$ inserts into the reverse plane partition $h_{i-1}*\dots*h_1*0$ for all $i\in[s]$ and $\pi=h_s*\dots*h_1*0$, where $0$ denotes the zero reverse plane partition.
\end{thm}

\begin{proof} The claim follows by induction on the size of $\pi$ since every non-zero reverse plane partition has a candidate by \refc{existcand} and thus also a factor due to \reft{existsfactor}.
\end{proof}

To conclude this section we prove two results on the behaviour of the set of candidates when a rim-hook is inserted or extracted.

\begin{lem}{candidateQ} Let $\lambda$ be a partition, $\pi\in\rpp_{\lambda}$ and $u,v\in\cand(\pi)$ be distinct candidates such that $(Q(v,\pi),\pi)$ is compatible and $\pi-Q(v,\pi)$ is a reverse plane partition.
Then $u\in\cand(\pi-Q(v,\pi))$.
\end{lem}

\begin{proof} The claim is trivially true if $v\triangleleft u$.
Thus suppose $u\triangleleft v$ and $u\notin\cand(\pi-Q(v,\pi))$.
It follows that $u\in Q(v,\pi)$.
By construction of $Q(v,\pi)$ and since $u\triangleleft\alpha(Q(v,\pi))$ we have $\w u\in Q(v,\pi)$.
But then $\pi(\w u)=\pi(u)$ by \refq{shh} contradicting $u\in\cand(\pi)$.
\end{proof}

\begin{lem}{candidateP} Let $\lambda$ be a partition, $\pi\in\rpp_{\lambda}$, $u\in\lambda-\cand(\pi)$ and $v\in\cand(\pi)$ such that $(Q(v,\pi),\pi)$ is compatible and $\pi-Q(v,\pi)$ is a reverse plane partition.
If $u\triangleleft v$ then $u\notin\cand(\pi-Q(v,\pi))$.
\end{lem}

\begin{proof} Suppose $u\in\cand(\pi-Q(v,\pi))$ but $u\notin\cand(\pi)$.
First note that $u\in\A\cup\O$.
As $u\in\cand(\pi-Q(v,\pi))$ and since $\w u\in Q(v,\pi)$ implies $u\in Q(v,\pi)$ we must have $\pi(\w u)<\pi(u)$.
Thus $u\in A$ and $\pi(u)=\pi(\n u)$ because $u\notin\cand(\pi)$.
Moreover $\n u\in\A\cup\O$ and $\n u\in Q(v,\pi)$.
Since $\n u\triangleleft v$ it follows that $\n\w u\in Q(v,\pi)$ but then $\pi(\n\w u)\leq\pi(\w u)<\pi(\n u)$ contradicts condition \refq{shh} for $Q(v,\pi)$.
\end{proof}

\section{A bijection}\label{Section:bijection}

In this section rim-hook insertion is used to obtain a bijection between the set of reverse plane partitions $\rpp_{\lambda}$ and multi-sets of rim-hooks of $\lambda$.

\smallskip
The following two results essentially state that certain paths used in the definitions of insertion and extraction of rim-hooks cannot cross. 

\begin{lem}{crossing1} Let $\lambda$ be a partition, $\pi\in\rpp_{\lambda}$, $h$ a rim-hook of $\lambda$ and $u\in\cand(\pi)$.
If $u\triangleleft\alpha(P(h,\pi))$ then $h(u,\pi)<h$.
\end{lem}

\begin{proof}
The claim is trivial if $c(u)>c(\alpha(P(h,\pi)))$.
Thus assume that $c(u)=c(\alpha(P(h,\pi)))$.
It suffices to show that $\omega(P(h,\pi))\triangleleft\omega(Q(u,\pi))$.
If this is not the case then there exists a cell $v\in\lambda$ such that $v,\w v\in P(h,\pi)$ and $v,\s v\in Q(u,\pi)$.
It follows from the construction of $Q(u,\pi)$ that $v\in\B\cup\I$ and $\pi(v)=\pi(\s v)$, which contradicts $\w v\in P(h,\pi)$.
\end{proof}

\begin{lem}{crossing2} Let $\lambda$ be a partition, $\pi\in\rpp_{\lambda}$, $h$ a rim-hook of $\lambda$ and $u\in\cand(\pi)$.
If $\alpha(P(h,\pi))\tleq u$ then $h\leq h(u,\pi)$.
\end{lem}

\begin{proof}
The claim is trivial if $c(u)<c(\alpha(P(h,\pi)))$
Thus assume that $c(u)=c(\alpha(P(h,\pi)))$.
It suffices to show that $\omega(Q(u,\pi))\tleq\omega(P(h,\pi))$.
If this is not the case then there exists a cell $v\in\lambda$ such that $v,\n v\in P(h,\pi)$ and $v\in Q(u,\pi)$ and $\n v\notin Q(u,\pi)$.
From the construction of $P(h,\pi)$ it follows that $v\in\O\cup\B$ and $\pi(v)=\pi(\n v)$.
But this contradicts $\n v\notin Q(u,\pi)$.
\end{proof}

The following theorem guarantees that the successive insertion of rim-hooks into the zero reverse plane partition never fails as long as we respect the reverse lexicographic order on rim-hooks.

\begin{thm}{welldef} Let $\lambda$ be a partition and $h_1,h_2,\dots,h_s$ be rim-hooks of $\lambda$ such that $h_i\leq h_{i+1}$ for all $i\in[s-1]$.
Then $h_{i}$ inserts into $h_{i+1}*\dots*h_s*0$ and the minimum of $\cand(h_i*\dots*h_s*0)$ with respect to the content order is $\alpha(P(h_i,h_{i+1}*\dots*h_s*0))$ for all $i\in[s]$.
\end{thm}

\begin{proof} By induction on $s$ we may assume that $h_{i}$ inserts into $h_{i+1}*\dots*h_s*0$ for all $i>1$ and that the minimum of $\cand(\pi)$ is given by $\alpha(Q)$ where $\pi=h_2*\dots*h_s*0$ and $Q=P(h_2,h_3*\dots*h_s*0)$.

If $h_1$ does not insert into $\pi$ then by \reft{noinsert} there exists a candidate $u\in\cand(\pi)$ with $u\triangleleft\alpha(P(h_1,\pi))$.
Without loss of generality assume that $u=\alpha(Q)$.
By \refl{factorpath} we obtain $Q=Q(u,\pi)$ and hence $h(u,\pi)=h_2$.
However, \refl{crossing1} implies $h_1>h(u,\pi)=h_2$, which is a contradiction.

Thus $h_1$ inserts into $\pi$.
To see the second claim, assume that there is a $v\in\cand(h_1*\pi)$ with $v\triangleleft\alpha(P(h,\pi))$.
Then $v\in\cand(\pi)$ due to \refl{candidateQ}, we can assume that $v=\alpha(Q)$ and deduce a contradiction as above.
\end{proof}

%


Let $h_1,h_2,\dots,h_s$ be a sequence of rim-hooks in $\lambda$ such that $h_i\leq h_{i+1}$ for all $i\in[s-1]$.
Then we call $\pi=h_1*h_2*\dots*h_s*0$ a \emph{lexicographic factorisation} of $\pi$.
The next theorem implies that every reverse plane partition possesses a lexicographic factorisation.

\begin{thm}{surj} Let $\lambda$ be a partition and $\pi\in\rpp_{\lambda}$.
Define a sequence of rim-hooks $h_1,h_2,\dots,h_s$ by letting $h_i=h(v_i,\pi_i)$ for all $i\in[s]$, where $v_i$ is the minimum of $\cand(\pi_i)$ with respect to the content order, $\pi_1=\pi$, $\pi_{i+1}=\pi_i-Q(v_i,\pi_i)$, and $\pi_{s+1}$ is the first reverse plane partition in this sequence with only zero entries.
Then $h_i\leq h_{i+1}$ for all $i\in[s-1]$ and $\pi=h_1*h_2*\dots*h_s*0$.
\end{thm}

\begin{proof} Let $u$ be the minimum of $\cand(\pi)$.
Note that $h(u,\pi)\in\F(\pi)$ due to \reft{existsfactor}.
By induction on the size of $\pi$ we may assume that $h_2*\dots*h_s*0$ is a lexicographic factorisation of $\pi_2$.
Let $v$ be the minimum of $\cand(\pi_2)$.
By \refl{candidateP} we have $u\tleq v$.
Using \refl{crossing2} and $Q(u,\pi)=P(h_1,\pi_2)$ we obtain $h_1\leq h(v,\pi_2)=h_2$.
%
\end{proof}

Furthermore lexicographic factorisations are unique.

\begin{thm}{inj} Let $\lambda$ be a partition, $\pi\in\rpp_{\lambda}$ and $\pi=h_1*h_2*\dots*h_s*0$ be a lexicographic factorisation of $\pi$.
Then $h_2*\dots*h_s*0=\pi-Q(v,\pi)$ where $v$ is the minimum of $\cand(\pi)$ with respect to the content order.
\end{thm}

\begin{proof} Set $\tilde{\pi}=h_2*\dots*h_s*0$ and assume that there exist candidates $v,w\in\cand(\pi)$ with $v\triangleleft w$ and $\tilde{\pi}=\pi-Q(w,\pi)$.
By \refl{candidateQ} we have $v\in\cand(\tilde{\pi})$.
Clearly $\tilde{\pi}=h_2*\dots*h_s*0$ is a lexicographic factorisation of $\tilde{\pi}$.
By induction on the size of $\pi$ we may assume that $h(u,\tilde{\pi})=h_2$ where $u\triangleleft w$.
Using \refl{crossing1} and $Q(w,\pi)=P(h_1,\tilde{\pi})$ we see that $h_2<h_1$, which is a contradiction.
%
\end{proof}

The previous results of this section are collected in main theorem of this article. 
Note that this is an equivalent formulation of \reft{main}.

\begin{thm}{bij} Let $\lambda$ be a partition. Then the map sending each reverse plane partition of shape $\lambda$ to its lexicographic factorisation is a bijection between $\rpp_{\lambda}$ and the multi-sets of rim-hooks of $\lambda$.
\end{thm}

\begin{proof}
Consider the map sending a multi-set of rim-hooks $h_1,h_2,\dots,h_s$ to the reverse plane partition $h_{\sigma(1)}*h_{\sigma(2)}*\cdots*h_{\sigma(s)}*0$, where $\sigma\in\S_s$ rearranges the rim-hooks in lexicographic order.
This map is well-defined by \reft{welldef}, surjective by \reft{surj} and injective by \reft{inj}.
Moreover it is clearly the inverse of the map described in the theorem.
\end{proof}

Figures~\ref{Figure:fac} and~\ref{Figure:lex} demonstrate that a reverse plane partition may have multiple factorisations.
However, only in \reff{lex} are the rim-hooks inserted in reverse lexicographic order. 

\myfig{}{fac}{
\begin{tikzpicture}[tdplot_main_coords]
\begin{scope}[xshift=0cm,yshift=-2cm]
\draw[line width=2pt,->](1.5,1.5,1.5)--(1.5,1.5,0);
\begin{scope}[xshift=6cm,yshift=0cm]
\draw[line width=2pt,->](2.5,.5,1.5)--(2.5,.5,0);
\end{scope}
\begin{scope}[xshift=12cm,yshift=0cm]
\draw[line width=2pt,->](2.5,.5,1.5)--(2.5,.5,0);
\end{scope}
\begin{scope}[xshift=18cm,yshift=0cm]
\draw[line width=2pt,->](2.5,.5,1.5)--(2.5,.5,0);
\end{scope}
\end{scope}
\begin{scope}[xshift=0cm,yshift=0cm]
\clip(2,1,0)--(2,2,0)--(2,3,0)--(2,4,0)--(1,4,0)--(0,4,0)--(0,4,1)--(0,3,1)--(1,3,1)--(1,2,1)--(1,1,1)--(2,1,1)--cycle;
\draw[line width=1pt,fill=black!5]
(1,3,1)--(2,3,1)--(2,4,1)--(1,4,1)--cycle
(1,2,1)--(2,2,1)--(2,3,1)--(1,3,1)--cycle
(1,1,1)--(2,1,1)--(2,2,1)--(1,2,1)--cycle
(0,3,1)--(1,3,1)--(1,4,1)--(0,4,1)--cycle
;
\draw[line width=1pt,fill=black!20]
(2,3,0)--(2,4,0)--(2,4,1)--(2,3,1)--cycle
(2,2,0)--(2,3,0)--(2,3,1)--(2,2,1)--cycle
(2,1,0)--(2,2,0)--(2,2,1)--(2,1,1)--cycle
;
\draw[line width=1pt,fill=black!50]
(1,4,0)--(1,4,1)--(2,4,1)--(2,4,0)--cycle
(0,4,0)--(0,4,1)--(1,4,1)--(1,4,0)--cycle
;
\draw[line width=2pt](2,1,0)--(2,2,0)--(2,3,0)--(2,4,0)--(1,4,0)--(0,4,0)--(0,4,1)--(0,3,1)--(1,3,1)--(1,2,1)--(1,1,1)--(2,1,1)--cycle;
\end{scope}
\begin{scope}[xshift=6cm,yshift=0cm]
\clip(3,0,0)--(3,1,0)--(3,2,0)--(2,2,0)--(1,2,0)--(1,2,1)--(1,1,1)--(2,1,1)--(2,0,1)--(3,0,1)--(3,0,0)--cycle;
\draw[line width=1pt,fill=cyan]
(1,1,1)--(2,1,1)--(2,2,1)--(1,2,1)--cycle
;
\draw[line width=1pt,fill=red]
(2,1,1)--(3,1,1)--(3,2,1)--(2,2,1)--cycle
(2,0,1)--(3,0,1)--(3,1,1)--(2,1,1)--cycle
;
\draw[line width=1pt,fill=black!20]
(3,1,0)--(3,2,0)--(3,2,1)--(3,1,1)--cycle
(3,0,0)--(3,1,0)--(3,1,1)--(3,0,1)--cycle
;
\draw[line width=1pt,fill=black!50]
(2,2,0)--(2,2,1)--(3,2,1)--(3,2,0)--cycle
(1,2,0)--(1,2,1)--(2,2,1)--(2,2,0)--cycle
;
\draw[line width=2pt](3,0,0)--(3,1,0)--(3,2,0)--(2,2,0)--(1,2,0)--(1,2,1)--(1,1,1)--(2,1,1)--(2,0,1)--(3,0,1)--(3,0,0)--cycle;
\end{scope}
\begin{scope}[xshift=12cm,yshift=0cm]
\clip(3,0,0)--(3,1,0)--(3,2,0)--(2,2,0)--(2,3,0)--(1,3,0)--(1,3,1)--(1,2,1)--(1,1,1)--(2,1,1)--(2,0,1)--(3,0,1)--(3,0,0)--cycle;
\draw[line width=1pt,fill=black!5]
(2,1,1)--(3,1,1)--(3,2,1)--(2,2,1)--cycle
(2,0,1)--(3,0,1)--(3,1,1)--(2,1,1)--cycle
(1,2,1)--(2,2,1)--(2,3,1)--(1,3,1)--cycle
(1,1,1)--(2,1,1)--(2,2,1)--(1,2,1)--cycle
;
\draw[line width=1pt,fill=black!20]
(3,1,0)--(3,2,0)--(3,2,1)--(3,1,1)--cycle
(3,0,0)--(3,1,0)--(3,1,1)--(3,0,1)--cycle
(2,2,0)--(2,3,0)--(2,3,1)--(2,2,1)--cycle
;
\draw[line width=1pt,fill=black!50]
(2,2,0)--(2,2,1)--(3,2,1)--(3,2,0)--cycle
(1,3,0)--(1,3,1)--(2,3,1)--(2,3,0)--cycle
;
\draw[line width=2pt](3,0,0)--(3,1,0)--(3,2,0)--(2,2,0)--(2,3,0)--(1,3,0)--(1,3,1)--(1,2,1)--(1,1,1)--(2,1,1)--(2,0,1)--(3,0,1)--(3,0,0)--cycle;
\end{scope}
\begin{scope}[xshift=18cm,yshift=0cm]
\clip(3,0,0)--(3,1,0)--(2,1,0)--(2,1,1)--(2,0,1)--(3,0,1)--cycle;
\draw[line width=1pt,fill=black!5]
(2,0,1)--(3,0,1)--(3,1,1)--(2,1,1)--cycle
;
\draw[line width=1pt,fill=black!20]
(3,0,0)--(3,1,0)--(3,1,1)--(3,0,1)--cycle
;
\draw[line width=1pt,fill=black!50]
(2,1,0)--(2,1,1)--(3,1,1)--(3,1,0)--cycle
;
\draw[line width=2pt](3,0,0)--(3,1,0)--(2,1,0)--(2,1,1)--(2,0,1)--(3,0,1)--cycle;
\end{scope}
\begin{scope}[xshift=0cm,yshift=-44mm]
\begin{scope}[xshift=0cm,yshift=0cm]
\clip(3,0,0)--(3,4,0)--(0,4,0)--(0,3,0)--(1,3,0)--(1,2,0)--(1,1,0)--(2,1,0)--(2,0,0)--(3,0,0)--cycle;
\draw[line width=1pt,fill=black!5]
(2,3,0)--(3,3,0)--(3,4,0)--(2,4,0)--cycle
(2,2,0)--(3,2,0)--(3,3,0)--(2,3,0)--cycle
(2,1,0)--(3,1,0)--(3,2,0)--(2,2,0)--cycle
(2,0,0)--(3,0,0)--(3,1,0)--(2,1,0)--cycle
(1,3,0)--(2,3,0)--(2,4,0)--(1,4,0)--cycle
(1,2,0)--(2,2,0)--(2,3,0)--(1,3,0)--cycle
(1,1,0)--(2,1,0)--(2,2,0)--(1,2,0)--cycle
(0,3,0)--(1,3,0)--(1,4,0)--(0,4,0)--cycle
;
\draw[line width=1pt,fill=black!20]
;
\draw[line width=1pt,fill=black!50]
;
\draw[line width=2pt](3,0,0)--(3,4,0)--(0,4,0)--(0,3,0)--(1,3,0)--(1,2,0)--(1,1,0)--(2,1,0)--(2,0,0)--(3,0,0)--cycle;
\end{scope}
\begin{scope}[xshift=6cm,yshift=0cm]
\clip(3,0,0)--(3,4,0)--(0,4,0)--(0,4,1)--(0,3,1)--(1,3,1)--(1,2,1)--(1,1,1)--(2,1,1)--(2,1,0)--(2,0,0)--(3,0,0)--cycle;
\draw[line width=1pt,fill=black!5]
(2,3,0)--(3,3,0)--(3,4,0)--(2,4,0)--cycle
(2,2,0)--(3,2,0)--(3,3,0)--(2,3,0)--cycle
(2,1,0)--(3,1,0)--(3,2,0)--(2,2,0)--cycle
(2,0,0)--(3,0,0)--(3,1,0)--(2,1,0)--cycle
(1,3,1)--(2,3,1)--(2,4,1)--(1,4,1)--cycle
(1,2,1)--(2,2,1)--(2,3,1)--(1,3,1)--cycle
(1,1,1)--(2,1,1)--(2,2,1)--(1,2,1)--cycle
(0,3,1)--(1,3,1)--(1,4,1)--(0,4,1)--cycle
;
\draw[line width=1pt,fill=black!20]
(2,3,0)--(2,4,0)--(2,4,1)--(2,3,1)--cycle
(2,2,0)--(2,3,0)--(2,3,1)--(2,2,1)--cycle
(2,1,0)--(2,2,0)--(2,2,1)--(2,1,1)--cycle
;
\draw[line width=1pt,fill=black!50]
(1,4,0)--(1,4,1)--(2,4,1)--(2,4,0)--cycle
(0,4,0)--(0,4,1)--(1,4,1)--(1,4,0)--cycle
;
\draw[line width=2pt](3,0,0)--(3,4,0)--(0,4,0)--(0,4,1)--(0,3,1)--(1,3,1)--(1,2,1)--(1,1,1)--(2,1,1)--(2,1,0)--(2,0,0)--(3,0,0)--cycle;
\end{scope}
\begin{scope}[xshift=12cm,yshift=0cm]
\clip(3,0,0)--(3,4,0)--(0,4,0)--(0,4,1)--(0,3,1)--(1,3,1)--(1,2,1)--(1,1,1)--(2,1,1)--(2,0,1)--(3,0,1)--(3,0,0)--cycle;
\draw[line width=1pt,fill=black!5]
(2,3,0)--(3,3,0)--(3,4,0)--(2,4,0)--cycle
(1,3,1)--(2,3,1)--(2,4,1)--(1,4,1)--cycle
(1,2,1)--(2,2,1)--(2,3,1)--(1,3,1)--cycle
(1,1,1)--(2,1,1)--(2,2,1)--(1,2,1)--cycle
(0,3,1)--(1,3,1)--(1,4,1)--(0,4,1)--cycle
;
\draw[line width=1pt,fill=red]
(2,0,1)--(3,0,1)--(3,1,1)--(2,1,1)--cycle
(2,1,1)--(3,1,1)--(3,2,1)--(2,2,1)--cycle
;
\draw[line width=1pt,fill=cyan]
(2,2,1)--(3,2,1)--(3,3,1)--(2,3,1)--cycle
;
\draw[line width=1pt,fill=black!20]
(3,2,0)--(3,3,0)--(3,3,1)--(3,2,1)--cycle
(3,1,0)--(3,2,0)--(3,2,1)--(3,1,1)--cycle
(3,0,0)--(3,1,0)--(3,1,1)--(3,0,1)--cycle
(2,3,0)--(2,4,0)--(2,4,1)--(2,3,1)--cycle
;
\draw[line width=1pt,fill=black!50]
(2,3,0)--(2,3,1)--(3,3,1)--(3,3,0)--cycle
(1,4,0)--(1,4,1)--(2,4,1)--(2,4,0)--cycle
(0,4,0)--(0,4,1)--(1,4,1)--(1,4,0)--cycle
;
\draw[line width=2pt](3,0,0)--(3,4,0)--(0,4,0)--(0,4,1)--(0,3,1)--(1,3,1)--(1,2,1)--(1,1,1)--(2,1,1)--(2,0,1)--(3,0,1)--(3,0,0)--cycle;
\end{scope}
\begin{scope}[xshift=18cm,yshift=0cm]
\clip(3,0,0)--(3,4,0)--(0,4,0)--(0,4,1)--(0,3,1)--(1,3,1)--(1,3,2)--(1,2,2)--(1,1,2)--(2,1,2)--(2,0,2)--(3,0,2)--(3,0,0)--cycle;
\draw[line width=1pt,fill=black!5]
(2,3,0)--(3,3,0)--(3,4,0)--(2,4,0)--cycle
(2,2,1)--(3,2,1)--(3,3,1)--(2,3,1)--cycle
(2,1,2)--(3,1,2)--(3,2,2)--(2,2,2)--cycle
(2,0,2)--(3,0,2)--(3,1,2)--(2,1,2)--cycle
(1,3,1)--(2,3,1)--(2,4,1)--(1,4,1)--cycle
(1,2,2)--(2,2,2)--(2,3,2)--(1,3,2)--cycle
(1,1,2)--(2,1,2)--(2,2,2)--(1,2,2)--cycle
(0,3,1)--(1,3,1)--(1,4,1)--(0,4,1)--cycle
;
\draw[line width=1pt,fill=black!20]
(3,2,0)--(3,3,0)--(3,3,1)--(3,2,1)--cycle
(3,1,1)--(3,2,1)--(3,2,2)--(3,1,2)--cycle
(3,1,0)--(3,2,0)--(3,2,1)--(3,1,1)--cycle
(3,0,1)--(3,1,1)--(3,1,2)--(3,0,2)--cycle
(3,0,0)--(3,1,0)--(3,1,1)--(3,0,1)--cycle
(2,3,0)--(2,4,0)--(2,4,1)--(2,3,1)--cycle
(2,2,1)--(2,3,1)--(2,3,2)--(2,2,2)--cycle
;
\draw[line width=1pt,fill=black!50]
(2,3,0)--(2,3,1)--(3,3,1)--(3,3,0)--cycle
(2,2,1)--(2,2,2)--(3,2,2)--(3,2,1)--cycle
(1,4,0)--(1,4,1)--(2,4,1)--(2,4,0)--cycle
(1,3,1)--(1,3,2)--(2,3,2)--(2,3,1)--cycle
(0,4,0)--(0,4,1)--(1,4,1)--(1,4,0)--cycle
;
\draw[line width=2pt](3,0,0)--(3,4,0)--(0,4,0)--(0,4,1)--(0,3,1)--(1,3,1)--(1,3,2)--(1,2,2)--(1,1,2)--(2,1,2)--(2,0,2)--(3,0,2)--(3,0,0)--cycle;
\end{scope}
\end{scope}
\end{tikzpicture}
}

\myfig[b]{}{lex}{
\begin{tikzpicture}[tdplot_main_coords]
\begin{scope}[xshift=0cm,yshift=-2cm]
\draw[line width=2pt,->](2.5,.5,1.5)--(2.5,.5,0);
\begin{scope}[xshift=6cm,yshift=0cm]
\draw[line width=2pt,->](1.5,1.5,1.5)--(1.5,1.5,0);
\end{scope}
\begin{scope}[xshift=12cm,yshift=0cm]
\draw[line width=2pt,->](2.5,.5,1.5)--(2.5,.5,0);
\end{scope}
\begin{scope}[xshift=18cm,yshift=0cm]
\draw[line width=2pt,->](2.5,.5,1.5)--(2.5,.5,0);
\end{scope}
\end{scope}
\begin{scope}[xshift=0cm,yshift=0cm]
\clip(3,0,0)--(3,1,0)--(3,2,0)--(2,2,0)--(2,3,0)--(2,4,0)--(1,4,0)--(0,4,0)--(0,4,1)--(0,3,1)--(1,3,1)--(1,2,1)--(1,1,1)--(2,1,1)--(2,0,1)--(3,0,1)--(3,0,0)--cycle;
\draw[line width=1pt,fill=black!5]
(2,1,1)--(3,1,1)--(3,2,1)--(2,2,1)--cycle
(2,0,1)--(3,0,1)--(3,1,1)--(2,1,1)--cycle
(1,3,1)--(2,3,1)--(2,4,1)--(1,4,1)--cycle
(1,2,1)--(2,2,1)--(2,3,1)--(1,3,1)--cycle
(1,1,1)--(2,1,1)--(2,2,1)--(1,2,1)--cycle
(0,3,1)--(1,3,1)--(1,4,1)--(0,4,1)--cycle
;
\draw[line width=1pt,fill=black!20]
(3,1,0)--(3,2,0)--(3,2,1)--(3,1,1)--cycle
(3,0,0)--(3,1,0)--(3,1,1)--(3,0,1)--cycle
(2,3,0)--(2,4,0)--(2,4,1)--(2,3,1)--cycle
(2,2,0)--(2,3,0)--(2,3,1)--(2,2,1)--cycle
;
\draw[line width=1pt,fill=black!50]
(2,2,0)--(2,2,1)--(3,2,1)--(3,2,0)--cycle
(1,4,0)--(1,4,1)--(2,4,1)--(2,4,0)--cycle
(0,4,0)--(0,4,1)--(1,4,1)--(1,4,0)--cycle
;
\draw[line width=2pt](3,0,0)--(3,1,0)--(3,2,0)--(2,2,0)--(2,3,0)--(2,4,0)--(1,4,0)--(0,4,0)--(0,4,1)--(0,3,1)--(1,3,1)--(1,2,1)--(1,1,1)--(2,1,1)--(2,0,1)--(3,0,1)--(3,0,0)--cycle;
\end{scope}
\begin{scope}[xshift=6cm,yshift=0cm]
\clip(2,1,0)--(2,2,0)--(2,3,0)--(1,3,0)--(1,3,1)--(1,2,1)--(1,1,1)--(2,1,1)--cycle;
\draw[line width=1pt,fill=black!5]
(1,2,1)--(2,2,1)--(2,3,1)--(1,3,1)--cycle
(1,1,1)--(2,1,1)--(2,2,1)--(1,2,1)--cycle
;
\draw[line width=1pt,fill=black!20]
(2,2,0)--(2,3,0)--(2,3,1)--(2,2,1)--cycle
(2,1,0)--(2,2,0)--(2,2,1)--(2,1,1)--cycle
;
\draw[line width=1pt,fill=black!50]
(1,3,0)--(1,3,1)--(2,3,1)--(2,3,0)--cycle
;
\draw[line width=2pt](2,1,0)--(2,2,0)--(2,3,0)--(1,3,0)--(1,3,1)--(1,2,1)--(1,1,1)--(2,1,1)--cycle;
\end{scope}
\begin{scope}[xshift=12cm,yshift=0cm]
\clip(3,0,0)--(3,1,0)--(3,2,0)--(2,2,0)--(1,2,0)--(1,2,1)--(1,1,1)--(2,1,1)--(2,0,1)--(3,0,1)--(3,0,0)--cycle;
\draw[line width=1pt,fill=green]
(1,1,1)--(2,1,1)--(2,2,1)--(1,2,1)--cycle
;
\draw[line width=1pt,fill=yellow]
(2,1,1)--(3,1,1)--(3,2,1)--(2,2,1)--cycle
(2,0,1)--(3,0,1)--(3,1,1)--(2,1,1)--cycle
;
\draw[line width=1pt,fill=black!20]
(3,1,0)--(3,2,0)--(3,2,1)--(3,1,1)--cycle
(3,0,0)--(3,1,0)--(3,1,1)--(3,0,1)--cycle
;
\draw[line width=1pt,fill=black!50]
(2,2,0)--(2,2,1)--(3,2,1)--(3,2,0)--cycle
(1,2,0)--(1,2,1)--(2,2,1)--(2,2,0)--cycle
;
\draw[line width=2pt](3,0,0)--(3,1,0)--(3,2,0)--(2,2,0)--(1,2,0)--(1,2,1)--(1,1,1)--(2,1,1)--(2,0,1)--(3,0,1)--(3,0,0)--cycle;
\end{scope}
\begin{scope}[xshift=18cm,yshift=0cm]
\clip(3,0,0)--(3,1,0)--(2,1,0)--(2,1,1)--(2,0,1)--(3,0,1)--cycle;
\draw[line width=1pt,fill=black!5]
(2,0,1)--(3,0,1)--(3,1,1)--(2,1,1)--cycle
;
\draw[line width=1pt,fill=black!20]
(3,0,0)--(3,1,0)--(3,1,1)--(3,0,1)--cycle
;
\draw[line width=1pt,fill=black!50]
(2,1,0)--(2,1,1)--(3,1,1)--(3,1,0)--cycle
;
\draw[line width=2pt](3,0,0)--(3,1,0)--(2,1,0)--(2,1,1)--(2,0,1)--(3,0,1)--cycle;
\end{scope}
\begin{scope}[xshift=0cm,yshift=-44mm]
\begin{scope}[xshift=0cm,yshift=0cm]
\clip(3,0,0)--(3,4,0)--(0,4,0)--(0,3,0)--(1,3,0)--(1,2,0)--(1,1,0)--(2,1,0)--(2,0,0)--(3,0,0)--cycle;
\draw[line width=1pt,fill=black!5]
(2,3,0)--(3,3,0)--(3,4,0)--(2,4,0)--cycle
(2,2,0)--(3,2,0)--(3,3,0)--(2,3,0)--cycle
(2,1,0)--(3,1,0)--(3,2,0)--(2,2,0)--cycle
(2,0,0)--(3,0,0)--(3,1,0)--(2,1,0)--cycle
(1,3,0)--(2,3,0)--(2,4,0)--(1,4,0)--cycle
(1,2,0)--(2,2,0)--(2,3,0)--(1,3,0)--cycle
(1,1,0)--(2,1,0)--(2,2,0)--(1,2,0)--cycle
(0,3,0)--(1,3,0)--(1,4,0)--(0,4,0)--cycle
;
\draw[line width=1pt,fill=black!20]
;
\draw[line width=1pt,fill=black!50]
;
\draw[line width=2pt](3,0,0)--(3,4,0)--(0,4,0)--(0,3,0)--(1,3,0)--(1,2,0)--(1,1,0)--(2,1,0)--(2,0,0)--(3,0,0)--cycle;
\end{scope}
\begin{scope}[xshift=6cm,yshift=0cm]
\clip(3,0,0)--(3,4,0)--(0,4,0)--(0,4,1)--(0,3,1)--(1,3,1)--(1,2,1)--(1,1,1)--(2,1,1)--(2,0,1)--(3,0,1)--(3,0,0)--cycle;
\draw[line width=1pt,fill=black!5]
(2,3,0)--(3,3,0)--(3,4,0)--(2,4,0)--cycle
(2,2,0)--(3,2,0)--(3,3,0)--(2,3,0)--cycle
(2,1,1)--(3,1,1)--(3,2,1)--(2,2,1)--cycle
(2,0,1)--(3,0,1)--(3,1,1)--(2,1,1)--cycle
(1,3,1)--(2,3,1)--(2,4,1)--(1,4,1)--cycle
(1,2,1)--(2,2,1)--(2,3,1)--(1,3,1)--cycle
(1,1,1)--(2,1,1)--(2,2,1)--(1,2,1)--cycle
(0,3,1)--(1,3,1)--(1,4,1)--(0,4,1)--cycle
;
\draw[line width=1pt,fill=black!20]
(3,1,0)--(3,2,0)--(3,2,1)--(3,1,1)--cycle
(3,0,0)--(3,1,0)--(3,1,1)--(3,0,1)--cycle
(2,3,0)--(2,4,0)--(2,4,1)--(2,3,1)--cycle
(2,2,0)--(2,3,0)--(2,3,1)--(2,2,1)--cycle
;
\draw[line width=1pt,fill=black!50]
(2,2,0)--(2,2,1)--(3,2,1)--(3,2,0)--cycle
(1,4,0)--(1,4,1)--(2,4,1)--(2,4,0)--cycle
(0,4,0)--(0,4,1)--(1,4,1)--(1,4,0)--cycle
;
\draw[line width=2pt](3,0,0)--(3,4,0)--(0,4,0)--(0,4,1)--(0,3,1)--(1,3,1)--(1,2,1)--(1,1,1)--(2,1,1)--(2,0,1)--(3,0,1)--(3,0,0)--cycle;
\end{scope}
\begin{scope}[xshift=12cm,yshift=0cm]
\clip(3,0,0)--(3,4,0)--(0,4,0)--(0,4,1)--(0,3,1)--(1,3,1)--(1,3,2)--(1,2,2)--(1,1,2)--(2,1,2)--(2,1,1)--(2,0,1)--(3,0,1)--(3,0,0)--cycle;
\draw[line width=1pt,fill=black!5]
(2,3,0)--(3,3,0)--(3,4,0)--(2,4,0)--cycle
(2,2,0)--(3,2,0)--(3,3,0)--(2,3,0)--cycle
(2,1,1)--(3,1,1)--(3,2,1)--(2,2,1)--cycle
(2,0,1)--(3,0,1)--(3,1,1)--(2,1,1)--cycle
(1,3,1)--(2,3,1)--(2,4,1)--(1,4,1)--cycle
(1,2,2)--(2,2,2)--(2,3,2)--(1,3,2)--cycle
(1,1,2)--(2,1,2)--(2,2,2)--(1,2,2)--cycle
(0,3,1)--(1,3,1)--(1,4,1)--(0,4,1)--cycle
;
\draw[line width=1pt,fill=black!20]
(3,1,0)--(3,2,0)--(3,2,1)--(3,1,1)--cycle
(3,0,0)--(3,1,0)--(3,1,1)--(3,0,1)--cycle
(2,3,0)--(2,4,0)--(2,4,1)--(2,3,1)--cycle
(2,2,1)--(2,3,1)--(2,3,2)--(2,2,2)--cycle
(2,2,0)--(2,3,0)--(2,3,1)--(2,2,1)--cycle
(2,1,1)--(2,2,1)--(2,2,2)--(2,1,2)--cycle
;
\draw[line width=1pt,fill=black!50]
(2,2,0)--(2,2,1)--(3,2,1)--(3,2,0)--cycle
(1,4,0)--(1,4,1)--(2,4,1)--(2,4,0)--cycle
(1,3,1)--(1,3,2)--(2,3,2)--(2,3,1)--cycle
(0,4,0)--(0,4,1)--(1,4,1)--(1,4,0)--cycle
;
\draw[line width=2pt](3,0,0)--(3,4,0)--(0,4,0)--(0,4,1)--(0,3,1)--(1,3,1)--(1,3,2)--(1,2,2)--(1,1,2)--(2,1,2)--(2,1,1)--(2,0,1)--(3,0,1)--(3,0,0)--cycle;
\end{scope}
\begin{scope}[xshift=18cm,yshift=0cm]
\clip(3,0,0)--(3,4,0)--(0,4,0)--(0,4,1)--(0,3,1)--(1,3,1)--(1,3,2)--(1,2,2)--(1,1,2)--(2,1,2)--(2,0,2)--(3,0,2)--(3,0,0)--cycle;
\draw[line width=1pt,fill=black!5]
(2,3,0)--(3,3,0)--(3,4,0)--(2,4,0)--cycle
(1,3,1)--(2,3,1)--(2,4,1)--(1,4,1)--cycle
(1,2,2)--(2,2,2)--(2,3,2)--(1,3,2)--cycle
(1,1,2)--(2,1,2)--(2,2,2)--(1,2,2)--cycle
(0,3,1)--(1,3,1)--(1,4,1)--(0,4,1)--cycle
;
\draw[line width=1pt,fill=yellow]
(2,0,2)--(3,0,2)--(3,1,2)--(2,1,2)--cycle
(2,1,2)--(3,1,2)--(3,2,2)--(2,2,2)--cycle
;
\draw[line width=1pt,fill=green]
(2,2,1)--(3,2,1)--(3,3,1)--(2,3,1)--cycle
;
\draw[line width=1pt,fill=black!20]
(3,2,0)--(3,3,0)--(3,3,1)--(3,2,1)--cycle
(3,1,1)--(3,2,1)--(3,2,2)--(3,1,2)--cycle
(3,1,0)--(3,2,0)--(3,2,1)--(3,1,1)--cycle
(3,0,1)--(3,1,1)--(3,1,2)--(3,0,2)--cycle
(3,0,0)--(3,1,0)--(3,1,1)--(3,0,1)--cycle
(2,3,0)--(2,4,0)--(2,4,1)--(2,3,1)--cycle
(2,2,1)--(2,3,1)--(2,3,2)--(2,2,2)--cycle
;
\draw[line width=1pt,fill=black!50]
(2,3,0)--(2,3,1)--(3,3,1)--(3,3,0)--cycle
(2,2,1)--(2,2,2)--(3,2,2)--(3,2,1)--cycle
(1,4,0)--(1,4,1)--(2,4,1)--(2,4,0)--cycle
(1,3,1)--(1,3,2)--(2,3,2)--(2,3,1)--cycle
(0,4,0)--(0,4,1)--(1,4,1)--(1,4,0)--cycle
;
\draw[line width=2pt](3,0,0)--(3,4,0)--(0,4,0)--(0,4,1)--(0,3,1)--(1,3,1)--(1,3,2)--(1,2,2)--(1,1,2)--(2,1,2)--(2,0,2)--(3,0,2)--(3,0,0)--cycle;
\end{scope}
\end{scope}
\end{tikzpicture}
}

\smallskip
Moreover we remark that the lexicographic factorisation of a given reverse plane partition $\pi$ can be obtained inductively by finding the minimal candidate $v\in\cand(\pi)$ and constructing the path $Q(v,\pi)$ along which $\pi$ is reduced.
For example, the lexicographic factorisation from \reff{lex} is obtained in \reff{inverse}.
At each step the candidates are circled.
Moreover, the minimal candidate $v$ and the path $Q(v,\pi)$ are shaded.
The tableaux in the second row record the multi-set of rim-hooks that have been extracted.

\myfig[t]{}{inverse}{
\begin{tikzpicture}
\begin{scope}[yshift=4cm]
\begin{scope}[xshift=0cm]
\fill[black!30]
(3,2)--(4,2)--(4,3)--(3,3);
\draw
(0,1)--(1,1)
(0,2)--(3,2)
(1,1)--(1,3)
(2,1)--(2,3)
(3,2)--(3,3)
;
\draw[line width=1pt](0,0)--(1,0)--(1,1)--(3,1)--(3,2)--(4,2)--(4,3)--(0,3)--cycle;
\draw[xshift=5mm,yshift=5mm]
(0,2)node{\small{$0$}}
(1,2)node[mybead]{\small{$1$}}
(2,2)node{\small{$2$}}
(3,2)node[mybead]{\small{$3$}}
(0,1)node{\small{$1$}}
(1,1)node[mybead]{\small{$2$}}
(2,1)node{\small{$2$}}
(0,0)node[mybead]{\small{$1$}};
\end{scope}
\begin{scope}[xshift=5.5cm]
\fill[black!30]
(1,2)--(4,2)--(4,3)--(1,3)--cycle
;
\draw
(0,1)--(1,1)(0,2)--(3,2)(1,1)--(1,3)(2,1)--(2,3)(3,2)--(3,3)
;
\draw[line width=1pt](0,0)--(1,0)--(1,1)--(3,1)--(3,2)--(4,2)--(4,3)--(0,3)--cycle;
\draw[xshift=5mm,yshift=5mm]
(0,2)node{\small{$0$}}
(1,2)node[mybead]{\small{$1$}}
(2,2)node{\small{$2$}}
(3,2)node{\small{$2$}}
(0,1)node{\small{$1$}}
(1,1)node[mybead]{\small{$2$}}
(2,1)node{\small{$2$}}
(0,0)node[mybead]{\small{$1$}};
\end{scope}
\begin{scope}[xshift=11cm]
\fill[black!30]
(1,1)--(3,1)--(3,2)--(1,2)--cycle
;
\draw(0,1)--(1,1)(0,2)--(3,2)(1,1)--(1,3)(2,1)--(2,3)(3,2)--(3,3)
;
\draw[line width=1pt](0,0)--(1,0)--(1,1)--(3,1)--(3,2)--(4,2)--(4,3)--(0,3)--cycle;
\draw[xshift=5mm,yshift=5mm]
(0,2)node{\small{$0$}}
(1,2)node{\small{$0$}}
(2,2)node{\small{$1$}}
(3,2)node{\small{$1$}}
(0,1)node{\small{$1$}}
(1,1)node[mybead]{\small{$2$}}
(2,1)node{\small{$2$}}
(0,0)node[mybead]{\small{$1$}};
\end{scope}
\begin{scope}[xshift=16.5cm]
\fill[black!30]
(0,0)--(1,0)--(1,1)--(3,1)--(3,2)--(4,2)--(4,3)--(2,3)--(2,2)--(0,2)--cycle
;
\draw
(0,1)--(1,1)(0,2)--(3,2)(1,1)--(1,3)(2,1)--(2,3)(3,2)--(3,3)
;
\draw[line width=1pt](0,0)--(1,0)--(1,1)--(3,1)--(3,2)--(4,2)--(4,3)--(0,3)--cycle;
\draw[xshift=5mm,yshift=5mm]
(0,2)node{\small{$0$}}
(1,2)node{\small{$0$}}
(2,2)node{\small{$1$}}
(3,2)node{\small{$1$}}
(0,1)node{\small{$1$}}
(1,1)node{\small{$1$}}
(2,1)node{\small{$1$}}
(0,0)node[mybead]{\small{$1$}};
\end{scope}
\begin{scope}[xshift=22cm]
\draw
(0,1)--(1,1)
(0,2)--(3,2)
(1,1)--(1,3)
(2,1)--(2,3)
(3,2)--(3,3)
;
\draw[line width=1pt](0,0)--(1,0)--(1,1)--(3,1)--(3,2)--(4,2)--(4,3)--(0,3)--cycle;
\draw[xshift=5mm,yshift=5mm]
(0,2)node{\small{$0$}}
(1,2)node{\small{$0$}}
(2,2)node{\small{$0$}}
(3,2)node{\small{$0$}}
(0,1)node{\small{$0$}}
(1,1)node{\small{$0$}}
(2,1)node{\small{$0$}}
(0,0)node{\small{$0$}};
\end{scope}
\end{scope}
\begin{scope}[yshift=0cm]
\begin{scope}[xshift=0cm]
\draw(0,1)--(1,1)(0,2)--(3,2)(1,1)--(1,3)(2,1)--(2,3)(3,2)--(3,3)
;
\draw[line width=1pt](0,0)--(1,0)--(1,1)--(3,1)--(3,2)--(4,2)--(4,3)--(0,3)--cycle
;
\draw[xshift=5mm,yshift=5mm]
(0,2)node{\small{$0$}}
(1,2)node{\small{$0$}}
(2,2)node{\small{$0$}}
(3,2)node{\small{$0$}}
(0,1)node{\small{$0$}}
(1,1)node{\small{$0$}}
(2,1)node{\small{$0$}}
(0,0)node{\small{$0$}};
\end{scope}
\begin{scope}[xshift=5.5cm]
\draw
(0,1)--(1,1)
(0,2)--(3,2)
(1,1)--(1,3)
(2,1)--(2,3)
(3,2)--(3,3)
;
\draw[line width=1pt](0,0)--(1,0)--(1,1)--(3,1)--(3,2)--(4,2)--(4,3)--(0,3)--cycle;
\draw[xshift=5mm,yshift=5mm]
(0,2)node{\small{$0$}}
(1,2)node{\small{$0$}}
(2,2)node{\small{$0$}}
(3,2)node{\small{$1$}}
(0,1)node{\small{$0$}}
(1,1)node{\small{$0$}}
(2,1)node{\small{$0$}}
(0,0)node{\small{$0$}}
;
\end{scope}
\begin{scope}[xshift=11cm]
\draw
(0,1)--(1,1)
(0,2)--(3,2)
(1,1)--(1,3)
(2,1)--(2,3)
(3,2)--(3,3)
;
\draw[line width=1pt](0,0)--(1,0)--(1,1)--(3,1)--(3,2)--(4,2)--(4,3)--(0,3)--cycle;
\draw[xshift=5mm,yshift=5mm]
(0,2)node{\small{$0$}}
(1,2)node{\small{$0$}}
(2,2)node{\small{$1$}}
(3,2)node{\small{$1$}}
(0,1)node{\small{$0$}}
(1,1)node{\small{$0$}}
(2,1)node{\small{$0$}}
(0,0)node{\small{$0$}}
;
\end{scope}
\begin{scope}[xshift=16.5cm]
\draw
(0,1)--(1,1)
(0,2)--(3,2)
(1,1)--(1,3)
(2,1)--(2,3)
(3,2)--(3,3)
;
\draw[line width=1pt](0,0)--(1,0)--(1,1)--(3,1)--(3,2)--(4,2)--(4,3)--(0,3)--cycle ;
\draw[xshift=5mm,yshift=5mm]
(0,2)node{\small{$0$}}
(1,2)node{\small{$0$}}
(2,2)node{\small{$1$}}
(3,2)node{\small{$1$}}
(0,1)node{\small{$0$}}
(1,1)node{\small{$1$}}
(2,1)node{\small{$0$}}
(0,0)node{\small{$0$}}
;
\end{scope}
\begin{scope}[xshift=22cm]
\draw
(0,1)--(1,1)
(0,2)--(3,2)
(1,1)--(1,3)
(2,1)--(2,3)
(3,2)--(3,3)
;
\draw[line width=1pt](0,0)--(1,0)--(1,1)--(3,1)--(3,2)--(4,2)--(4,3)--(0,3)--cycle
;
\draw[xshift=5mm,yshift=5mm]
(0,2)node{\small{$1$}}
(1,2)node{\small{$0$}}
(2,2)node{\small{$1$}}
(3,2)node{\small{$1$}}
(0,1)node{\small{$0$}}
(1,1)node{\small{$1$}}
(2,1)node{\small{$0$}}
(0,0)node{\small{$0$}}
;
\end{scope}
\end{scope}
\end{tikzpicture}
}

\smallskip
Finally we demonstrate that the bijection from \reft{bij} can be used to prove the generating function identity from \reft{Gansner}.

\begin{proof}[Proof of \reft{Gansner}]
For each rim-hook $h$ the set of contents $\{c(u):u\in h\}$ is an interval $\{j-\lambda_j',\dots,\lambda_i-i\}$ where $(i,j)$ is the cell corresponding to $h$.
Moreover if $P$ is a north-east-path in $\lambda$ with $\omega(P)=\omega(h)$ and $\ell(P)=\ell(h)$ then $\{c(u):u\in P\}=\{c(u):u\in h\}$.
Hence if $h$ inserts into a reverse plane partition $\pi$ then
\begin{eq*}
\prod_{k\in\Z}q_k^{\tr_k(h*\pi)}
=\prod_{k\in\Z}q_k^{\tr_k(\pi)}\cdot\prod_{k=j-\lambda_j'}^{\lambda_i-i}q_k\,.
\end{eq*}
The claim therefore follows from \reft{bij}.
\end{proof}

\section{Connections to known bijections}\label{Section:connections}

In this section we try to shed some light on how the bijection of \reft{bij} is related to previous work.

\smallskip
Let $\tab_{\lambda}=\{t:\lambda\to\N\}$.
The elements $t\in\tab_{\lambda}$ are sometimes called \emph{tableaux} of shape $\lambda$.
Note that $\tab_\lambda$ is in bijection with the multi-sets of rim-hooks of $\lambda$.
To obtain the multi-set corresponding to the function $t$ simply take $t(u)$ copies of the rim-hook $h^u$ for each $u\in\lambda$.

We view the bijection of \reft{bij} as a map from tableaux to reverse plane partitions of the same shape, and denote it by $\Phi:\tab_{\lambda}\to\rpp_{\lambda}$.

\smallskip
I.~Pak~\cite[Sec.~4 and~5]{Pak:hook_length_formula_geometric} describes a bijection between reverse plane partitions and tableaux of the same shape that is obtained inductively as the concatenation of piecewise linear maps between certain polytopes.
He uses this map to derive a proof of the hook-length formula and notes further connections to the Robinson--Schensted--Knuth correspondence.
Similar observations are made in \cite{Hopkins:rsk}.
It turns out that this map coincides with the bijection obtained in \refs{bijection}.

\smallskip
We now present the bijection defined in~\cite{Pak:hook_length_formula_geometric}.
Let $\lambda$ be a partition and $x\in\lambda$ an outer corner.
Set $\mu=\lambda-\{x\}$.
Define a map $\zeta_{\lambda,x}:\rpp_{\lambda}\to\rpp_{\mu}$ by
\begin{eq*}
\zeta_{\lambda,x}(\pi)(u)=
\begin{cases}
\max\{\pi(\n u),\pi(\w u)\}+\min\{\pi(\e u),\pi(\s u)\}-\pi(u)&\quad\text{if }c(u)=c(x),\\
\pi(u)&\quad\text{otherwise}.
\end{cases}
\end{eq*}
It is easy to see that $\zeta_{\lambda,x}(\pi)$ is indeed a reverse plane partition.

\smallskip
Furthermore define a map $\xi_{\lambda}:\rpp_{\lambda}\to\tab_{\lambda}$ inductively.
Assume that $\xi_{\mu}$ is already defined and set
\begin{eq*}
\xi_{\lambda,x}(\pi)(u)=
\begin{cases}
\pi(x)-\max\{\pi(\n u),\pi(\w u)\}&\quad\text{if }u=x,\\
\xi_{\mu}\circ\zeta_{\lambda,x}(\pi)(u)&\quad\text{otherwise.}
\end{cases}
\end{eq*}
One can show that $\xi_{\lambda,x}$ is independent of the choice of the outer corner $x$~\cite[Thm.~4]{Pak:hook_length_formula_geometric}.
Thus set $\xi_{\lambda}=\xi_{\lambda,x}$.
See \reff{xi} for an example.

\myfig[b]{}{xi}{
\begin{tikzpicture}
\begin{scope}[xshift=0cm]
\draw(0,0)grid(3,3);
\draw[line width=1pt](0,0)rectangle(3,3);
\draw[xshift=5mm,yshift=5mm]
(0,2)node{\small{$1$}}
(1,2)node{\small{$1$}}
(2,2)node{\small{$4$}}
(0,1)node{\small{$2$}}
(1,1)node{\small{$3$}}
(2,1)node{\small{$4$}}
(0,0)node{\small{$4$}}
(1,0)node{\small{$4$}}
(2,0)node{\small{$4$}}
;
\draw[->,line width=1pt,xshift=35mm,yshift=12mm](0,0)--node[anchor=south]{\small{$\zeta_{(3,3)}$}}(1,0)
;
\draw[->,line width=1pt,xshift=12mm,yshift=-15mm](0,1)--node[anchor=west]{\small{$\xi$}}(0,0)
;
\end{scope}
\begin{scope}[xshift=5cm]
\draw(0,0)grid(3,3);
\draw[line width=1pt](0,0)rectangle(3,3)(2,0)--(2,1)--(3,1);
\draw[xshift=5mm,yshift=5mm]
(0,2)node{\small{$0$}}
(1,2)node{\small{$1$}}
(2,2)node{\small{$4$}}
(0,1)node{\small{$2$}}
(1,1)node{\small{$3$}}
(2,1)node{\small{$4$}}
(0,0)node{\small{$4$}}
(1,0)node{\small{$4$}}
(2,0)node{\small{$0$}}
;
\draw[->,line width=1pt,xshift=35mm,yshift=12mm](0,0)--node[anchor=south]{\small{$\zeta_{(2,3)}$}}(1,0)
;
\end{scope}
\begin{scope}[xshift=10cm]
\draw(0,0)grid(3,3);
\draw[line width=1pt](0,0)rectangle(3,3)(2,0)--(2,2)--(3,2);
\draw[xshift=5mm,yshift=5mm]
(0,2)node{\small{$0$}}
(1,2)node{\small{$2$}}
(2,2)node{\small{$4$}}
(0,1)node{\small{$2$}}
(1,1)node{\small{$3$}}
(2,1)node{\small{$0$}}
(0,0)node{\small{$4$}}
(1,0)node{\small{$4$}}
(2,0)node{\small{$0$}}
;
\draw[->,line width=1pt,xshift=35mm,yshift=12mm](0,0)--node[anchor=south]{\small{$\zeta_{(1,3)}$}}(1,0)
;
\end{scope}
\begin{scope}[xshift=15cm]
\draw(0,0)grid(3,3);
\draw[line width=1pt](0,0)rectangle(3,3)(2,0)--(2,3);
\draw[xshift=5mm,yshift=5mm]
(0,2)node{\small{$0$}}
(1,2)node{\small{$2$}}
(2,2)node{\small{$2$}}
(0,1)node{\small{$2$}}
(1,1)node{\small{$3$}}
(2,1)node{\small{$0$}}
(0,0)node{\small{$4$}}
(1,0)node{\small{$4$}}
(2,0)node{\small{$0$}}
;
\draw[->,line width=1pt,xshift=35mm,yshift=12mm](0,0)--node[anchor=south]{\small{$\zeta_{(3,2)}$}}(1,0)
;
\end{scope}
\begin{scope}[xshift=20cm]
\draw(0,0)grid(3,3);
\draw[line width=1pt](0,0)rectangle(3,3)(1,0)--(1,1)--(2,1)--(2,3);
\draw[xshift=5mm,yshift=5mm]
(0,2)node{\small{$0$}}
(1,2)node{\small{$2$}}
(2,2)node{\small{$2$}}
(0,1)node{\small{$1$}}
(1,1)node{\small{$3$}}
(2,1)node{\small{$0$}}
(0,0)node{\small{$4$}}
(1,0)node{\small{$0$}}
(2,0)node{\small{$0$}}
;
\draw[->,line width=1pt,xshift=18mm,yshift=-15mm](0,1)--node[anchor=west]{\small{$\zeta_{(2,2)}$}}(0,0)
;
\end{scope}
\begin{scope}[yshift=-5cm]
\begin{scope}[xshift=20cm]
\draw(0,0)grid(3,3);
\draw[line width=1pt](0,0)rectangle(3,3)(1,0)--(1,2)--(2,2)--(2,3);
\draw[xshift=5mm,yshift=5mm]
(0,2)node{\small{$1$}}
(1,2)node{\small{$2$}}
(2,2)node{\small{$2$}}
(0,1)node{\small{$1$}}
(1,1)node{\small{$1$}}
(2,1)node{\small{$0$}}
(0,0)node{\small{$4$}}
(1,0)node{\small{$0$}}
(2,0)node{\small{$0$}}
;
\end{scope}
\begin{scope}[xshift=15cm]
\draw(0,0)grid(3,3);
\draw[line width=1pt](0,0)rectangle(3,3)(1,0)--(1,3);
\draw[xshift=5mm,yshift=5mm]
(0,2)node{\small{$1$}}
(1,2)node{\small{$1$}}
(2,2)node{\small{$2$}}
(0,1)node{\small{$1$}}
(1,1)node{\small{$1$}}
(2,1)node{\small{$0$}}
(0,0)node{\small{$4$}}
(1,0)node{\small{$0$}}
(2,0)node{\small{$0$}}
;
\draw[<-,line width=1pt,xshift=35mm,yshift=12mm](0,0)--node[anchor=south]{\small{$\zeta_{(1,2)}$}}(1,0)
;
\end{scope}
\begin{scope}[xshift=10cm]
\draw(0,0)grid(3,3);
\draw[line width=1pt](0,0)rectangle(3,3)(0,1)--(1,1)--(1,3);
\draw[xshift=5mm,yshift=5mm]
(0,2)node{\small{$1$}}
(1,2)node{\small{$1$}}
(2,2)node{\small{$2$}}
(0,1)node{\small{$1$}}
(1,1)node{\small{$1$}}
(2,1)node{\small{$0$}}
(0,0)node{\small{$3$}}
(1,0)node{\small{$0$}}
(2,0)node{\small{$0$}}
;
\draw[<-,line width=1pt,xshift=35mm,yshift=12mm](0,0)--node[anchor=south]{\small{$\zeta_{(3,1)}$}}(1,0)
;
\end{scope}
\begin{scope}[xshift=5cm]
\draw(0,0)grid(3,3);
\draw[line width=1pt](0,0)rectangle(3,3)(0,2)--(1,2)--(1,3);
\draw[xshift=5mm,yshift=5mm]
(0,2)node{\small{$1$}}
(1,2)node{\small{$1$}}
(2,2)node{\small{$2$}}
(0,1)node{\small{$0$}}
(1,1)node{\small{$1$}}
(2,1)node{\small{$0$}}
(0,0)node{\small{$3$}}
(1,0)node{\small{$0$}}
(2,0)node{\small{$0$}}
;
\draw[<-,line width=1pt,xshift=35mm,yshift=12mm](0,0)--node[anchor=south]{\small{$\zeta_{(2,1)}$}}(1,0)
;
\end{scope}
\begin{scope}[xshift=0cm]
\draw(0,0)grid(3,3);
\draw[line width=1pt](0,0)rectangle(3,3);
\draw[xshift=5mm,yshift=5mm]
(0,2)node{\small{$1$}}
(1,2)node{\small{$1$}}
(2,2)node{\small{$2$}}
(0,1)node{\small{$0$}}
(1,1)node{\small{$1$}}
(2,1)node{\small{$0$}}
(0,0)node{\small{$3$}}
(1,0)node{\small{$0$}}
(2,0)node{\small{$0$}}
;
\draw[<-,line width=1pt,xshift=35mm,yshift=12mm](0,0)--node[anchor=south]{\small{$\zeta_{(1,1)}$}}(1,0)
;
\end{scope}
\end{scope}
\end{tikzpicture}
}

\smallskip
Given an outer corner $x\in\lambda$ define
\begin{eq*}
\rpp_{\lambda,x}=\big\{\pi\in\rpp_{\lambda}:\pi(x)=\max\{\pi(\n x),\pi(\w x)\}\big\}.
\end{eq*}
Note that $\pi\in\rpp_{\lambda,x}$ if and only if $\xi_{\lambda}(\pi)(x)=0$.

We need to consider the four regions $\I,\O,\A$ and $\B$ defined in \refs{rimhooks} for multiple partitions at the same time.
To distinguish properly denote the corresponding subsets of $\lambda$ and $\mu$ by $\I_{\lambda}$ and $\I_{\mu}$ respectively, and use analogous notation for the other regions.

\smallskip
We first establish how the map $\zeta_{\lambda,x}$ affects candidates and the paths $Q(v,\pi)$ defined in \refs{factors}.
The following two lemmas are slightly technical, but not very difficult to prove.

\begin{lem}{pakcand}
Let $\lambda$ be a partition, $x\in\lambda$ an outer corner, $\mu=\lambda-\{x\}$, $\pi\in\rpp_{\lambda,x}$, $\rho=\zeta_{\lambda,x}(\pi)$, and $v\in\lambda$ be a cell.
Regard $\cand(\pi)$ and $\cand(\rho)$ as ordered sets with respect to the content order.
\begin{enumerate}[(i)]
\myi{candiff}
If $c(v)\notin\{c(x)-1,c(x),c(x)+1\}$ then $v\in\cand(\pi)$ if and only if $v\in\cand(\rho)$.
\myi{cx}
Let $c(v)=c(x)$ and $v$ be the minimum of $\cand(\pi)$.
Then $\pi(v)=\pi(\n v)$ implies that $\n v\in\cand(\rho)$ and $\e v\notin\cand(\rho)$.
On the other hand $\pi(v)>\pi(\n v)$ implies $\e v\in\cand(\rho)$.
\myi{cx+1}
Let $c(v)=c(x)+1$ and $v$ be the minimum of $\cand(\rho)$.
Then $\rho(v)\leq\rho(\s\w v)$ implies $\w v\in\cand(\pi)$ and $\s v\notin\cand(\pi)$.
On the other hand $\rho(v)>\rho(\s\w v)$ implies $\s v\in\cand(\pi)$.
\myi{cx-1}
Let $c(v)=c(x)-1$.
Then $v\in\cand(\pi)$ implies $v\in\cand(\rho)$, and $v=\min\cand(\rho)$ implies $v\in\cand(\pi)$.
\end{enumerate}
\end{lem}

\begin{proof}
Claim~\refi{candiff} is trivial.

\smallskip
Next consider claim~\refi{cx}.
From $v\in\cand(\pi)$ it follows that $\pi(v)>\pi(\w v)$.
Moreover $\pi(\s v)=\pi(\e\s v)\geq\pi(\e v)$ since $\pi(\e\s v)\notin\cand(\pi)$.

If $\pi(v)=\pi(\n v)$ then $\n v\in\mu$ and $\n v\in\A_{\mu}\cup\O_{\mu}$ because $c(v)=c(x)$ and $x$ is an outer corner of $\lambda$.
Furthermore $\pi(\n v)>\pi(\w v)$ implies
\begin{eq*}
\rho(\n\w v)
=\max\{\pi(\n\n\w v),\pi(\n\w\w v)\}-\pi(\n\w v)+\pi(\w v)\leq\pi(\w v)
<\pi(\n v)
=\rho(\n v).
\end{eq*}
Thus $\n v\in\cand(\rho)$ unless $\n v\in\A_{\mu}$ and $\rho(\n v)=\rho(\n\n v)$.
In that case, however, $\n\n v\in\A_{\mu}\cup\O_{\mu}$ and $\rho(\n\n\w v)<\rho(\n\n v)$.
By \refl{cand} there exists a candidate $u\in\cand(\rho)$ with $c(u)>c(x)+1$.
By \refi{candiff} this contradicts the minimality of $v$ in $\cand(\pi)$, and we conclude that $\n v\in\cand(\rho)$.
Moreover $\rho(v)=\pi(\e v)=\rho(\e v)$ and $\e v\notin\cand(\rho)$.

On the other hand, if $\pi(v)>\pi(\n v)$ then $v\neq x$ and $\e v\in\A_{\mu}\cup\O_{\mu}$.
We compute $\rho(v)<\rho(\e v)$, which yields $\e v\in\cand(\rho)$ by use of \refl{cand} and claim~\refi{candiff} as above.

\smallskip
The proof of claim~\refi{cx+1} is similar.
Here $\rho(\w v)<\rho(v)$ and $\s v\in\lambda$.
First suppose $\rho(v)\leq\rho(\s\w v)$.
Then $\w v, \s\w v\in\mu$.
We compute
\begin{eq*}
\pi(\w v)
=\max\{\rho(\n\w v),\rho(\w\w v)\}+\rho(v)-\rho(\w v)
>\rho(\w\w v)
=\pi(\w\w v),
\end{eq*}
and $\w v\in\cand(\pi)$ since $\w v\in\O_{\lambda}$.
If $\s v=x$ then $\pi(x)=\rho(\s\w v)$ and $\s v\notin\cand(\pi)$.
If $\s v\neq x$ then $\e\s v\in\mu$ and we claim that $\rho(\s v)=\rho(\e\s v)$.
To see this note that if $v\in\A_{\mu}$ then \refl{cand} and the minimality of $v$ in $\cand(\rho)$ imply $\rho(v)=\rho(\e v)$.
Hence $\rho(\s v)=\rho(\e\s v)$ because $\e\s v\notin\cand(\rho)$.
We conclude that $\rho(\e\s v)\leq\rho(\s\s v)$, $\pi(\s v)=\pi(\s\w v)$ and $\s v\notin\cand(\pi)$ as claimed.

On the other hand suppose $\rho(v)>\rho(\s\w v)$.
Then $\pi(\s v)\geq\rho(v)>\pi(\s\w v)$ and $\s v\in\cand(\pi)$.

\smallskip
Finally turn to the proof of claim~\refi{cx-1}.
If $v\in\cand(\pi)$ then $v\in\A_{\lambda}$ and $\pi(v)>\pi(\w v)$.
Hence $v\in\O_{\mu}$ and $v\in\cand(\rho)$.

Conversely if $v\in\cand(\rho)$ then $v\in\O_{\mu}$, $v\in\A_{\lambda}$ and $\pi(\w v)<\pi(v)$.
It remains to show $\pi(v)>\pi(\n v)$.
We are done if $\n v\notin\mu$.
Thus assume the contrary, which also yields $\e\n v\in\mu$ because $\e v\in\lambda$.
Clearly $\n v\in\I_{\mu}$ and $\e\n v\in\A_{\mu}\cup\O_{\mu}$.
The minimality of $v$ together with \refl{cand} provide that $\rho(\e\n v)=\rho(\n v)\leq\rho(v)$.
Consequently
\begin{eq*}
\pi(\n v)
=\max\{\rho(\n\w v),\rho(\n\n v)\}+\rho(\e\n v)-\rho(\n v)
=\max\{\rho(\n\w v),\rho(\n\n v)\}.
\end{eq*}
On the one hand $\rho(\n\w v)\leq\rho(\w v)<\rho(v)$.
On the other hand $\rho(\n\n v)=\rho(\n\n\w v)$, where we use again \refl{cand} and the minimality of $v$.
Hence $\rho(\n\n v)\leq\rho(\n\w)<\rho(v)$, which completes the proof.
\end{proof}

\begin{lem}{pakpath}
Let $\lambda$ be a partition, $x\in\lambda$ an outer corner, $\mu=\lambda-\{x\}$, $\pi\in\rpp_{\lambda,x}$, $\rho=\zeta_{\lambda,x}(\pi)$, $v=\min\cand(\pi)$, and $u=\min\cand(\rho)$ with respect to the content order.
\begin{enumerate}[(i)]
\myi{cvge}
If $c(v)>c(x)$ then $u=v$ and $Q(v,\pi)=Q(u,\rho)$.
\myi{cveq}
If $c(v)=c(x)$ then $\pi(v)=\pi(\n v)$ implies $u=\n v$ and $\pi(v)>\pi(\n v)$ implies $u=\e v$.
Moreover $Q(v,\pi)=(v,Q(u,\pi))$.
\myi{cvle_avoid}
If $c(v)<c(x)$ and $\omega(Q(v,\pi))\triangleleft x$ then $u=v$ and $Q(v,\pi)=Q(u,\rho)$.
\myi{cvle_touch}
If $c(v)<c(x)$ and $\omega(Q(v,\pi))=x$ then $u=v$ and $Q(v,\pi)=(Q(u,\rho),x)$.
\myi{cvle_cross}
Let $c(v)<c(x)$ and $x\triangleleft\omega(Q(v,\pi))$.
Then $u=v$ and there exists a cell $y\in Q(v,\pi)$ with $c(y)=c(x)$.
If $\pi(y)=\pi(\n u)$ then $Q(u,\rho)$ is obtained from $Q(v,\pi)$ by replacing $y$ with $\n\w y$.
If otherwise $\pi(y)>\pi(\n y)$ then $Q(v,\pi)=Q(u,\rho)$.
\end{enumerate}
\end{lem}

\begin{proof}
Claim~\refi{cvge} follows from \refl{pakcand}~\refi{candiff}.
Claim~\refi{cveq} is a consequence of \refl{pakcand}~\refi{candiff}--\refi{cx+1}.
Claims~\refi{cvle_avoid} and~\refi{cvle_touch} follow from \refl{pakcand}~\refi{candiff}--\refi{cx-1}.

\smallskip
Finally to see claim~\refi{cvle_cross} note that \refl{pakcand}~\refi{candiff}--\refi{cx-1} imply $u=v$.
Moreover it is easily seen that $Q(v,\pi)$ contains a unique cell $y$ with $c(y)=c(x)$, and that the paths $Q(v,\pi)$ and $Q(u,\rho)$ coincide up to (but not necessarily including) that cell.

By definition of $Q(v,\pi)$ the cell preceding $y$ in $Q(v,\pi)$ is $\w y$ since $\w y,\s y\in\I_{\lambda}\cup\A_{\lambda}$ if they lie in $\lambda$ at all.
Furthermore $\pi(\w y)=\pi(y)$ by minimality of $v$.

If $\pi(y)=\pi(\n y)$ then the cell proceeding $y$ in $Q(v,\pi)$ is $\n y$.
In particular $\n y, \n\w y\in\mu$
Note that $\w y\in\O_{\mu}\cup\B_{\mu}$.
Using $\pi(\n\w y)=\pi(\n\w\w y)$ since $\n\w y\notin\cand(\pi)$ we compute $\rho(\n\w y)=\rho(\w y)$. 
Thus the cell proceeding $\w y$ in $Q(u,\rho)$ is $\n\w y$.
The cell $\n\w y$ itself is by definition of $Q(u,\rho)$ proceeded by $\n y$ because $\n\w y\in\I_{\mu}$.
Moreover $\n y\in\B_{\lambda}$ if and only if $\n y\in\O_{\mu}$ and $\n y\in\I_{\lambda}$ if and only if $\n y\in\A_{\mu}$.
Thus the paths $Q(v,\pi)$ and $Q(u,\rho)$ coincide for all remaining cells after $\n y$.

If on the other hand $\pi(y)>\pi(\n y)$ then the cell proceeding $y$ in $Q(v,\pi)$ is $\e y$.
Using $\pi(\n\w\w y)=\pi(\n\w y)$ since $\n\w y\notin\cand(\pi)$ we compute $\rho(\n\w y)=\pi(\n y)<\rho(\w y)$.
Therefore the cell proceeding $\w y$ in $Q(u,\rho)$ is $y$, which is in turn proceeded by $\e y$.
The paths $Q(v,\pi)$ and $Q(u,\rho)$ coincide.
\end{proof}

Suppose two partitions $\lambda$ and $\mu$ both contain the cell $u$.
The rim-hook of $\lambda$ corresponding to $u$ might differ from the rim-hook of $\mu$ corresponding to $u$.
To avoid any ambiguity we denote these rim-hooks by $h_{\lambda}^{u}$ and $h_{\mu}^u$ respectively.
Similarly denote by $0_{\lambda}$ and $0_{\mu}$ the reverse plane partitions of shape $\lambda$ respectively $\mu$ with all entries equal to zero.

\smallskip
The following is the key lemma to understanding why our bijection coincides with the one defined by I.~Pak.
Both maps are defined recursively but the two recursions are very different in nature.
On the one hand $\xi$ is defined by inductively reducing the size of the partition.
On the other hand $\Phi$ is computed by inductively reducing the sum of the entries of a tableaux.
However, we can combine these reductions to great effect.
To be more precise, there exists a kind of commutation relation between them.

\begin{lem}{commute}
Let $\lambda$ be a partition, $x\in\lambda$ an outer corner , $\mu=\lambda-\{x\}$, $u_1,u_2,\dots,u_s\in\mu$ cells satisfying $u_i\leq u_{i+1}$ for all $i\in[s-1]$, and set $\pi=h_{\lambda}^{u_2}*\dots*h_{\lambda}^{u_s}*0_{\lambda}$.
Then
\begin{eq*}
\zeta_{\lambda,x}(h_{\lambda}^{u_1}*\pi)=h_{\mu}^{u_1}*\zeta_{\lambda,x}(\pi).
\end{eq*}
\end{lem}

To understand the statement of \refl{commute} the reader should compare Figures~\ref{Figure:332} and~\ref{Figure:333}.
Each reverse plane partition in the second row of \reff{332} is the image of the corresponding reverse plane partition in the second row of \reff{333} under the map $\zeta_{\lambda,x}$, where $\lambda=(3,3,3)$ and $x=(3,3)$.

\myfig{}{332}{
\begin{tikzpicture}
\begin{scope}
\draw(0,1)--(2,1)(0,2)--(3,2)(1,3)--(1,0)(2,3)--(2,1)
;
\draw[line width=1pt](0,0)--(2,0)--(2,1)--(3,1)--(3,3)--(0,3)--cycle
;
\draw[xshift=5mm,yshift=5mm]
(0,2)node{\small{$1$}}
(1,2)node{\small{$1$}}
(2,2)node{\small{$2$}}
(0,1)node{\small{$0$}}
(1,1)node{\small{$1$}}
(2,1)node{\small{$0$}}
(0,0)node{\small{$3$}}
(1,0)node{\small{$0$}}
;
\end{scope}
\begin{scope}[xshift=4cm]
\draw(0,1)--(2,1)(0,2)--(3,2)(1,3)--(1,0)(2,3)--(2,1)
;
\draw[line width=1pt](0,0)--(2,0)--(2,1)--(3,1)--(3,3)--(0,3)--cycle
;
\draw[xshift=5mm,yshift=5mm]
(0,2)node{\small{$0$}}
(1,2)node{\small{$1$}}
(2,2)node{\small{$2$}}
(0,1)node{\small{$0$}}
(1,1)node{\small{$1$}}
(2,1)node{\small{$0$}}
(0,0)node{\small{$3$}}
(1,0)node{\small{$0$}}
;
\end{scope}
\begin{scope}[xshift=8cm]
\draw(0,1)--(2,1)(0,2)--(3,2)(1,3)--(1,0)(2,3)--(2,1)
;
\draw[line width=1pt](0,0)--(2,0)--(2,1)--(3,1)--(3,3)--(0,3)--cycle
;
\draw[xshift=5mm,yshift=5mm]
(0,2)node{\small{$0$}}
(1,2)node{\small{$1$}}
(2,2)node{\small{$2$}}
(0,1)node{\small{$0$}}
(1,1)node{\small{$1$}}
(2,1)node{\small{$0$}}
(0,0)node{\small{$0$}}
(1,0)node{\small{$0$}}
;
\end{scope}
\begin{scope}[xshift=12cm]
\draw(0,1)--(2,1)(0,2)--(3,2)(1,3)--(1,0)(2,3)--(2,1)
;
\draw[line width=1pt](0,0)--(2,0)--(2,1)--(3,1)--(3,3)--(0,3)--cycle
;
\draw[xshift=5mm,yshift=5mm]
(0,2)node{\small{$0$}}
(1,2)node{\small{$0$}}
(2,2)node{\small{$2$}}
(0,1)node{\small{$0$}}
(1,1)node{\small{$1$}}
(2,1)node{\small{$0$}}
(0,0)node{\small{$0$}}
(1,0)node{\small{$0$}}
;
\end{scope}
\begin{scope}[xshift=16cm]
\draw(0,1)--(2,1)(0,2)--(3,2)(1,3)--(1,0)(2,3)--(2,1)
;
\draw[line width=1pt](0,0)--(2,0)--(2,1)--(3,1)--(3,3)--(0,3)--cycle
;
\draw[xshift=5mm,yshift=5mm]
(0,2)node{\small{$0$}}
(1,2)node{\small{$0$}}
(2,2)node{\small{$2$}}
(0,1)node{\small{$0$}}
(1,1)node{\small{$0$}}
(2,1)node{\small{$0$}}
(0,0)node{\small{$0$}}
(1,0)node{\small{$0$}}
;
\end{scope}
\begin{scope}[xshift=20cm]
\draw(0,1)--(2,1)(0,2)--(3,2)(1,3)--(1,0)(2,3)--(2,1)
;
\draw[line width=1pt](0,0)--(2,0)--(2,1)--(3,1)--(3,3)--(0,3)--cycle
;
\draw[xshift=5mm,yshift=5mm]
(0,2)node{\small{$0$}}
(1,2)node{\small{$0$}}
(2,2)node{\small{$1$}}
(0,1)node{\small{$0$}}
(1,1)node{\small{$0$}}
(2,1)node{\small{$0$}}
(0,0)node{\small{$0$}}
(1,0)node{\small{$0$}}
;
\end{scope}
\begin{scope}[xshift=24cm]
\draw(0,1)--(2,1)(0,2)--(3,2)(1,3)--(1,0)(2,3)--(2,1)
;
\draw[line width=1pt](0,0)--(2,0)--(2,1)--(3,1)--(3,3)--(0,3)--cycle
;
\draw[xshift=5mm,yshift=5mm]
(0,2)node{\small{$0$}}
(1,2)node{\small{$0$}}
(2,2)node{\small{$0$}}
(0,1)node{\small{$0$}}
(1,1)node{\small{$0$}}
(2,1)node{\small{$0$}}
(0,0)node{\small{$0$}}
(1,0)node{\small{$0$}}
;
\end{scope}
\begin{scope}[yshift=-4cm]
\begin{scope}
\draw(0,1)--(2,1)(0,2)--(3,2)(1,3)--(1,0)(2,3)--(2,1)
;
\draw[line width=1pt](0,0)--(2,0)--(2,1)--(3,1)--(3,3)--(0,3)--cycle
;
\draw[xshift=5mm,yshift=5mm]
(0,2)node{\small{$0$}}
(1,2)node{\small{$0$}}
(2,2)node{\small{$0$}}
(0,1)node{\small{$0$}}
(1,1)node{\small{$0$}}
(2,1)node{\small{$0$}}
(0,0)node{\small{$0$}}
(1,0)node{\small{$0$}}
;
\end{scope}
\begin{scope}[xshift=4cm]
\draw(0,1)--(2,1)(0,2)--(3,2)(1,3)--(1,0)(2,3)--(2,1)
;
\draw[line width=1pt](0,0)--(2,0)--(2,1)--(3,1)--(3,3)--(0,3)--cycle
;
\draw[xshift=5mm,yshift=5mm]
(0,2)node{\small{$0$}}
(1,2)node{\small{$0$}}
(2,2)node{\small{$1$}}
(0,1)node{\small{$0$}}
(1,1)node{\small{$1$}}
(2,1)node{\small{$1$}}
(0,0)node{\small{$1$}}
(1,0)node{\small{$1$}}
;
\end{scope}
\begin{scope}[xshift=8cm]
\draw(0,1)--(2,1)(0,2)--(3,2)(1,3)--(1,0)(2,3)--(2,1)
;
\draw[line width=1pt](0,0)--(2,0)--(2,1)--(3,1)--(3,3)--(0,3)--cycle
;
\draw[xshift=5mm,yshift=5mm]
(0,2)node{\small{$0$}}
(1,2)node{\small{$0$}}
(2,2)node{\small{$1$}}
(0,1)node{\small{$0$}}
(1,1)node{\small{$1$}}
(2,1)node{\small{$1$}}
(0,0)node{\small{$4$}}
(1,0)node{\small{$4$}}
;
\end{scope}
\begin{scope}[xshift=12cm]
\draw(0,1)--(2,1)(0,2)--(3,2)(1,3)--(1,0)(2,3)--(2,1)
;
\draw[line width=1pt](0,0)--(2,0)--(2,1)--(3,1)--(3,3)--(0,3)--cycle
;
\draw[xshift=5mm,yshift=5mm]
(0,2)node{\small{$0$}}
(1,2)node{\small{$0$}}
(2,2)node{\small{$2$}}
(0,1)node{\small{$1$}}
(1,1)node{\small{$2$}}
(2,1)node{\small{$2$}}
(0,0)node{\small{$4$}}
(1,0)node{\small{$4$}}
;
\end{scope}
\begin{scope}[xshift=16cm]
\draw(0,1)--(2,1)(0,2)--(3,2)(1,3)--(1,0)(2,3)--(2,1)
;
\draw[line width=1pt](0,0)--(2,0)--(2,1)--(3,1)--(3,3)--(0,3)--cycle
;
\draw[xshift=5mm,yshift=5mm]
(0,2)node{\small{$0$}}
(1,2)node{\small{$0$}}
(2,2)node{\small{$2$}}
(0,1)node{\small{$2$}}
(1,1)node{\small{$3$}}
(2,1)node{\small{$3$}}
(0,0)node{\small{$4$}}
(1,0)node{\small{$4$}}
;
\end{scope}
\begin{scope}[xshift=20cm]
\draw(0,1)--(2,1)(0,2)--(3,2)(1,3)--(1,0)(2,3)--(2,1)
;
\draw[line width=1pt](0,0)--(2,0)--(2,1)--(3,1)--(3,3)--(0,3)--cycle
;
\draw[xshift=5mm,yshift=5mm]
(0,2)node{\small{$0$}}
(1,2)node{\small{$1$}}
(2,2)node{\small{$3$}}
(0,1)node{\small{$2$}}
(1,1)node{\small{$3$}}
(2,1)node{\small{$3$}}
(0,0)node{\small{$4$}}
(1,0)node{\small{$4$}}
;
\end{scope}
\begin{scope}[xshift=24cm]
\draw(0,1)--(2,1)(0,2)--(3,2)(1,3)--(1,0)(2,3)--(2,1)
;
\draw[line width=1pt](0,0)--(2,0)--(2,1)--(3,1)--(3,3)--(0,3)--cycle
;
\draw[xshift=5mm,yshift=5mm]
(0,2)node{\small{$0$}}
(1,2)node{\small{$1$}}
(2,2)node{\small{$4$}}
(0,1)node{\small{$2$}}
(1,1)node{\small{$3$}}
(2,1)node{\small{$4$}}
(0,0)node{\small{$4$}}
(1,0)node{\small{$4$}}
;
\end{scope}
\end{scope}
\end{tikzpicture}
}

\myfig{}{333}{
\begin{tikzpicture}
\begin{scope}
\draw(0,1)--(3,1)(0,2)--(3,2)(1,3)--(1,0)(2,3)--(2,0)
;
\draw[line width=1pt](0,0)--(3,0)--(3,3)--(0,3)--cycle
;
\draw[xshift=5mm,yshift=5mm]
(0,2)node{\small{$1$}}
(1,2)node{\small{$1$}}
(2,2)node{\small{$2$}}
(0,1)node{\small{$0$}}
(1,1)node{\small{$1$}}
(2,1)node{\small{$0$}}
(0,0)node{\small{$3$}}
(1,0)node{\small{$0$}}
(2,0)node{\small{$0$}}
;
\end{scope}
\begin{scope}[xshift=4cm]
\draw(0,1)--(3,1)(0,2)--(3,2)(1,3)--(1,0)(2,3)--(2,0)
;
\draw[line width=1pt](0,0)--(3,0)--(3,3)--(0,3)--cycle
;
\draw[xshift=5mm,yshift=5mm]
(0,2)node{\small{$0$}}
(1,2)node{\small{$1$}}
(2,2)node{\small{$2$}}
(0,1)node{\small{$0$}}
(1,1)node{\small{$1$}}
(2,1)node{\small{$0$}}
(0,0)node{\small{$3$}}
(1,0)node{\small{$0$}}
(2,0)node{\small{$0$}}
;
\end{scope}
\begin{scope}[xshift=8cm]
\draw(0,1)--(3,1)(0,2)--(3,2)(1,3)--(1,0)(2,3)--(2,0)
;
\draw[line width=1pt](0,0)--(3,0)--(3,3)--(0,3)--cycle
;
\draw[xshift=5mm,yshift=5mm]
(0,2)node{\small{$0$}}
(1,2)node{\small{$1$}}
(2,2)node{\small{$2$}}
(0,1)node{\small{$0$}}
(1,1)node{\small{$1$}}
(2,1)node{\small{$0$}}
(0,0)node{\small{$0$}}
(1,0)node{\small{$0$}}
(2,0)node{\small{$0$}}
;
\end{scope}
\begin{scope}[xshift=12cm]
\draw(0,1)--(3,1)(0,2)--(3,2)(1,3)--(1,0)(2,3)--(2,0)
;
\draw[line width=1pt](0,0)--(3,0)--(3,3)--(0,3)--cycle
;
\draw[xshift=5mm,yshift=5mm]
(0,2)node{\small{$0$}}
(1,2)node{\small{$0$}}
(2,2)node{\small{$2$}}
(0,1)node{\small{$0$}}
(1,1)node{\small{$1$}}
(2,1)node{\small{$0$}}
(0,0)node{\small{$0$}}
(1,0)node{\small{$0$}}
(2,0)node{\small{$0$}}
;
\end{scope}
\begin{scope}[xshift=16cm]
\draw(0,1)--(3,1)(0,2)--(3,2)(1,3)--(1,0)(2,3)--(2,0)
;
\draw[line width=1pt](0,0)--(3,0)--(3,3)--(0,3)--cycle
;
\draw[xshift=5mm,yshift=5mm]
(0,2)node{\small{$0$}}
(1,2)node{\small{$0$}}
(2,2)node{\small{$2$}}
(0,1)node{\small{$0$}}
(1,1)node{\small{$0$}}
(2,1)node{\small{$0$}}
(0,0)node{\small{$0$}}
(1,0)node{\small{$0$}}
(2,0)node{\small{$0$}}
;
\end{scope}
\begin{scope}[xshift=20cm]
\draw(0,1)--(3,1)(0,2)--(3,2)(1,3)--(1,0)(2,3)--(2,0)
;
\draw[line width=1pt](0,0)--(3,0)--(3,3)--(0,3)--cycle
;
\draw[xshift=5mm,yshift=5mm]
(0,2)node{\small{$0$}}
(1,2)node{\small{$0$}}
(2,2)node{\small{$1$}}
(0,1)node{\small{$0$}}
(1,1)node{\small{$0$}}
(2,1)node{\small{$0$}}
(0,0)node{\small{$0$}}
(1,0)node{\small{$0$}}
(2,0)node{\small{$0$}}
;
\end{scope}
\begin{scope}[xshift=24cm]
\draw(0,1)--(3,1)(0,2)--(3,2)(1,3)--(1,0)(2,3)--(2,0)
;
\draw[line width=1pt](0,0)--(3,0)--(3,3)--(0,3)--cycle
;
\draw[xshift=5mm,yshift=5mm]
(0,2)node{\small{$0$}}
(1,2)node{\small{$0$}}
(2,2)node{\small{$0$}}
(0,1)node{\small{$0$}}
(1,1)node{\small{$0$}}
(2,1)node{\small{$0$}}
(0,0)node{\small{$0$}}
(1,0)node{\small{$0$}}
(2,0)node{\small{$0$}}
;
\end{scope}
\begin{scope}[yshift=-4cm]
\begin{scope}
\draw(0,1)--(3,1)(0,2)--(3,2)(1,3)--(1,0)(2,3)--(2,0)
;
\draw[line width=1pt](0,0)--(3,0)--(3,3)--(0,3)--cycle
;
\draw[xshift=5mm,yshift=5mm]
(0,2)node{\small{$0$}}
(1,2)node{\small{$0$}}
(2,2)node{\small{$0$}}
(0,1)node{\small{$0$}}
(1,1)node{\small{$0$}}
(2,1)node{\small{$0$}}
(0,0)node{\small{$0$}}
(1,0)node{\small{$0$}}
(2,0)node{\small{$0$}};
\end{scope}
\begin{scope}[xshift=4cm]
\draw(0,1)--(3,1)(0,2)--(3,2)(1,3)--(1,0)(2,3)--(2,0)
;
\draw[line width=1pt](0,0)--(3,0)--(3,3)--(0,3)--cycle
;
\draw[xshift=5mm,yshift=5mm]
(0,2)node{\small{$0$}}
(1,2)node{\small{$0$}}
(2,2)node{\small{$1$}}
(0,1)node{\small{$0$}}
(1,1)node{\small{$0$}}
(2,1)node{\small{$1$}}
(0,0)node{\small{$1$}}
(1,0)node{\small{$1$}}
(2,0)node{\small{$1$}}
;
\end{scope}
\begin{scope}[xshift=8cm]
\draw(0,1)--(3,1)(0,2)--(3,2)(1,3)--(1,0)(2,3)--(2,0)
;
\draw[line width=1pt](0,0)--(3,0)--(3,3)--(0,3)--cycle
;
\draw[xshift=5mm,yshift=5mm]
(0,2)node{\small{$0$}}
(1,2)node{\small{$0$}}
(2,2)node{\small{$1$}}
(0,1)node{\small{$0$}}
(1,1)node{\small{$0$}}
(2,1)node{\small{$1$}}
(0,0)node{\small{$4$}}
(1,0)node{\small{$4$}}
(2,0)node{\small{$4$}}
;
\end{scope}
\begin{scope}[xshift=12cm]
\draw(0,1)--(3,1)(0,2)--(3,2)(1,3)--(1,0)(2,3)--(2,0)
;
\draw[line width=1pt](0,0)--(3,0)--(3,3)--(0,3)--cycle
;
\draw[xshift=5mm,yshift=5mm]
(0,2)node{\small{$0$}}
(1,2)node{\small{$0$}}
(2,2)node{\small{$2$}}
(0,1)node{\small{$1$}}
(1,1)node{\small{$1$}}
(2,1)node{\small{$2$}}
(0,0)node{\small{$4$}}
(1,0)node{\small{$4$}}
(2,0)node{\small{$4$}}
;
\end{scope}
\begin{scope}[xshift=16cm]
\draw(0,1)--(3,1)(0,2)--(3,2)(1,3)--(1,0)(2,3)--(2,0)
;
\draw[line width=1pt](0,0)--(3,0)--(3,3)--(0,3)--cycle
;
\draw[xshift=5mm,yshift=5mm]
(0,2)node{\small{$0$}}
(1,2)node{\small{$0$}}
(2,2)node{\small{$2$}}
(0,1)node{\small{$2$}}
(1,1)node{\small{$2$}}
(2,1)node{\small{$3$}}
(0,0)node{\small{$4$}}
(1,0)node{\small{$4$}}
(2,0)node{\small{$4$}}
;
\end{scope}
\begin{scope}[xshift=20cm]
\draw(0,1)--(3,1)(0,2)--(3,2)(1,3)--(1,0)(2,3)--(2,0)
;
\draw[line width=1pt](0,0)--(3,0)--(3,3)--(0,3)--cycle
;
\draw[xshift=5mm,yshift=5mm]
(0,2)node{\small{$1$}}
(1,2)node{\small{$1$}}
(2,2)node{\small{$3$}}
(0,1)node{\small{$2$}}
(1,1)node{\small{$2$}}
(2,1)node{\small{$3$}}
(0,0)node{\small{$4$}}
(1,0)node{\small{$4$}}
(2,0)node{\small{$4$}}
;
\end{scope}
\begin{scope}[xshift=24cm]
\draw(0,1)--(3,1)(0,2)--(3,2)(1,3)--(1,0)(2,3)--(2,0)
;
\draw[line width=1pt](0,0)--(3,0)--(3,3)--(0,3)--cycle
;
\draw[xshift=5mm,yshift=5mm]
(0,2)node{\small{$1$}}
(1,2)node{\small{$1$}}
(2,2)node{\small{$4$}}
(0,1)node{\small{$2$}}
(1,1)node{\small{$3$}}
(2,1)node{\small{$4$}}
(0,0)node{\small{$4$}}
(1,0)node{\small{$4$}}
(2,0)node{\small{$4$}}
;
\end{scope}
\end{scope}
\end{tikzpicture}
}

\begin{proof}[Proof of \refl{commute}]
Set $\rho=\zeta_{\lambda,x}(h_{\lambda}^{u_1}*\pi)$ and $\sigma=h_{\lambda}^{u_1}*\pi$.
Note that $\sigma\in\rpp_{\lambda,x}$.
It follows from \reft{inj} and \refl{pakpath} that the lexicographic factorisation of $\rho$ begins with $h_{\mu}^{u_1}$.
Let $u=\min\cand(\rho)$ and $v=\min\cand(\sigma)$ with respect to the content order.
It suffices to show that
\begin{eq}{commute}
\rho-Q(u,\rho)
=
\zeta_{\lambda,x}(\sigma-Q(v,\sigma)).
\end{eq}
Set $\tilde{\rho}=\rho-Q(u,\rho)$ and $\tilde{\sigma}=\zeta_{\lambda,x}(\sigma-Q(v,\sigma))$.
By \refl{pakpath} if $y\in\mu$ with $c(y)\neq c(x)$ then $y\in Q(u,\rho)$ if and only if $y\in Q(v,\sigma)$.
Consequently it suffices to check equality in \refq{commute} for cells $y\in\mu$ with $c(y)=c(x)$ that have a neighbour on the path $Q(u,\rho)$.
In cases~\refi{cvge} or~\refi{cvle_avoid} of \refl{pakpath} there remains nothing to do.

\smallskip
Consider case~\refi{cveq} of \refl{pakpath} and suppose that $\sigma(v)=\sigma(\n v)$.
Then
\begin{eq}{ii1}
\tilde{\rho}(\n\w v)
=\max\{\sigma(\n\n\w v),\sigma(\n\w\w v)\}+\sigma(\w v)-\sigma(\n\w v)
=\tilde{\sigma}(\n\w v).
\end{eq}
Moreover if $v\neq x$ then
\begin{eq}{ii2}
\tilde{\rho}(v)
=\sigma(\n v)+\min\{\sigma(\e v),\sigma(\s v)\}-\sigma(v)
=\tilde{\sigma}(v).
\end{eq}
If $\sigma(v)>\sigma(\n v)$ then $v\neq x$ and $\sigma(\s v)<\sigma(\e\s v)$ since $\e\s v\notin\cand(\sigma)$.
It follows that $\tilde{\rho}(v)=\tilde{\sigma}(v)$ as in \refq{ii1}, and if $\e\s v\neq x$ then $\tilde{\rho}(\e\s v)=\tilde{\sigma}(\e\s v)$ as in \refq{ii2}.

\smallskip
In case~\refi{cvle_touch} of \refl{pakpath} we are done if $\n\w x\notin\mu$.
Therefore assume $\n\w x\in\mu$.
Since $Q(v,\sigma)$ terminates in $x$, we conclude $\sigma(x)>\sigma(\n x)$.
Moreover $\sigma(\n\w x)=\sigma(\n\w\w x)$ because $\n\w x\notin\cand(\sigma)$.
Hence $\tilde{\rho}(\n\w x)=\sigma(\n x)=\tilde{\sigma}(\n\w x)$.

\smallskip
Finally suppose we are in case~\refi{cvle_cross} and let $y\in Q(v,\sigma)$ be the cell with $c(y)=c(x)$.
Then $\w y$ precedes $y$ in $Q(v,\sigma)$ and $\sigma(\w y)=\sigma(y)$ since $y\notin\cand(\sigma)$.
Similarly $\sigma(\n\w y)=\sigma(\n\w\w y)$.
If $\sigma(y)=\sigma(\n y)$ then the cell $\n y$ proceeds $y$ in $Q(v,\sigma)$, and $\tilde{\rho}(\n\w y)=\sigma(y)-1=\tilde{\sigma}(\n\w y)$ because $\n\w y\in Q(u,\rho)$.
Moreover if $y\neq x$ then
\begin{eq*}
\tilde{\rho}(y)
=\min\{\sigma(\e y),\sigma(\s y)\}
=\tilde{\sigma}(y).
\end{eq*}
There remains the case $\sigma(y)>\sigma(\n y)$ in which $\e y$ proceeds $y$ in $Q(v,\sigma)$ and $\sigma(\e\s y)=\sigma(\s y)$ as $\e\s y\notin\cand(\sigma)$.
If $y\in\mu$ then $\tilde{\rho}(y)=\sigma(\e y)-1=\tilde{\sigma}(y)$ and if $\e\s\in\mu$ then
\begin{eq*}
\tilde{\rho}(\e\s y)
=\min\{\sigma(\e\e\s y),\sigma(\e\s\s y)\}
=\tilde{\sigma}(\e\s y).
\qedhere
\end{eq*}
\end{proof}


Having done most of the work, we conclude the main result of this section.

\begin{thm}{pak}
Let $\lambda$ be a partition.
Then $\Phi:\tab_{\lambda}\to\rpp_{\lambda}$ and $\xi_{\lambda}:\rpp_{\lambda}\to\tab_{\lambda}$ are inverse bijections.
\end{thm}

\begin{proof}
Let $x\in\lambda$ be the minimum with respect to the reverse lexicographic order and set $\mu=\lambda-\{x\}$.
Let $\pi\in\rpp_{\lambda}$.
We prove the claim by induction on $\abs{\lambda}+\abs{\pi}$.

If $\pi(x)>\max\{\pi(\n x),\pi(\w x)\}$ then define a reverse plane partition by setting $\tilde{\pi}(x)=\pi(x)-1$ and $\tilde{\pi}(u)=\pi(u)$ for all $u\in\mu$.
Using the definitions of $\xi$ and $\Phi$ and the induction hypothesis we obtain
\begin{eq*}
\xi_{\lambda}(\pi)(x)
=\xi_{\lambda}(\tilde{\pi})(x)+1
=\Phi^{-1}(\tilde{\pi})(x)+1
=\Phi^{-1}(\pi)(x),
\end{eq*}
and
\begin{eq*}
\xi_{\lambda}(\pi)(u)
=\xi_{\lambda}(\tilde{\pi})(u)
=\Phi^{-1}(\tilde{\pi})(u)
=\Phi^{-1}(\pi)(u)
\end{eq*}
for all $u\in\mu$.

If $\pi(x)=\max\{\pi(\n x),\pi(\w x)\}$ then the lexicographic factorisation $\pi=h_{\lambda}^{u_1}*\dots*h_{\lambda}^{u_s}*0_{\lambda}$ satisfies $u_i\in\mu$ for all $i\in[s]$.
We compute $\xi_{\lambda}(\pi)(x)=0=\Phi^{-1}(\pi)(x)$ and
\begin{eq*}
\begin{split}
\xi_{\lambda}(\pi)(u)
&=\xi_{\mu}\circ\zeta_{\lambda,x}(\pi)(u)\\
&=\Phi^{-1}\circ\zeta_{\lambda,x}(h_{\lambda}^{u_1}*\dots*h_{\lambda}^{u_s}*0_{\lambda})(u)\\
&=\Phi^{-1}(h_{\mu}^{u_1}*\dots*h_{\mu}^{u_s}*0_{\mu})(u)\\
&=\Phi^{-1}(h_{\lambda}^{u_1}*\dots*h_{\lambda}^{u_s}*0_{\lambda})(u)
=\Phi^{-1}(\pi)(u)
\end{split}
\end{eq*}
for all $u\in\mu$, where we use first the induction hypothesis and then \refl{commute}.
\end{proof}

In the remainder of the article we discuss what can be said about the relation between the bijection $\Phi$, the RSK correspondence and the Hillman--Grassl correspondence.

\smallskip
A \emph{semi-standard Young tableau} of shape $\lambda$ is a reverse plane partition $\pi:\lambda\to\N$ that is column-strict, that is, $\pi(u)>\pi(\n u)$ for all cells $u\in\lambda$.
A \emph{standard Young tableau} of shape $\lambda$ is a reverse plane partition that is also a bijection $\pi:\lambda\to[n]$, where $n=\abs{\lambda}$.

The Hillman--Grassl correspondence $\HG:\rpp_{\lambda}\to\tab_{\lambda}$ is viewed as a bijection from reverse plane partitions to tableaux.
The Robinson--Schensted--Knuth correspondence $\RSK$ is viewed as a map that sends each $t\in\tab_{\lambda}$, where the partition $\lambda=(n^n)$ is a square, to a pair $(P,Q)$ of semi-standard Young tableaux of the same shape $\mu$, where $\mu$ is a partition of size $\abs{t}$.
Moreover, $\RSK$ restricts to a bijection from permutation matrices, which form a subset of $\tab_{\lambda}$, to pairs $(P,Q)$ of standard Young tableaux of the same shape $\mu$, where $\mu$ ranges over all partitions of $n$.
For a thorough introduction to these maps the reader is referred to~\cite[Ch.~7]{StanleyEC2}.

\smallskip
Let $\lambda$ be a partition.
It is very easy to read off the $k$-th trace of a reverse plane partition $\pi\in\rpp_{\lambda}$ directly from the corresponding tableau $t=\Phi^{-1}(\pi)$.
For $k\in\Z$ define the rectangle
\begin{eq*}
R_k=\big\{(i,j)\in\lambda:i\leq i_k\text{ and }j\leq j_k\big\},
\end{eq*}
where $(i_k,j_k)\in\lambda$ is the south-easternmost cell of $\lambda$ with content $k$.
If $\lambda$ contains no cell of content $k$ then let $R_k=\emptyset$.

\begin{prop}{diag}
Let $\lambda$ be a partition, $k\in\Z$, $t\in\tab_{\lambda}$, and set $\pi=\Phi(t)$.
Then
\begin{eq*}
\tr_k(\pi)
=\sum_{u\in R_k}t(u).
\end{eq*}
\end{prop}

\begin{proof}
It is easy to see that the insertion of a hook $h^u$ increases $\tr_k$ if and only if $h^u$ contains a cell with content $k$.
But this is the case if and only if $u\in R_k$.
\end{proof}

A classical result states that for the Hillman--Grassl correspondence one can not only read off the $k$-th trace, but also determine the partition formed by the entries of $\pi$ in the $k$-th diagonal of $\lambda$, solely by looking at the tableau $\HG(\pi)$.
This is achieved via the RSK correspondence.
See for example~\cite[Thm.~3.3]{Gansner1981}.

An analogous statement for the map $\Phi$ that refines \refp{diag} is due to A.~Garver and R.~Patrias.

A \emph{weak south-east-chain} in $\lambda$ is a sequence of cells $(i_1,j_1),\dots,(i_s,j_s)\in\lambda$ such that $i_k\leq i_{k+1}$ and $j_k\leq j_{k+1}$ for all $k\in[s-1]$.
A \emph{strict north-east-chain} in $\lambda$ is a sequence of cells $(i_1,j_1),\dots,(i_s,j_s)\in\lambda$ such that $i_k>i_{k+1}$ and $j_k<j_{k+1}$ for all $k\in[s-1]$.
The length $s$ of such a chain $C$ is denoted by $\abs{C}$.

\begin{thm}{GK}
\textnormal{\cite[Thm.~6.1]{GarverPatrias}}
Let $\lambda$ be a partition, $k\in\Z$, $t\in\tab_{\lambda}$, $\pi=\Phi(t)$, and let $\mu$ denote the partition given by the entries in the $k$-th diagonal of $\pi$.
Then for all $r\in\N$
\begin{eq*}
\mu_1+\dots+\mu_r=\max\big\{\abs{C_1}+\dots+\abs{C_r}\big\},
\end{eq*}
where the maximum is taken over all families of weak south-east-chains $C_1,\dots,C_r$ in $R_k$ that contain each cell $u\in R_k$ at most $t(u)$ times.
Moreover for all $r\in\N$
\begin{eq*}
\mu_1'+\dots+\mu_r'=\max\big\{\abs{D_1}+\dots+\abs{D_r}\big\},
\end{eq*}
where the maximum is taken over all families of strict north-east-chains $D_1,\dots,D_r$ in $R_k$ that contain each cell $u\in R_k$ at most $t(u)$ times.
\end{thm}

For example consider the partition $\mu=(4,3,1)$ in the diagonal of content $k=0$ of the reverse plane partition $\pi$ in \reff{rsk}.
In this case $R_0=\lambda=(3,3,3)$.
We may chose $C_1=(1,1)(3,1)(3,1)(3,1)$, $C_2=(1,2)(1,3)(1,3)$, $C_3=(2,2)$, $D_1=(3,1)(2,2)(1,3)$, $D_2=(3,1)(1,2)$, $D_3=(3,1)(1,3)$, and $D_4=(1,1)$.
Then $\mu_1+\dots+\mu_r=\abs{C_1}+\dots+\abs{C_r}$ and $\mu_1'+\dots+\mu_r'=\abs{D_1}+\dots+\abs{D_r}$ for all $r$.

\smallskip
\reft{GK} relates the map $\Phi$ to $\RSK$.
See \reff{rsk}.
I.~Pak was aware of this connection between the map $\xi$ and RSK although he provides no proof in~\cite{Pak:hook_length_formula_geometric}.
The analogue of \reft{GK} for the map $\xi$ can be obtained by complementing the ideas of \cite[Sec.~4]{Hopkins:rsk} with the results of either V.~Danilov and G.~Koshevoy~\cite{DK:octahedron_recurrence_RSK} or M.~Farber, S.~Hopkins and W.~Trongsiriwat~\cite{FHT}.

\reft{pak} can be deduced from \reft{GK} and the analogue statement for the map $\xi$.
The proof given in this paper is more direct and does not use Greene--Kleitman invariants.

\myfig{
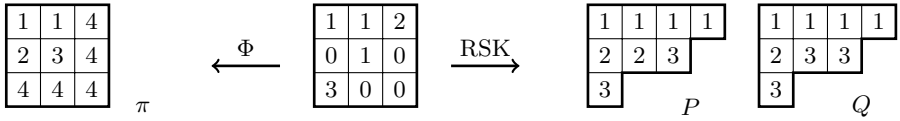
The partitions given by the diagonals of $\pi$ correspond to the shapes of the semi-standard Young tableaux $P$ and $Q$ restricted to the numbers $1,\dots,k$ for $k=1,2,3$.
}{rsk}{
\begin{tikzpicture}
\begin{scope}
\draw(0,1)--(3,1)(0,2)--(3,2)(1,3)--(1,0)(2,3)--(2,0)
;
\draw[line width=1pt](0,0)--(3,0)--(3,3)--(0,3)--cycle
;
\draw[xshift=5mm,yshift=5mm]
(0,2)node{\small{$1$}}
(1,2)node{\small{$1$}}
(2,2)node{\small{$4$}}
(0,1)node{\small{$2$}}
(1,1)node{\small{$3$}}
(2,1)node{\small{$4$}}
(0,0)node{\small{$4$}}
(1,0)node{\small{$4$}}
(2,0)node{\small{$4$}}
;
\draw(4,0)node{\small{$\pi$}}
;
\end{scope}
\begin{scope}[xshift=9cm]
\draw(0,1)--(3,1)(0,2)--(3,2)(1,3)--(1,0)(2,3)--(2,0)
;
\draw[line width=1pt](0,0)--(3,0)--(3,3)--(0,3)--cycle
;
\draw[xshift=5mm,yshift=5mm]
(0,2)node{\small{$1$}}
(1,2)node{\small{$1$}}
(2,2)node{\small{$2$}}
(0,1)node{\small{$0$}}
(1,1)node{\small{$1$}}
(2,1)node{\small{$0$}}
(0,0)node{\small{$3$}}
(1,0)node{\small{$0$}}
(2,0)node{\small{$0$}}
;
\draw[->,line width=1pt,xshift=4cm,yshift=12mm](0,0)--node[anchor=south]{\small{$\RSK$}}(2,0);
\draw[->,line width=1pt,xshift=-3cm,yshift=12mm](2,0)--node[anchor=south]{\small{$\Phi$}}(0,0);
\end{scope}
\begin{scope}[xshift=17cm]
\draw(0,1)--(1,1)(0,2)--(3,2)(1,3)--(1,1)(2,3)--(2,1)(3,3)--(3,2)
;
\draw[line width=1pt](0,0)--(1,0)--(1,1)--(3,1)--(3,2)--(4,2)--(4,3)--(0,3)--cycle
;
\draw[xshift=5mm,yshift=5mm]
(0,2)node{\small{$1$}}
(1,2)node{\small{$1$}}
(2,2)node{\small{$1$}}
(3,2)node{\small{$1$}}
(0,1)node{\small{$2$}}
(1,1)node{\small{$2$}}
(2,1)node{\small{$3$}}
(0,0)node{\small{$3$}}
;
\draw(3,0)node{\small{$P$}};
\end{scope}
\begin{scope}[xshift=22cm]
\draw(0,1)--(1,1)(0,2)--(3,2)(1,3)--(1,1)(2,3)--(2,1)(3,3)--(3,2)
;
\draw[line width=1pt](0,0)--(1,0)--(1,1)--(3,1)--(3,2)--(4,2)--(4,3)--(0,3)--cycle
;
\draw[xshift=5mm,yshift=5mm]
(0,2)node{\small{$1$}}
(1,2)node{\small{$1$}}
(2,2)node{\small{$1$}}
(3,2)node{\small{$1$}}
(0,1)node{\small{$2$}}
(1,1)node{\small{$3$}}
(2,1)node{\small{$3$}}
(0,0)node{\small{$3$}}
;
\draw(3,0)node{\small{$Q$}};
\end{scope}
\end{tikzpicture}
}

\smallskip
Finally, combining \reft{GK} and its counterpart for the Hillman--Grassl correspondence, one finds the following relation between the bijections $\HG$ and $\Phi$ in the case of standard Young tableaux respectively permutation matrices.

\begin{thm}{syt}
Let $\lambda=(n^n)$ be a square partition, and $\pi\in\rpp_{\lambda}$ be a reverse plane partition such that $\tr_{k}(\pi)=\tr_{-k}(\pi)=n-k$ for all $k\in\{0,\dots,n-1\}$.
Then $\Phi\circ\HG(\pi)$ is obtained by inserting in each diagonal the conjugate of the partition in the corresponding diagonal of $\pi$.
\end{thm}

\begin{thm}{rsk}
Let $\lambda=(n^n)$ be a square partition, and $\sigma:\lambda\to\N$ be a tableau corresponding to a permutation matrix with $\RSK(\sigma)=(P,Q)$.
Then $\RSK\circ\HG\circ\Phi(\sigma)=(P',Q')$, where $P'$ and $Q'$ denote the conjugate standard Young tableaux of $P$ and $Q$.
\end{thm}

\begin{cor}{involution}
The map $\HG\circ\Phi$ defines an involution on permutations.
\end{cor}

Note that the map $\HG\circ\Phi$ is not an involution in general.
On the flip side this means that one can define a non-trivial variation of RSK as in \reff{rsk} by using the Hillman--Grassl correspondence instead of the map $\Phi$.
To the best of the authors knowledge this idea has not been pursued so far.
Another open question brought up by S.~Hopkins is whether there exists an alternative description of the Hillman--Grassl correspondence in the style of the map $\xi$.


\providecommand{\bysame}{\leavevmode\hbox to3em{\hrulefill}\thinspace}
\providecommand{\MR}{\relax\ifhmode\unskip\space\fi MR }
\providecommand{\MRhref}[2]{%
  \href{http://www.ams.org/mathscinet-getitem?mr=#1}{#2}
}
\providecommand{\href}[2]{#2}

\end{document}